\documentclass[12pt, a4paper, abstract]{scrartcl}
\usepackage[english]{babel}
\usepackage{amsmath,amsthm,amssymb,amsfonts,mathtools}
\usepackage{graphicx}
\usepackage{caption}
\usepackage{subcaption}
\usepackage[section]{placeins}
\usepackage{delarray} 
\usepackage{booktabs}

\newcolumntype{Y}{>{\centering\arraybackslash}X}
\usepackage{listings} 
\usepackage{comment}

\usepackage{xcolor}
\usepackage{tikz}
\usetikzlibrary{calc,math}

\definecolor{lightgray}{cmyk}{0, 0, 0, 0.05}
\definecolor{lightorange}{cmyk}{0, 0.12, 0.18, 0}
\definecolor{lightblue}{cmyk}{0.15, 0.1, 0, 0}
\definecolor{fullblue}{cmyk}{1, 0.67, 0, 0}

\usepackage{todonotes}

\usepackage{float}

\usepackage{algpseudocode}
\usepackage{algorithm}
\algrenewcommand\algorithmicrequire{\textbf{Input}}
\algrenewcommand\algorithmicensure{\textbf{Output}}

\makeatletter
\newenvironment{breakablealgorithm}
  {
   \begin{center}
     \refstepcounter{algorithm}
     \hrule height.8pt depth0pt \kern2pt
     \renewcommand{\caption}[2][\relax]{
       {\raggedright\textbf{\fname@algorithm~\thealgorithm} ##2\par}%
       \ifx\relax##1\relax 
         \addcontentsline{loa}{algorithm}{\protect\numberline{\thealgorithm}##2}%
       \else 
         \addcontentsline{loa}{algorithm}{\protect\numberline{\thealgorithm}##1}%
       \fi
       \kern2pt\hrule\kern2pt
     }
  }{
     \kern2pt\hrule\relax
   \end{center}
  }
\makeatother

\usepackage{xy}
\xyoption{all}

\usepackage{enumerate} 
\usepackage[autostyle=true]{csquotes} 

\usepackage{ stmaryrd } 

\usepackage[hyphens]{url} 
\usepackage[backend=bibtex, style=alphabetic]{biblatex}
\usepackage{setspace}
\usepackage{footmisc}
\usepackage{biblatex}
\usepackage{hyperref}
\usepackage{xurl}
\hypersetup{breaklinks=true}

\addtokomafont{disposition}{\boldmath}

\graphicspath{ {images/} }

\makeatletter
\renewcommand*{\ext@figure}{lot}
\let\c@figure\c@table
\let\ftype@figure\ftype@table
\renewcommand*\listtablename{List of Tables and Figures}
\let\listoftablesandfigures\listoftables
\makeatother

\newcommand\widebar[1]{\mathop{\overline{#1}}}

\theoremstyle{definition}
\newtheorem{defn}{Definition}[section]

\newtheorem{rmk}[defn]{Remark}
\newtheorem{notat}[defn]{Notation}
\theoremstyle{plain}
\newtheorem{lem}[defn]{Lemma}
\newtheorem{lemdefn}[defn]{Lemma-Definition}
\newtheorem{prop}[defn]{Proposition}
\newtheorem{thm}[defn]{Theorem}
\newtheorem{cor}[defn]{Corollary}

\newtheorem*{thm*}{Theorem} 

\newcommand{\Fq}{\mathbb{F}_q} 
\newcommand{\Z}{\mathbb{Z}} 
\newcommand{\N}{\mathbb{N}} 
\renewcommand{\P}{\mathbb{P}} 
\newcommand{\F}{F_{\infty}} %
\newcommand{\Oi}{\mathcal{O}_{\infty}}
\newcommand{\C}{\mathbb{C}_{\infty}}

\newcommand{\uB}{\underline{B}}
\DeclareMathSymbol{\shortminus}{\mathbin}{AMSa}{"39}

\newcommand{\la}{\lambda}
\newcommand{\G}{\Gamma}
\newcommand{\g}{\gamma}

\renewcommand{\epsilon}{\varepsilon}

\newcommand{\B}{\mathcal{B}} 
\newcommand{\V}{\mathcal{V}} 
\newcommand{\W}{\mathcal{W}} 
\newcommand{\A}{\mathcal{A}} 
\renewcommand{\S}{\mathcal{S}} 

\renewcommand{\t}{\tau}
\newcommand{\pr}{\mathrm{pr}}

\renewcommand{\|}{\, | \,} 
\newcommand{\lb}{\left\lbrace}
\newcommand{\rb}{\right\rbrace}
\newcommand{\lbk}{\left\lbrack}
\newcommand{\rbk}{\right\rbrack}
\newcommand{\smatrix}[9]{%
\left(\begin{smallmatrix}
 \vphantom{f}#1 & #2 & #3 \\
 #4 & \phantom{,}\vphantom{f}#5\phantom{,} & #6 \\
 #7 & #8 & \vphantom{f}#9 \\
 \end{smallmatrix}\right)%
}

\DeclareMathOperator{\Sym}{Sym}
\DeclareMathOperator{\Maps}{Maps}
\DeclareMathOperator{\GL}{GL}
\DeclareMathOperator{\PGL}{PGL}
\DeclareMathOperator{\SL}{SL}
\DeclareMathOperator{\St}{St}
\DeclareMathOperator{\Stab}{Stab}
\DeclareMathOperator{\Char}{C_{har}}
\DeclareMathOperator{\sgn}{sgn}

\begin{document}
\allowdisplaybreaks 
\title{$U$-Operators Acting on Harmonic Cocycles for $\GL_3$ and Their Slopes}

\author{Gebhard Böckle, Peter Mathias Gräf, Theresa Kaiser}
\date{\today}

\maketitle

\begin{abstract}
In this article, we describe a computational study of the action of the two natural $U$-operators acting on $\G$-invariant spaces of harmonic cocycles for $\GL_3$ for certain congruence subgroups $\G$, in a positive characteristic setting. The cocycle spaces we consider are conjecturally isomorphic to spaces of Drinfeld cusp forms of rank $3$ and level $\G$ via an analogue of Teitelbaum's residue map. We give explicit descriptions of the spaces of harmonic cocycles as subspaces of the vector space of coefficients, and of the resulting $U$- and Hecke operators acting on these. We then implement these formulas in a computer algebra system. Using the resulting data of slopes (and characteristic polynomials) for the Hecke actions, we observe several patterns and interesting phenomena present in our slope tables. This appears to be the first such study in a $\GL_3$ setting. 
\end{abstract}

\setcounter{tocdepth}{1}
\tableofcontents

\section{Introduction}
\label{Introduction}
Let $F=\Fq(t)$ with $q=p^e$ for a prime $p$. We denote by $A=\Fq[t]\subset F$ the subring of functions regular away from $\infty$ and by $\F=\Fq\left(\left(\tfrac{1}{t}\right)\right)$ the completion of $F$ at $\infty$. Let $\C$ be the completion of an algebraic closure of $\F$.
For general rank $r$ and any congruence subgroup $\G\subset\GL_r(A)$, Basson-Breuer-Pink have constructed in \cite{bas.bre.pin2024} spaces of Drinfeld cusp forms $S_{k,n}(\G)$ of any weight $k\geq r$ and type $n\in\Z$. In analogy with the work of Teitelbaum in \cite{tei1991} in rank $2$, one expects that these spaces of Drinfeld cusp forms are isomorphic as Hecke-modules to certain spaces of harmonic cocycles on the Bruhat-Tits building $\B^r$ for $\mathrm{PGL}_r(\F)$. For $r=2$ this correspondence has proven to be a crucial tool in understanding the behavior of Hecke-operators on Drinfeld cusp forms as well as making them computationally accessible. 

In rank $r=3$, the work \cite{gr2021} of the second author provides an analogue of Teitelbaum's isomorphism, namely a residue map
\begin{align}
\label{dmf-har}
S_{k+3,n}(\Gamma)\rightarrow \Char(\Gamma,V_{k,n})\otimes_F\C,
\end{align}
which is conjectured to be an isomorphism of Hecke-modules
; here $\Char(\Gamma,V_{k,n})$ denotes the space of harmonic cocycles on $\B := \B^3$, that are invariant under $\Gamma$ and with values in a certain algebraic $\GL_3(F)$-representation $V_{k,n}$. So in order to explore $U$-operators on Drinfeld cusp forms it seemed natural to explore instead these operators on the side of harmonic cocycles, where the underlying combinatorial structure is more amenable to computations.

For $r=2$, the work of Bandini-Valentino in \cite{ban.val2018,ban.val2019} provides a simple model to compute the action of the $U_t$-operator in levels $\Gamma_1(t)$ and $\Gamma_0(t)$. They, as well as Hattori in \cite{hat2021}, used this to gather much data for exploring the eigenvalues and slopes of $U_t$ and for studying its diagonalizability. Because of \cite{tei1991}, their computations are known to give the $U_t$-operator on $S_k(\G_1(t))$.

The starting point for the present work was the realization, that perhaps the computation in \cite{ban.val2018} for $r=2$ and $\G_1(t)$ would have a direct generalization to $r=3$ -- or in fact to any $r\ge2$. Namely, from a more abstract view point, the work \cite{ban.val2018} relied on the following simple steps: They observe that a fundamental domain for the action of $\G_1(t)$ on the Bruhat-Tits tree $\B^2$ was the `standard apartment', that the latter contained a unique stable edge $e_0$, that Teitelbaum's work gives (also) an isomorphism $\Char(\G_1(t),V)\to V$ for any $\GL_2(F)$-representation, {\em and} that the $U_t$ operator could be transferred by explicit formulas to an endomorphism of~$V$. There were some clues that this should work for any $r$, and in the present work, we give the full details and some related computations for the case $r=3$. Let us also note that in rank $r$ there will be $r-1$ $U$-operators, and hence we have to deal with two of them, which we call $U_1$ and $U_2$. For $r=3$, our results directly concern $\Char(\G,V_{k,n})$ for certain coefficients $V_{k,n}$ and congruence subgroups $\G_1(t)\subset\G\subset\GL_3(A)$, but only conjecturally the space $S_{k+3,n}(\G)$ from \cite{bas.bre.pin2024}.

Let us summarize in the following theorem some of the main results in relation to the above expectations. This combines Corollary~\ref{cor:A-FD},  Corollary~\ref{cor:stab-a}, Theorem~\ref{thm:isoG1} and Theorem \ref{thm:IsomG0}. We denote by $\A$ the standard apartment of $\B$ from Definition \ref{def:StdApt}, by $\G_1(t)$, $\G_0(t)$ the groups from Definition~\ref{def:Gamma1}, and by $V$ any $\GL_3(F)$-representation over an $F$-vector space.
\begin{thm*}
    The following hold:
    \begin{enumerate}
        \item $\A$ is a fundamental domain for the action of $\G_1(t)$ on $\B$.
        \item The apartment $\A$ contains a unique stable simplex, the chamber $s_0$ from \eqref{eq:s0-def}.
        \item\label{part3}
        The map $\psi: \Char(\Gamma_1(t),V)  \rightarrow V,   c \mapsto c(s_0)
        $ is an isomorphism of $F$-vector spaces.
        \item \label{part4} With $D\subset \GL_3(\Fq)$ the subgroup of diagonal matrices, the map in \ref{part3}. restricts to an isomorphism
        $\Char(\G_0(t),V)\simeq V^D$.
    \end{enumerate}
\end{thm*}
Theorem~\ref{thm:isoG1} also gives an explicit formula for the converse of~$\psi$, and we obtain results similar to \ref{part4} also for other $\G$; see Subsection~\ref{subsec:CharOnSubgps}.

We then go on to make explicit in Theorem~\ref{thm:AiBiCi} the action of the operators $U_1$ and $U_2$ from \eqref{eq:U_i-def} on any representation $V$. The theorem may appear a bit technical, and it is too long to be stated here. But as an immediate consequence of its formulas, one finds that our Hecke operators can be written as concatenations of the action of simple diagonal matrices and operators that can be described by matrices with entries in $\Fq$.

In Section \ref{sec:magma} we finally turn to the coefficients $V_{k,n}$ that are conjecturally related to the spaces $S_{k,n}(\G)$ of \cite{bas.bre.pin2024}. We give some indications on how our computations are implemented in the computer algebra system \texttt{Magma}. For the coefficients $V_{k,n}$, we observe that, with respect to natural basis, the $U_i$-operators are in fact given as a matrix of the form $\alpha\delta$ with $\alpha$ having only coefficients in $\mathbb{F}_p$ (!), and $\delta$ a diagonal matrix with only powers of $t$ as entries. For rank $2$ such a structural result was observed in \cite{hat2021}. See also the remarks in Subsection~\ref{subsec:entries}.

As the reader certainly has noticed by now, the main focus of this work is in computing $U_i$-operators on certain spaces of harmonic cocycles. Some results and many observations from our computations will be presented in the final Section~\ref{sec:Ti+Ui-ops}, where we give some tables of slopes for the action of $U_1$ and $U_2$ on $\Char(\G_0(t),V_{k,n})$ and certain subspaces. 
To our knowledge, this is the first exploration of slopes in the $\GL_3$-case, over number or function fields, in the spirit of \cite{GouveaMazurSlopes}. Let us mention a few interesting phenomena that we are able to observe in this situation: 

The slopes of the $U_1$-operator exhibit a periodicity akin to a conjecture of Hattori in \cite{hat2021} in rank 2. For the $U_2$-operator, we observe a new phenomenon --  the slopes display a periodicity, but moreover the multiplicities of the finite slopes grow in each repetition by a fixed increment. As a consequence of this, with growing $k$ many slopes of $U_2$ will have much higher multiplicity than the slopes of $U_1$. For $q=2$ only, we also have good evidence that rank $2$ cocycles map to certain rank $3$ cocyles, because we find the $U_t$-eigenvalues for rank $2$ in weight $k$ as $U_1$-eigenvalues for rank $3$ again in weight $k$. 
On a different note, the slopes of what might be a sensible notion of \emph{$\Gamma_0(t)$-newforms} seem to be only given by $2k/3$ and $k/3$ for $U_1$ and $U_2$ respectively. Some of our observations will be explained in more detail in forthcoming work, where we investigate the notions of new- and oldforms in this situation from a theoretical perspective. For quite a number of our observations we have no good heuristics or explanations at the moment.

Let us end with some comments on the structure of this paper. In Section~\ref{sec:building} we first give some background on the building $\B = \B^3$ and present a detailed study of the action of certain congruence subgroups on $\B$. We describe the stabilizers of the simplices of the standard apartment $\A$ and identify it as a fundamental domain for $\G_1(t)$. We also show in Theorem~\ref{thm:index_q} that moving away from the $\G_1(t)$-stable simplex $s_0$ along galleries inside $\A$ increases the $\G_1(t)$-stabilizers in a precise way. In Section~\ref{sec:cocycles} we obtain identifications of $\Char(\G,V)$ with certain subspace of $V$ for several natural groups $\G$ intermediate to $\G_1(t)\subset\GL_3(A)$; see Theorems \ref{thm:isoG1}, \ref{thm:IsomG0}, \ref{thm:IdentifyGL3}, \ref{thm:Isog-P0} and \ref{thm:Isog-P2}. In Section \ref{sec:heckeop} 
we give formulas for the Hecke operators acting on the respective subspaces  of $V$. Section \ref{sec:magma} explains how Theorem \ref{thm:AiBiCi} is used to compute these Hecke operators on simple bases under the above isomorphisms in computer algebra systems such as \texttt{Magma}. Finally Section \ref{sec:Ti+Ui-ops} 
contains computational data on the slopes of Hecke operators computed in \texttt{Magma} and various observations on said data. The Appendix \ref{appendix} gives some further technical results used in Section~\ref{sec:building} and~\ref{sec:heckeop}.

\noindent\textbf{Acknowledgements:} This work received funding by the Deutsche Forschungsgemeinschaft (DFG, German Research Foundation) TRR 326 \textit{Geometry and Arithmetic of Uniformized Structures}, project number 444845124. Moreover, P.G. received funding by the Deutsche Forschungsgemeinschaft (DFG, German Research Foundation) -- project number 546550122. The present work incorporates results from the bachelor and master thesis 
of the third author.

\section{Bruhat-Tits Building}
\label{sec:building}
Let $p$ be a prime and $q$ a power of $p$. We fix the following notations that we will need throughout this article: 
\begin{align*}
    A & := \Fq[t] = \text{the polynomial ring over $\Fq$ in the variable $t$,} \\
	F & := \Fq(t) = \text{the fraction field of $A$,}\\
    |\cdot| & := 
    \text{
    the $\infty$-adic norm on $F$, normalized so that $|t| = q$,} \\
	\F & := \Fq\left(\left(\tfrac{1}{t}\right)\right) = \text{the completion of $F$ with respect to $|\cdot|$. We often write } \pi := \tfrac{1}{t}, \\ 
	\Oi & := \Fq\left\lbrack\left\lbrack\tfrac{1}{t}\right\rbrack\right\rbrack = \text{the corresponding valuation ring,} \\
	\C & := \text{the completion of a fixed algebraic closure of $\F$. It is algebraically closed.} 
\end{align*}

\subsection{Basic definitions}
This subsection introduces the Bruhat-Tits building $\B$ of $\PGL_3(\F)$ and describes the matrix action on it. For the first three definitions and for Subsection \ref{subsec:action-on-simplices}, we follow \cite{m2014} who often relies on \cite{geb1996} . 
We assume that the reader is familiar with the notion of simplicial complexes, if not, see \cite[][Chapter I, Appendix]{Brown} or \cite{m2014}. For a given simplicial complex $\mathcal{C}$, we denote its set of simplices by $\S(\mathcal{C})$, the set of simplices of dimension $i$ by $\S_i(\mathcal{C})$ and the set of vertices by $\V(\mathcal{C}) = \S_0(\mathcal{C})$.

\begin{defn}
Let $r \in \N$.
The \emph{Bruhat-Tits building of $\PGL_r(\F)$} is an $(r-1)$-dimensional simplicial complex $\B^r := (\V, \S)$. Its set of vertices consists of homothety classes of lattices: \[\V := \lb [L] = \lb x L \| x \in \F^{\times}\rb \| L \text{ an } \Oi \text{-lattice in } \F^r \rb; \]
here an $\Oi$-lattice $L$ is a rank $r$ free $\Oi$-submodule in $\F^r$. The simplices are given by  
\[\S := \lb \{ v_0, \dots, v_k \} \, \Big| \, \forall i \in \{0, \dots, k \} \, \exists L_i \in v_i : L_0 \supsetneq L_1 \supsetneq \dots \supsetneq L_k \supsetneq \pi L_0 \rb, \]
where it is obvious that $k$ can be at most $(r-1)$. 
\end{defn}

From now on, we will only consider the case $r = 3$ with the 2-dimensional building $\B = \B^3$. We will refer to simplices in $\S_1(\B)$ as \emph{edges} and to those in $\S_2(\B)$ as \emph{chambers}.

\begin{defn}
Let $B = (b_1, b_2, b_3)$ denote a basis of $\F^3$. 

Then the maximal subcomplex $\A_B$ of $\B$ with set of vertices
\[ \V (\A_B) = \lb [\langle \pi^i b_1, \pi^j b_2, \pi^k b_3 \rangle_{\Oi} ] \| i, j, k \in \Z \rb \]
is called the \emph{apartment} attached to the ordered basis $B$.

The maximal subcomplex $\W_B$ of $\B$ with vertices
\[ \V (\W_B) = \lb [\langle \pi^i b_1, \pi^j b_2, \pi^k b_3 \rangle_{\Oi} ] \| i, j, k \in \Z, i \leq j \leq k \rb \]
is the \emph{sector} attached to $B$.
\end{defn}

\begin{defn} \label{def:StdApt}
Given $i, j, k \in \Z$ and $E = (e_1, e_2, e_3)$ the standard basis of $\F^3$, we define:
\begin{enumerate}[(i)]
\item $[i, j, k] := [\langle \pi^i e_1, \pi^j e_2, \pi^k e_3 \rangle_{\Oi} ]$, viewed as a vertex in $\B$,
\item the \emph{standard apartment} $\A := \A_E$, and
\item the \emph{standard sector} $\W := \W_E$.
\end{enumerate}
\end{defn}
Since the vertices are given as homothety classes, we have $[i, j, k] = [i + l, j + l, k + l]$ for every $l \in  \Z$. Therefore, we will normalize $i$ as 0 and often write
\[ [j, k] := [0, j, k] = [l, j + l, k + l]. \] 
The standard apartment and standard sector then have vertex sets
$$
\V(\A) \ = \ \lb [j, k] \| j, k \in \Z \rb \ \text{ and } \ \V(\W) \ = \ \lb [j, k] \| j, k \in \Z, 0 \leq j \leq k \rb.$$
An illustration of these is given in Figure \ref{fig:apartment-sector}.

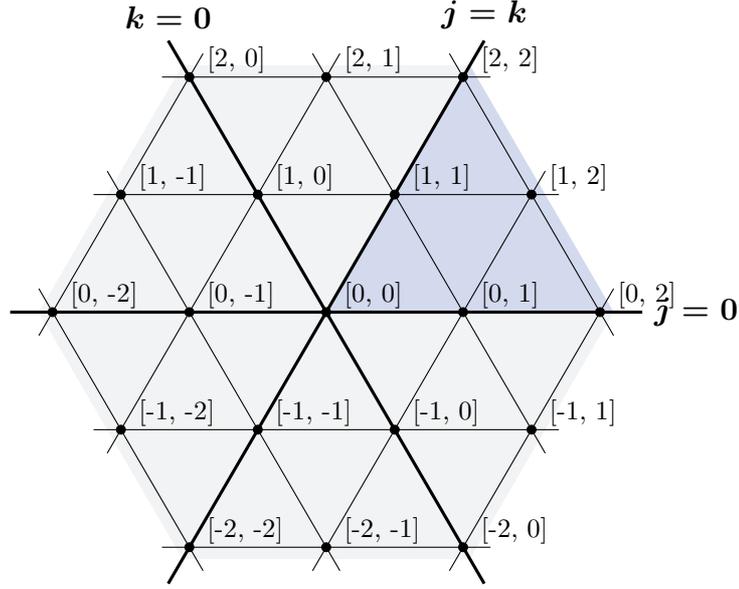
\begin{figure}[h!]
 \centering 
 \hspace*{0.9cm} \begin{tikzpicture} 
\tikzmath{\f = 1.8; \n = 2; \cont = 0.2;} 
\coordinate (Origin)   at (0,0);

\begin{scope}[y=(-60:1), label distance={-2mm}, font={\footnotesize}]
\pgfmathsetmacro{\bg}{(\n+0.5*\cont)*\f}
\path [fill=lightgray] (\bg,-\bg) -- (0,-\bg) -- (- \bg,0) -- (-\bg, \bg) -- (0,\bg) -- (\bg, 0); 
\path [fill=lightblue] (0,0) -- (\bg,-\bg) -- (\bg,0); 

\foreach \k in {-\n,...,\n}{

  \draw [] (\k*\f,- {(\n + min(0,\k))*\f}  - \cont *\f) -- (\k*\f,+ {(\n - max(0,\k))*\f} + \cont *\f);
  
  \draw [] (- {(\n + min(0,\k))*\f}  - \cont *\f, \k*\f) -- (+ {(\n - max(0,\k))*\f} + \cont *\f, \k*\f); 
  
  \draw [rotate=-60] (\k*\f,- {(\n + min(0,\k))*\f}  - \cont *\f) -- (\k*\f,+ {(\n - max(0,\k))*\f} + \cont *\f) ;
	\pgfmathsetmacro{\lower}{int(-\n+max(\k,0))}
	\pgfmathsetmacro{\upper}{int(\n+min(\k,0))}
  \foreach \j in {\lower,...,\upper}{
    \node[draw,circle,inner sep=0.1em,fill, label={above right:\hspace{0.2em} [\j, \k] } ] at (\k*\f,- \j*\f) {}; 
  }
}
\end{scope}
\coordinate (XAxisMin) at (-{2*\n/sqrt(3)}*\f,0);
\coordinate (XAxisMax) at ({2*\n/sqrt(3)*\f},0);
\coordinate (YAxisMin) at (-{\n/sqrt(3)}*\f,-\n*\f);
\coordinate (YAxisMax) at ({\n/sqrt(3)*\f},\n*\f);
\coordinate (ZAxisMin) at ({\n/sqrt(3)*\f},-\n*\f);
\coordinate (ZAxisMax) at (-{\n/sqrt(3)*\f},\n*\f);
\draw [line width=0.1em] (XAxisMin) -- (XAxisMax) node[right] {$\boldsymbol{j=0}$};
\draw [line width=0.1em] (YAxisMin) -- (YAxisMax) node[above] {$\boldsymbol{j=k}$} ;
\draw [line width=0.1em] (ZAxisMin) -- (ZAxisMax) node[above] {$\boldsymbol{k=0}$} ;
\end{tikzpicture}
 \caption{The standard apartment $\A$ with the standard sector $\W$ colored in blue.}
 \label{fig:apartment-sector}
\end{figure}

\begin{defn}\label{def:distance}
Let $r, s$ denote chambers in $\B$.
\begin{enumerate}[(i)]
\item One calls $r$ and $s$ \emph{adjacent} if they share an edge.
\item A sequence $(x_0, x_1, \dots, x_n)$ of adjacent chambers in $\B$ with $r=x_0$ and $s=x_n$ is called a \emph{gallery from $r$ to $s$}. The integer $n$ is called the  \emph{length} of the gallery.
\item The minimal length of a gallery from $r$ to $s$ is called the \emph{distance between $r$ and $s$} and denoted by $d(r, s)$.
\end{enumerate}
\end{defn}
Note that for $r,s\in\S_2(\A)$ a gallery of minimal length from $r$ to $s$ exists within $\A$; see \cite[Corollary~4.34]{abramenko-brown}.

In terms of shapes, cf.~Figure~\ref{fig:apartment-sector}, there are two types of shapes for chambers of $\A$:
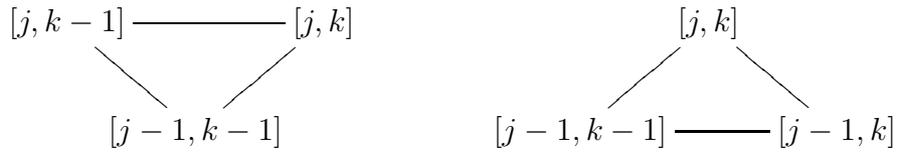
\[ \xymatrix@!0@C+2pc@R+1.4pc{ [j,k-1]\ar@{-}[rr]&&[j,k] &&&[j,k]\ar@{-}[dl]\ar@{-}[dr]&\\
&[j-1,k-1]\ar@{-}[ur]\ar@{-}[ul]&&&[j-1,k-1]\ar@{-}[rr]&&[j-1,k]}\]
\begin{defn} 
We define a sign function on chambers $s$ of $\A$ as follows: chambers of the left shape have $\sgn(s)=+1$, those of the right shape have $\sgn(s)=-1$.
\end{defn}

Note also that there are three types of edges, in terms of their angle relative to the horizontal axis, as is apparent from Figure~\ref{fig:apartment-sector}: 
\[\{[j-1,k-1],[j,k]\} ,\{[j-1,k],[j,k]\} ,\{[j,k-1],[j,k]\}.\]

As will be explained in Corollary~\ref{cor:stab-a}, the chamber 
\begin{equation}\label{eq:s0-def}
s_0:=\{[0,0],[-1,-1],[-1,0]\}    
\end{equation}\
is of particular importance to us. We have $\sgn(s_0)=1$.
\begin{lem} \label{lem:distance}
\begin{enumerate}
\item Let $s=\{[j,k],[j,k-1],[j-1,k-1]\}$ be a positive chamber of $\A$.
\begin{enumerate}[(a)]
\item If $k\cdot j\ge0$, i.e., if $k$ and $j$ have the same sign, then $d(s, s_0) = 2 \max \{ |j|, |k| \}$.
\item Else, $d(s, s_0) = 2 |j-k| = 2(|j| + |k|)$.
\end{enumerate}
\item Now let $s=\{[j,k],[j-1,k],[j-1,k-1]\}$ be a negative chamber of $\A$. Then 
\begin{enumerate}[(a)]
\item If $k\cdot j\ge0$, then $d(s, s_0) = \begin{cases}
  2 \max \{ |j|, |k| \} + 1, & \text{ if } j \leq k \\
  2 \max \{ |j|, |k| \} - 1, & \text{ if } j > k
\end{cases}$.
\item Else, $d(s, s_0) = \begin{cases}
  2 |j-k| + 1, & \text{ if } j \leq 0 \\
  2 |j-k| - 1, & \text{ if } j > 0
\end{cases}$.
\end{enumerate}
\end{enumerate}
In particular, $\sgn(s)=(-1)^{d(s,s_0)}$.
\end{lem}
\begin{proof}
    The proof can be obtained by visual inspection of Figure~\ref{fig:apartment-sector}. In order to prove 1., one can also show that $d(s,s_0)=\min\{|n|+|n+j+1|+|n+k+1|\mid n\in\Z\}$.
\end{proof}
\begin{defn}
We define a partial ordering on chambers $r,s$ of $\A$: We say \emph{$r$ is closer to the stable simplex $s_0$ than $s$} if $d(r, s_0) < d(s, s_0)$ holds. In this case, we write $r \prec s$.
\end{defn}

\subsection{Simplices as equivalence classes of matrices} \label{subsec:action-on-simplices}
The group $\GL_3(\F)$ acts on $\Oi$-lattices by matrix multiplication: For an $\Oi$-lattice $L = \langle b_1, b_2, b_3 \rangle_{\Oi}$ with basis $B = (b_1, b_2, b_3)$ and $\g \in \GL_3(\F)$, one defines
\[\g L := \g \langle b_1, b_2, b_3 \rangle_{\Oi} := \langle \g b_1, \g b_2, \g b_3 \rangle_{\Oi}\]
as the lattice with basis $\g B$. This is well-defined on homothety classes of lattices and hence induces an operation on the vertices and simplices in $\B$ by
\[
\g \{ [L_0], \dots, [L_k] \} := \{ [\g L_0], \dots, [\g L_k] \}.
\]

This operation is transitive on the sets $\S_0(\B)$, $\S_1(\B)$, and $\S_2(\B)$. Therefore computing the stabilizers of chosen {\em standard simplices} for each dimension, allows one to describe the simplices as equivalence classes of matrices. We introduce the following notation:  

\begin{defn}
\label{matrixR}
\begin{enumerate}[(i)]
\item We define the \emph{standard Iwahori subgroup of $\GL_3(\Oi)$} as
\[
I = \lb M \in \GL_3(\Oi)
\, \Big| \, M \equiv \left(\begin{smallmatrix}
 * & * & * \\
 0 & * & * \\
 0 & 0 & * \\
\end{smallmatrix}\right) \mod \pi, * \in \Fq \rb \subset \GL_3(\Oi).
\]
\item The \emph{first\footnote{We omit the second standard parahoric parahoric subgroup of $\GL_3(\Oi)$ as we do not use it.} standard parahoric subgroup of $\GL_3(\Oi)$} we define as 
\[
P = \lb M \in \GL_3(\Oi)
\, \Big| \, M \equiv \left(\begin{smallmatrix}
 * & * & * \\
 * & * & * \\
 0 & 0 & * \\
\end{smallmatrix}\right) \mod \pi, * \in \Fq \rb \subset \GL_3(\Oi).
\]
\item We denote by $R\in\GL_3(F_\infty)$ the matrix
 \[ R := \smatrix{0}{1}{0}{0}{0}{1}{\pi}{0}{0}.\]
\end{enumerate}
\end{defn}
We observe that $R$ lies in the the normalizer of $I$ in $\GL_3(F_\infty)$.

Now we are able to give the promised description.
\begin{thm}\label{thm:matsimpl}
The following maps are bijective, where in each case by $b_1, b_2$ and $b_3$ we denote the columns of the indeterminate matrix $g \in \GL_3(\F)$:
\begin{align*}
\GL_3(\F)/\GL_3(\Oi)\F^{\times}	& \rightarrow \V(\B), \\
g \GL_3(\Oi)\F^{\times}		 	& \mapsto    [g]_0 :=  [\langle b_1, b_2, b_3 \rangle_{\Oi}], \\[0.5 em]
\GL_3(\F)/P\F^{\times}	& \rightarrow \S_1(\B), \\
g P\F^{\times}			& \mapsto     [g]_1 := \{ [\langle b_1, b_2, b_3 \rangle_{\Oi}], [\langle b_1, b_2, \pi b_3 \rangle_{\Oi}] \}, \\[0.5 em]
\GL_3(\F)/\langle R \rangle I \F^{\times}	& \rightarrow \S_2(\B), 
\\
g \langle R \rangle I \F^{\times} 			& \mapsto  [g]_2 :=   \{ [\langle b_1, b_2, b_3 \rangle_{\Oi}], [\langle b_1, b_2, \pi b_3 \rangle_{\Oi}], [\langle b_1, \pi b_2, \pi b_3 \rangle_{\Oi}] \}.
\end{align*}
\end{thm}
\begin{proof}
See \cite[][Definition 1.9 and Satz 1.12]{m2014}.
\end{proof}

\begin{rmk}
The matrix $R$ takes into account that we only consider unordered simplices. Given a matrix $g \in \GL_3(\F)$, the images of $g, Rg$ and $R^2g$ contain exactly the same vertices, just in a different order. Note that $R^3 = \pi \left( \begin{smallmatrix}1&0&0\\0&1&0\\0&0&1\\ \end{smallmatrix} \right)$, so $R^3g = \pi g$ gives the same lattices as $g$ up to homothety, hence the same simplex, again.
\end{rmk}

In the following, we will mostly interpret simplices in $\B$ as equivalence classes of matrices and describe them by matrix representatives. 
For example, for $j, k \in \Z$ one has
\[ [j, k] = \lbk \smatrix{1}{0}{0}{0}{\pi^j}{0}{0}{0}{\pi^k} \rbk_0 \in \V(\A).
\]

\subsection{The action of congruence subgroups}
In this subsection, we introduce certain congruence subgroups of $\GL_3(A)$ and describe for these the stabilizers of the simplices of the standard apartment $\A$. In Corollary~\ref{cor:A-FD}, we show that the standard apartment is a fundamental domain in $B$ for the $\G_1(t)$-action. The final Theorem~\ref{thm:index_q} asserts that to each chamber $s$ of $\A$ there is a gallery $(t_0=s_0,t_1,\ldots,t_r=s)$ from the unique $\G_1(t)$-stable simplex $s_0$ in $\A$, such that in each step $i\to i+1$, the stabilizer size $|\Stab_{\G_1(t)}(s_i)|$ increases by a factor of~$q$. 
\subsubsection{Definitions}
We shall be interested in the action of $\GL_3(A)$ and of certain subgroups thereof.
\begin{defn}
For a monic polynomial $N \in A$ we define the \emph{principal congruence subgroup}
\[ \Gamma(N) := \lb \gamma \in \GL_3(A) \, \middle| \, \gamma \equiv  
\left(\begin{smallmatrix}
 1 & 0 & 0 \\
 0 & 1 & 0 \\
 0 & 0 & 1 \\
\end{smallmatrix}\right)
\mod N
\rb.
\]
A \emph{congruence subgroup} $\Gamma$ is a subgroup of $\GL_3(A)$ with $\Gamma(N) \subset \Gamma$ for some monic $N \in A$.
\end{defn}
To describe those congruence subgroups of $\GL_3(A)$ we are primarily interested in, we first introduce some subgroups of $\GL_3(\Fq)$.
\begin{defn} \label{def:borel} 
By the \emph{standard Borel subgroup} of $\GL_3(\Fq)$ we shall mean the group
\begin{align*}
B & := B(\Fq) := \lb \left(\begin{smallmatrix}
\alpha & a & b \\ 
0 & \beta & c \\ 
0 & 0 & \gamma
\end{smallmatrix}\right) \,\middle|\, a, b, c \in \Fq, \alpha, \beta, \gamma \in \Fq^{\times} \rb; \\
\intertext{its $p$-Sylow subgroup is the group of unipotent upper triangular matrices}
U & := U(\Fq) := \lb \left(\begin{smallmatrix}
1 & a & b \\ 
0 & 1 & c \\ 
0 & 0 & 1
\end{smallmatrix}\right) \,\middle|\,  a, b, c \in \Fq \rb; 
\\
\intertext{the subgroup of diagonal matrices in $\GL_3(\Fq)$ will be}
D & := D(\Fq) := \lb \left(\begin{smallmatrix}
a & 0 & 0 \\ 
0 & b & 0 \\ 
0 & 0 & c
\end{smallmatrix}\right) \,\middle|\, a, b, c \in \Fq^{\times} \rb ;\\
\intertext{we denote two standard parabolic subgroups of $\GL_3(\Fq)$ by}
P_0 & := P_0(\Fq) := \lb (a_{i,j})_{i,j\in\{1,2,3\}}\in\GL_3(\Fq)\,\middle|\, a_{3,1}=a_{3,2}=0\rb,\\
P_2 & := P_2(\Fq) := \lb (a_{i,j})_{i,j\in\{1,2,3\}}\in\GL_3(\Fq)\,\middle|\, a_{2,1}=a_{3,1}=0\rb;\\
\intertext{the subgroup of permutation matrices or the \emph{Weyl group} of $\GL_3(\Fq)$ will be}
W & := \lb \left(\begin{smallmatrix}
       1 & 0 & 0 \\ 
       0 & 1 & 0 \\ 
       0 & 0 & 1
       \end{smallmatrix}\right), \left(\begin{smallmatrix}
       1 & 0 & 0 \\ 
       0 & 0 & 1 \\ 
       0 & 1 & 0
       \end{smallmatrix}\right), \left(\begin{smallmatrix}
       0 & 1 & 0 \\ 
       1 & 0 & 0 \\ 
       0 & 0 & 1
       \end{smallmatrix}\right), \left(\begin{smallmatrix}
       0 & 0 & 1 \\ 
       0 & 1 & 0 \\ 
       1 & 0 & 0
       \end{smallmatrix}\right), \left(\begin{smallmatrix}
       0 & 1 & 0 \\ 
       0 & 0 & 1 \\ 
       1 & 0 & 0
       \end{smallmatrix}\right), \left(\begin{smallmatrix}
       0 & 0 & 1 \\ 
       1 & 0 & 0 \\ 
       0 & 1 & 0
       \end{smallmatrix}\right)  \rb.
\end{align*}
For shorter notation, we often identify elements of $W$ with corresponding cycle representatives in the permutation group $\mathfrak{S}_3$, so that the set $W$ written in the same order is $\{(),(23),(12),(13),(132),(123)\}.$
\end{defn}
We note that the Weyl group $W$ preserves the standard apartment. The following lemma records its action on the vertices of $\A$:
\begin{lem}\label{lem:S3-acts-OnA}
    Let $\sigma\in\mathfrak{S}_3$ and let $v=[n_1,n_2,n_3]\in\V(\A)$ for $(n_1,n_2,n_3)\in\Z^3$. Then the action of $\sigma$ permutes the entries of~$v$, i.e.,
    \[\sigma (v)=[n_{\sigma(1)},n_{\sigma(2)},n_{\sigma(3)}].\]
\end{lem}

Let 
\begin{equation}\label{eq:pr}
\pr: \GL_3(A) \rightarrow \GL_3(\Fq)    
\end{equation}
denote the componentwise reduction mod $t$, and note that its kernel is $\Gamma(t)$ and $\pr$ gives an identification $\GL_3(\Fq)/\Gamma(t)\cong\GL_3(\Fq)$. In the remainder of this article, we will focus on the following congruence subgroups of level $N=t$ of $\GL_3(A)$.
\begin{defn}\label{def:Gamma1} We set
\[
	\G_1(t) := \pr^{-1}(U)=\left\lbrace \g \overset{\operatorname{mod} t}{\equiv} \!\left(\begin{smallmatrix}1&*&*\\0&1&*\\0&0&1\end{smallmatrix}\right) \!\right\rbrace \subset
	\G_0(t) := \pr^{-1}(B)=\left\lbrace \g \overset{\operatorname{mod} t}{\equiv} \!\left(\begin{smallmatrix}*&*&*\\0&*&*\\0&0&*\end{smallmatrix}\right) \!\right\rbrace \subset
	\GL_3(A),
	\]
    where $*$ denotes arbitrary entries in $\Fq$, and moreover we define $\G^P_{0},\G^P_{2}\subset\GL_3(A)$ by 
	\[ \G_0(t)
    \begin{array}{ccc}
	\rotatebox[origin=c]{30}{$\subset$} & \G^P_0 := \pr^{-1}(P_0)=\left\lbrace \g \overset{\operatorname{mod} t}{\equiv} \left(\begin{smallmatrix}*&*&*\\ *&*&*\\0&0&*\end{smallmatrix}\right) \right\rbrace & \rotatebox[origin=c]{-30}{$\subset$}\\[0.5em] 
	\rotatebox[origin=c]{-30}{$\subset$} & \G^P_2 := \pr^{-1}(P_2)=\left\lbrace \g \overset{\operatorname{mod} t}{\equiv} \left(\begin{smallmatrix}*&*&*\\ 0&*&*\\0&*&*\end{smallmatrix}\right) \right\rbrace & \rotatebox[origin=c]{30}{$\subset$}
	\end{array} 
	\GL_3(A).
	\] 
\end{defn}
Via $\GL_3(A)\subset\GL_3(F_\infty)$, congruence subgroups act on $\B$, and it will be useful to introduce a notion of fundamental domain in this context.
\begin{defn}\label{def:FD}
Let $\G \subset \GL_3(A)$ denote a congruence subgroup. A (connected) subset $\mathcal{F}$ of the simplices $\S(\B)$ is called \emph{fundamental domain} for the action of $\G$ if it contains exactly one element of each $\G$-orbit. 
\end{defn}

\begin{rmk}
    A fundamental domain might not necessarily be a subcomplex of $\B$, but all the ones that we encounter in this article are.

    A fundamental domain $\mathcal{F}$ as in Definition~\ref{def:FD} is quite different from classical notions, e.g., for the action of congruence subgroups of $\SL_2(\Z)$ on the upper half plane. It is not required (and would make no sense) that the stabilizers of simplices in $\mathcal{F}$ are trivial.
\end{rmk}

\begin{notat}
For a subgroup $\G \subset \GL_3(A)$ and a simplex $s \in \S(\B)$ we write
\[ \G_s := \Stab_{\G}(s). \]
\end{notat}

To concretely describe stabilizers of simplices in $\B$, under the action of congruence subgroups $\G\subset\GL_3(A)$, the following notation will be useful. 
\begin{notat}
\label{schreibweise}
We define the following symbols as a shorthand for certain sets, where $k \in \Z$ is arbitrary.
\begin{align*}
0 & : \text{ The singleton set containing zero, i.\,e. }0\in  \Fq. \\
1 & : \text{ The singleton set containing one, i.\,e. }1\in  \Fq. \\
* & : \text{ The set } \Fq. \\
\{ k \} & : \text{ The set of polynomials of degree $\leq k$ in } \Fq[t]. \\
t \{ k \} & : \text{ The set of polynomials of degree $\leq k+1$ in } \Fq[t] \text{ which are divisible by } t.
\end{align*}
Clearly, the symbols $\{0\}$ and $*$ have the same meaning. In the case $k < 0$, our symbols satisfy $\{k\}= t \{ k \}=0$. Note that taking the intersection $\{k_1\} \cap \{k_2\}$ clearly is equal to $\{ \min \{k_1, k_2\}\}$.

Now for $i, j \in \{1, 2, 3\}$ choose a symbol $\epsilon_{ij}$ from the list above. Then by the matrix $(\epsilon_{ij})_{i, j \in \{1, 2, 3\}}$ we mean a set of matrices, namely 
\begin{align*}
(\epsilon_{ij})_{i, j \in \{1, 2, 3\}} :=& \lb (a_{ij})_{i, j \in \{1, 2, 3\}} \in \GL_3(A) \,\middle|\, a_{ij} \in\epsilon_{ij}\, \forall i, j \in \{1, 2, 3\}\rb.
\end{align*}
\end{notat}

\subsubsection{The action of $\GL_3(A)$}
\begin{prop} \label{prop:w-fundamental}
The standard sector $\W$ is a fundamental domain for the action of $\GL_3(A)$ on the simplices of $\B$.
\end{prop}
\begin{proof}
See \cite[][Theorem 7.8]{gr2021} or \cite[][Satz 3.18]{geb1996}.
\end{proof}

\begin{prop}
\label{prop:stab_a}
The stabilizer of the vertex $[j, k]$ in the standard apartment $\A$ is
\[
\GL_3(A)_{[j, k]} = 
\begin{pmatrix}
 * & \{j\} & \{k\} \\
 \{-j\} & * & \{k-j\} \\
 \{-k\} & \{j-k\} & * \\
\end{pmatrix}.
\]
The stabilizer of a higher dimensional simplex can be determined as the intersection of the stabilizers of its vertices.
\end{prop} 
\begin{proof}
See \cite[][Lemma 7.11]{gr2021} and \cite[][Proposition 7.4]{gr2021} or \cite[][Satz 2.15]{geb1996} and \cite[][Korollar 3.5]{geb1996}.
\end{proof}

We obtain the following immediate corollary whose proof is left as an exercise.
\begin{cor}\label{cor:mod-t-stabilizers}
Let $\pr$ be the map from \eqref{eq:pr}. 
For the vertices $[j, k]$ in $\V(\W)$ we have
\[ 
\pr(\GL_3(A)_{[j, k]}) = 
\begin{cases}
\GL_3(\Fq),~ & \text{if } 0 = j = k, \\
P_0, & \text{if } 0 = j < k, \\
P_2,~ & \text{if } 0 < j = k,  \\
B,~ & \text{if } 0 < j < k. 
\end{cases}
\]
For higher dimensional simplices $w \in \S(\W)$, the group $\pr(\GL_3(A)_w)$ is the intersection $\bigcap_{v\in w} \pr(\GL_3(A)_v)$ over the vertices $v$ of~$w$. 
\end{cor}
Note that because of the shape of the $\GL_3(A)_w$, forming the intersection $\bigcap_{v\in w}$ indeed commutes with $\pr$.

\subsubsection{The action of $\G_1(t)$}
In this subsection, we will show that the standard apartment $\A$ is a fundamental domain for the action of $\G_1(t)$, and make explicit the $\G_1(t)$-stabilizers of its simplices. To do this we shall make explicit the $\G_1(t)$-orbits of $\GL_3(A)w$ for any $w\in\W$. Let us give some details.

For $w\in\S(\W)$, the orbit-stabilizer theorem gives us a bijection
\[
    \GL_3(A) w  \cong \GL_3(A) / \GL_3(A)_w, 
    \g w  \mapsto \g \GL_3(A)_w.
\]
Now $\G_1(t)$ operates from the left on both sides, and taking $\G_1(t)$-orbits gives
\begin{align*}
\G_1(t) \backslash \GL_3(A) w & \cong \G_1(t) \backslash \GL_3(A) / \GL_3(A)_w, \\
\G_1(t) \g w & \mapsto \G_1(t) \g \GL_3(A)_w.    
\end{align*}
To understand the right hand side we shall rely on the following elementary lemma from group theory whose simple proof we leave to the reader.
\begin{lem}
    Let $G$ be a group, let $N$ be a normal subgroup and write $\pi:G\to G/N=:\bar G$ for the natural factor map. Let $H,K\subset G$ be subgroups and assume $N\subset K$. Then $\pi$ yields an identification of double cosets
    \[ \pi:K\backslash G/H \stackrel{\simeq}\longrightarrow \pi(K)\backslash \bar G/\pi(H).\]
\end{lem}
We apply the lemma with $G=\GL_3(A)$, $N=\Gamma(t)$, $K=\Gamma_1(t)$, $H=\GL_3(A)_w$ and $\pi=\pr:G\to\GL_3(\Fq)$. This gives

\begin{prop}\label{prop:G1-Orbits-InW}
For every $w \in \S(\W)$ we have a bijective map
\label{reduktion} 
\begin{align*}
\pr: \G_1(t) \backslash \GL_3(A) / \GL_3(A)_w  & \rightarrow 
U \backslash \GL_3(\Fq) / \widebar{\GL_3(A)_w}, \\
\G_1(t) \, \alpha \, \GL_3(A)_w & \mapsto U\, \widebar{\alpha}  \,\widebar{\GL_3(A)_w}.
\end{align*}
Here an overbar means that we apply the reduction mod $t$ map $\pr$. 
\end{prop}

We can now rely on Corollary~\ref{cor:mod-t-stabilizers} to compute the right hand sides in Proposition~\ref{prop:G1-Orbits-InW} for all $w\in\S(\W)$.
We distinguish three cases.
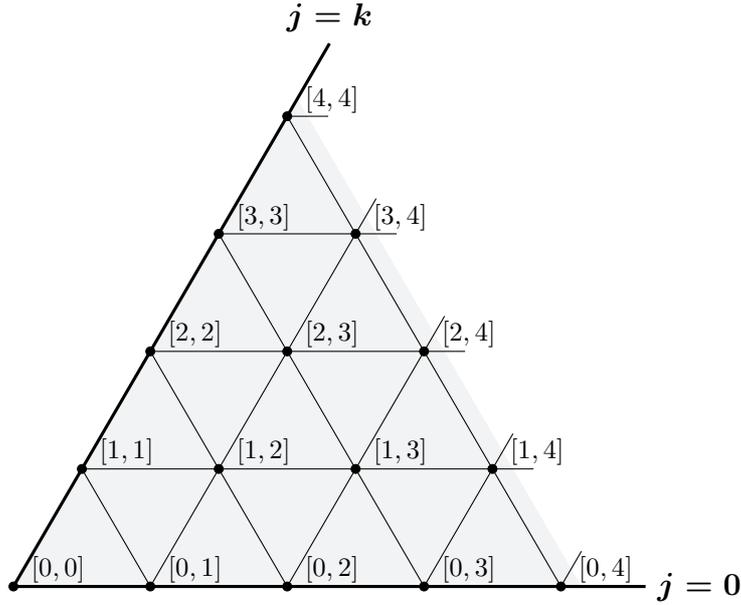
\begin{figure}[ht]
 \centering 
 \begin{tikzpicture} 
\tikzmath{\f = 1.8; \n = 4; \cont = 0.3;} 
\coordinate (Origin)   at (0,0);

\begin{scope}[y=(-60:1), label distance={-2mm}, font={\footnotesize}]
\pgfmathsetmacro{\bg}{(\n+0.5*\cont)*\f}

\path [fill=lightgray] (0,0) -- (\bg,-\bg) -- (\bg,0); 

\foreach \k in {0,...,\n}{

  \draw [] (\k*\f, -\k*\f) -- (\k*\f, 0);
  
  \draw [] (\k*\f, -\k*\f) -- (+ {(\n - max(0,-\k))*\f} + \cont *\f, -\k*\f); 
  
  \draw [rotate=-60] (\k*\f,- {(\n + min(0,\k))*\f}  - \cont *\f) -- (\k*\f, -\k*\f) ;
  \foreach \j in
    {0,...,\k}{
    \node[draw,circle,inner sep=0.1em,fill, label={above right:\hspace{0.2em} $[{\j}, {\k}]$ } ] at (\k*\f,- \j*\f) {};
  }
}
\end{scope}
\coordinate (XAxisMin) at (0,0);
\coordinate (XAxisMax) at ({2*\n/sqrt(3)*\f},0);
\coordinate (YAxisMin) at (0,0);
\coordinate (YAxisMax) at ({\n/sqrt(3)*\f},\n*\f);
\draw [line width=0.1em] (XAxisMin) -- (XAxisMax) node[right] {$\boldsymbol{j=0}$};
\draw [line width=0.1em] (YAxisMin) -- (YAxisMax) node[above] {$\boldsymbol{j=k}$} ;
\end{tikzpicture}
 \caption{Illustration of the standard sector $\W$}
 \label{fig:sector}
\end{figure}

\subparagraph{The origin}
Only for the vertex $w = [0, 0] \in \S_0(\W)$ we do get $\widebar{\GL_3(A)_w} = \GL_3(\Fq)$. Then the double quotient 
\[
U \backslash \GL_3(\Fq) / \widebar{\GL_3(A)_{[0, 0]}} = 
U\backslash \GL_3(\Fq) / \GL_3(\Fq)
\]
is a singleton set, and we choose $[0,0]\in\A$ as its single representative.

\subparagraph{The sides of $\W$}
For all vertices $w = [0, k] \in \S_0(\W)$ with $0 < k$ and for all edges of the form $w = \{ [0, k-1], [0, k] \} \in \S_1(\W)$ with $0 < k$ (i.\,e. on the ray $j=0$ in Figure \ref{fig:sector}) we get  
\[
\widebar{\GL_3(A)_w} = P_0=\left(\begin{smallmatrix}*&*&*\\ *&*&*\\0&0&*\end{smallmatrix}\right)
\] 
Similarly, for all vertices $w = [j, j] \in \S_0(\W)$ with $0 < j$ and for all edges of the form $w = \{ [j-1, j-1], [j, j] \} \in \S_1(\W)$ with $0 < j$ (i.\,e. on the ray $j=k$ in Figure \ref{fig:sector}) we get 
\[
\widebar{\GL_3(A)_w} = P_2
\] 
In both cases, the double cosets  
\[
U\backslash \GL_3(\Fq)/P_0 
\text{ and }
U\backslash \GL_3(\Fq)/P_2 
\]
have the same set of representatives (see Proposition \ref{prop:double-quot-fund} for details), namely 
\[(),(123),(132)\in W.\] 

These form a subgroup of $\mathfrak{S}_3$ of order $3$. Using Lemma~\ref{lem:S3-acts-OnA}, we can make explicit the representatives in $\A$ obtained via these representatives for vertices; the simple extension to edges is left to the reader:
\[()[j,j]\hspace{-.21pt}=\hspace{-.21pt}[j,j],  (123)[j,j]\hspace{-.21pt}=\hspace{-.21pt}(123)[0,j,j]\hspace{-.21pt}=\hspace{-.21pt}[j,0,j]\hspace{-.21pt}=\hspace{-.21pt}[-j,0],  (132)[j,j]\hspace{-.21pt}=\hspace{-.21pt}[j,j,0]\hspace{-.21pt}=\hspace{-.21pt}[0,-j], \]
i.e., in Figure~\ref{fig:apartment-sector} or Figure \ref{fig:permutationen} we rotate by $2\pi/3$, and the same rotation pattern holds for vertices of the form $[0,j]$. 

\subparagraph{The interior}
The remaining simplices are all vertices $w = [j, k] \in \S_0(\W)$ with $0 < j < k$, all edges $w \in \S_1(\W)$ containing at least one of those vertices, and all chambers $w \in \S_2(\W)$. For all these, one has 
\[
\widebar{\GL_3(A)_w} = B.
\]
For this case, we rely on a classic result from linear algebra, that we specialize to $\GL_3(\Fq)$.
\begin{prop}[Bruhat decomposition, cf.~{\cite[][Chapter 27]{bum2013}}] \label{prop:bruhat-decomp}
\!The group $\GL_3(\Fq)$ is the disjoint union
\[ \GL_3(\Fq) = \bigsqcup_{\sigma \in \mathfrak{S}_3} U \sigma B,
\]
where as usual we identify $\mathfrak{S}_3$ with the Weyl group~$W$.
\end{prop}
We deduce that in the present case we have
\[U \backslash \GL_3(\Fq) / \widebar{\GL_3(A)_w}=
U \backslash \GL_3(\Fq) / B=\mathfrak{S}_3.\]
\begin{figure}[H]
 \centering
 \hspace*{0.9cm}\begin{tikzpicture} 
\tikzmath{\f = 1.1; \n = 4; \cont = 0.6;} 
\coordinate (Origin)   at (0,0);

\begin{scope}[y=(-60:1), label distance={-2mm}, font={\footnotesize}]
\pgfmathsetmacro{\bg}{(\n+0.5*\cont)*\f}
\path [fill=lightgray] (\bg,-\bg) -- (0,-\bg) -- (- \bg,0) -- (-\bg, \bg) -- (0,\bg) -- (\bg, 0); 
\end{scope}
\coordinate (XAxisMin) at (-{2*\n/sqrt(3)}*\f,0);
\coordinate (XAxisMax) at ({2*\n/sqrt(3)*\f},0);
\coordinate (YAxisMin) at (-{\n/sqrt(3)}*\f,-\n*\f);
\coordinate (YAxisMax) at ({\n/sqrt(3)*\f},\n*\f);
\coordinate (ZAxisMin) at ({\n/sqrt(3)*\f},-\n*\f);
\coordinate (ZAxisMax) at (-{\n/sqrt(3)*\f},\n*\f);
\draw [line width=0.1em] (XAxisMin) -- (XAxisMax) node[right] {$\boldsymbol{j=0}$};
\draw [line width=0.1em] (YAxisMin) -- (YAxisMax) node[above] {$\boldsymbol{j=k}$} ;
\draw [line width=0.1em] (ZAxisMin) -- (ZAxisMax) node[above] {$\boldsymbol{k=0}$} ;

\begin{scope}[y=(-60:1), label distance={-0mm}, font={\footnotesize}]
\foreach \k in {-\n,...,\n}{
  \draw [] (\k*\f,- {(\n + min(0,\k))*\f}  - \cont *\f) -- (\k*\f,+ {(\n - max(0,\k))*\f} + \cont *\f); 
  \draw [] (- {(\n + min(0,\k))*\f}  - \cont *\f, \k*\f) -- (+ {(\n - max(0,\k))*\f} + \cont *\f, \k*\f); 
  \draw [rotate=-60] (\k*\f,- {(\n + min(0,\k))*\f}  - \cont *\f) -- (\k*\f,+ {(\n - max(0,\k))*\f} + \cont *\f) ;
}

  \path [draw=fullblue, line width=1pt, fill=lightblue] (3*\f,-3*\f) -- (3*\f,-2*\f) -- (2*\f,-2*\f) -- cycle ;
  \path [draw=fullblue, line width=1pt, fill=lightblue] (3*\f,-3*\f) -- (2*\f,-3*\f) -- (2*\f,-2*\f) -- cycle ; 

  \path [draw=fullblue, line width=1pt, fill=lightblue] (-2*\f,0*\f) -- (-3*\f,0*\f) -- (-2*\f,-1*\f) -- cycle ; 
  \path [draw=fullblue, line width=1pt, fill=lightblue] (-2*\f,0*\f) -- (-3*\f,0*\f) -- (-3*\f,1*\f) -- cycle ; 
  
  \path [draw=fullblue, line width=1pt, fill=lightblue] (0*\f,3*\f) -- (0*\f,2*\f) -- (1*\f,2*\f) -- cycle ; 
  \path [draw=fullblue, line width=1pt, fill=lightblue] (0*\f,3*\f) -- (0*\f,2*\f) -- (-1*\f,3*\f) -- cycle ; 

\foreach \k in {-\n,...,\n}{
	\pgfmathsetmacro{\lower}{int(-\n+max(\k,0))}
	\pgfmathsetmacro{\upper}{int(\n+min(\k,0))}
  \foreach \j in
    {\lower,...,\upper}{
    \node[draw,circle,inner sep=0.07em,fill ] at (\k*\f,- \j*\f) {};
  }
}

    \node[draw,circle,inner sep=0.07em,fill=fullblue, label={[fill=lightblue] below: $[j, k]$}] at (3*\f,-2*\f) {}; 
    \node[draw,circle,inner sep=0.07em,fill=fullblue, label={[fill=lightblue, align=right] above: $\left( \begin{smallmatrix} 1&0&0\\0&0&1\\0&1&0 \end{smallmatrix}\right) [j, k]$\\$= [k, j]$}] at (2*\f,-3*\f) {}; 

    \node[draw,circle,inner sep=0.07em,fill=fullblue, label={[fill=lightblue, align=right] above: $\left( \begin{smallmatrix} 0&1&0\\0&0&1\\1&0&0 \end{smallmatrix}\right) [j, k]$\\$= [k-j, -j]$}] at (-2*\f,-1*\f) {}; 
    \node[draw,circle,inner sep=0.07em,fill=fullblue, label={[fill=lightblue, align=right] below: $\left( \begin{smallmatrix} 0&0&1\\0&1&0\\1&0&0 \end{smallmatrix}\right) [j, k]$\\$= [j-k, -k]$}] at (-3*\f,1*\f) {}; 

    \node[draw,circle,inner sep=0.07em,fill=fullblue, label={[fill=lightblue, align=right] above: $\left( \begin{smallmatrix} 0&1&0\\1&0&0\\0&0&1 \end{smallmatrix}\right) [j, k]$\\$= [-j,k-j]$}] at (1*\f,2*\f) {}; 
    \node[draw,circle,inner sep=0.07em,fill=fullblue, label={[fill=lightblue, align=right] below: $\left( \begin{smallmatrix} 0&0&1\\1&0&0\\0&1&0 \end{smallmatrix}\right) [j, k]$\\$= [-k,j-k]$}] at (-1*\f,3*\f) {};
\end{scope}
\end{tikzpicture}
 \caption{Orbit of an interior chamber with vertex $[j, k]$ under the action of $W = \mathfrak{S}_3$}
 \label{fig:permutationen}
\end{figure}
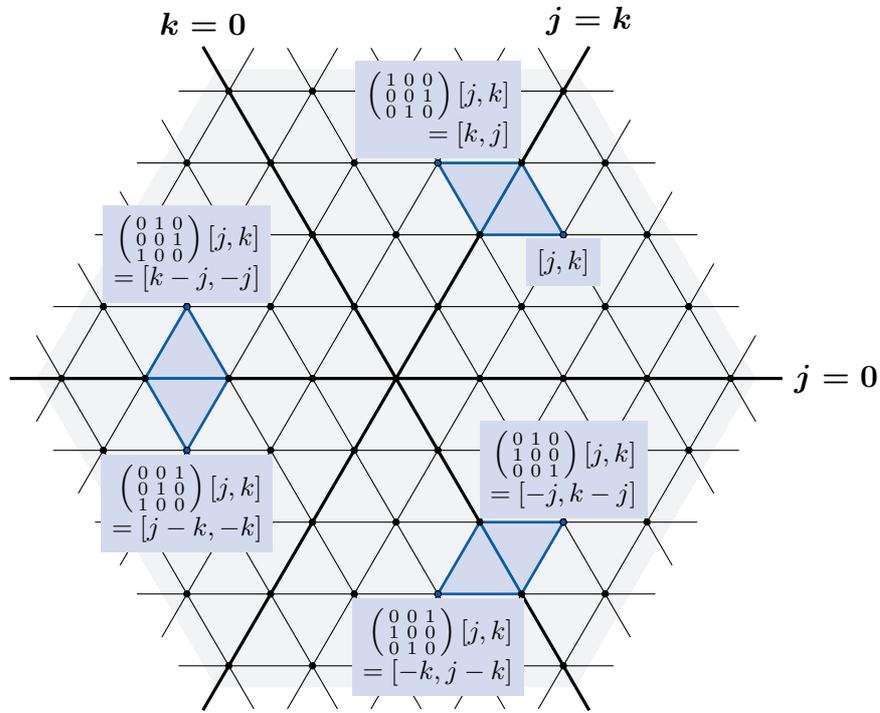
Hence the $\mathfrak{S}_3$-orbit in $\A$ of $[j,k]$ is a set of representatives for the $\G_1(t)$-orbit of $[j,k]$ in $\B$, and Lemma~\ref{lem:S3-acts-OnA} gives this $\mathfrak{S}_3$-orbit concretely as
\begin{eqnarray*}
\lefteqn{\{[0,j,k],[j,0,k],[j,k,0],[0,k,j],[k,0,j],[k,j,0]\}}\\
&=\{[j,k],[-j,k-j],[k-j,-j],[k,j],[-k,j-k],[j-k,-k]\}.
\end{eqnarray*}

A similar computation gives the result for edges and chambers. An illustration is given in Figure  \ref{fig:permutationen}.

Altogether we have proved the following result:
\begin{cor}\label{cor:A-FD}
The standard apartment $\A$ is a fundamental domain for the action of $\G_1(t)$ on $\B$.
\end{cor}

For later use, we also record the following variant of the Bruhat decomposition:
\begin{prop} \label{prop:bruhat_decomp}
We can write $\GL_3(\Fq)$ as a disjoint union
\[\GL_3(\Fq) = \bigsqcup_{w \in W} U w B^T, \]
where $B^T = \lb M^T | M \in B \rb$ is the set of lower triangular matrices in $\GL_3(\Fq)$. 
\end{prop}
\begin{proof}
Let $\tau=(13)\in W$ (be the long word of $W$), and observe that we have $B^T=\tau B\tau$. Now, using $\GL_3(\Fq)=\GL_3(\Fq)\tau$ and $\mathfrak{S}_3\tau=\mathfrak{S}_3$, from Proposition \ref{prop:bruhat-decomp} we deduce
\begin{align*}
\GL_3(\Fq)& = 
\bigsqcup_{\sigma \in \mathfrak{S}_3} U \sigma B\tau = 
\bigsqcup_{\sigma \in \mathfrak{S}_3} U \sigma \tau B\tau = \bigsqcup_{\sigma \in \mathfrak{S}_3} U \sigma B^T.
\end{align*}
\end{proof}

\subsubsection{$\Gamma_1(t)$-stabilizers}
We make more observations on the relation of simplices in the standard apartment that will become useful in the following Section \ref{sec:cocycles}. For that purpose, we examine the stabilizers $\G_1(t)$ of all the simplices that occur in the fundamental domain $\A$ in more detail. We will call a simplex \emph{$\G_1(t)$-stable} if its stabilizer is trivial.
\begin{cor} \label{cor:stab-a}
We again use Notation \ref{schreibweise} to describe the $\G_1(t)$-stabilizers of simplices in the standard apartment. For a vertex $[j,k]$ in $\A$ we have
\[ \G_1(t)_{[j, k]} = \begin{pmatrix}
 1 & \{j\} & \{k\} \\
 t \{-j - 1\} & 1 & \{k-j\} \\
 t \{-k -1 \} & t \{j-k - 1\} & 1 \\
\end{pmatrix}. \]
    The stabilizer of a higher-dimensional simplex is the intersection of the stabilizers of its vertices.    The only stable simplex in $\S(\A)$ is the chamber $s_0$ from \eqref{eq:s0-def}.
\end{cor}
\begin{proof}
The formula for the stabilizers is immediate from Proposition \ref{prop:stab_a} by taking the intersection with $\G_1(t)$. The proof that $s_0$ is the only stable simplex follows by inspection of the chamber stabilizers displayed below in formulas \eqref{eq:Stab-s-} and \eqref{eq:Stab-s+}. From this it follows that only simplices contained in $s_0$ could be stable. But by direct computation, one sees that this only holds for $s_0$ itself.
\end{proof}

Now we can formulate the main result of this subsection.

\begin{thm}
\label{thm:index_q}
For each edge $e \in \S_1(\A)$ in the standard apartment, the following applies: $e$ is a face of exactly two chambers $r, s \in \S_2(\A)$ in the standard apartment. One of these is closer to the stable chamber than the other, such as $r \prec s$. Then the following applies 
\[\G_1(t)_r \subset \G_1(t)_s = \G_1(t)_e \text{ and } |\G_1(t)_s / \G_1(t)_r| = q.\]
If we choose representatives $\gamma_1, \dots, \gamma_q$ of $\G_1(t)_s / \G_1(t)_r$, these operate transitively on the $q$ different chambers except $s$, which are adjacent to $e$.
In particular, all chambers other than $s$ that are adjacent to $e$ are equivalent to $r$ under the action of $\G_1(t)$.
\end{thm}
\begin{proof}
The proof is similar to the proof of \cite[][Proposition 7.14]{gr2021}. 
As was explained earlier, there are three types of edges, in terms of their angle relative to the horizontal axis: 
\[\{[j-1,k-1],[j,k]\} ,\{[j-1,k],[j,k]\} ,\{[j,k-1],[j,k]\}.\]
We explain the proof for an edge $e = \{[j-1,k-1],[j,k]\}$. The other cases work similarly and are left as an exercise to the reader.

\begin{figure}[h!]
 \centering 
 \begin{tikzpicture} 
\tikzmath{\f = 0.9;} 
\begin{scope}[y=(-60:1), label distance={-2mm}, font={\footnotesize}]
    \draw [fill=lightgray] (\f,-\f) -- (\f, \f) -- (-\f, \f) -- (\f, -\f) -- (-\f, -\f) -- (-\f, \f) ; 
    \node [fill=lightgray,rounded corners=2pt,inner sep=1pt] at (0,0) {$e$};
    \node[draw,circle,inner sep=0.1em,fill, label={above right: $[j, k]$ } ] at (\f,- \f) {};
    \node[draw,circle,inner sep=0.1em,fill, label={below left: $[j-1, k-1]$ } ] at (-\f, \f) {};
    \node[draw,circle,inner sep=0.1em,fill, label={above left: $[j, k-1]$ } ] at (-\f,- \f) {};
    \node[draw,circle,inner sep=0.1em,fill, label={below right: $[j-1, k]$ } ] at (\f, \f) {};
    \node at (-\f/3.0 ,-\f/3.0) {$s_+$};
    \node at (\f/3.0, \f/3.0) {$s_-$};
\end{scope}
\end{tikzpicture}
 \caption{An edge $e = \{[j-1,k-1],[j,k]\}$ with its two adjacent simplices in $\A$}
 \label{fig:two-simplices}
\end{figure}
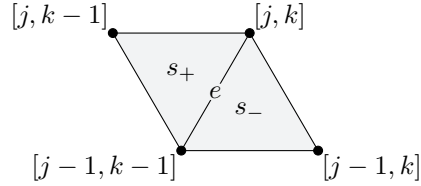
In our case, the adjacent chambers in $\A$ are $s_+ = \{[j,k],[j,k-1],[j-1,k-1]\}$ and $s_- = \{[j,k],[j-1,k],[j-1,k-1]\}$, see also Figure \ref{fig:two-simplices}. We use Corollary \ref{cor:stab-a} to compute  
\begin{align}
\nonumber\G_1(t)_e & = \begin{pmatrix} 
 1 & \{j-1\} & \{k-1\} \\
 t \{-j - 1\} & 1 & \boldsymbol{\{k-j\}} \\
 t \{-k -1 \} & \boldsymbol{t \{j-k - 1\}} & 1 \\
\end{pmatrix}, \\[0.5em]
\label{eq:Stab-s+}\G_1(t)_{s_+} & = \begin{pmatrix}
 1 & \{j-1\} & \{k-1\} \\
 t \{-j - 1\} & 1 & \boldsymbol{\{k-j-1\}} \\
 t \{-k -1 \} & t \{j-k - 1\} & 1 \\
\end{pmatrix}, \\[0.5em]
\label{eq:Stab-s-}\G_1(t)_{s_-} & = \begin{pmatrix}
 1 & \{j-1\} & \{k-1\} \\
 t \{-j - 1\} & 1 & \{k-j\} \\
 t \{-k -1 \} & \boldsymbol{t \{j-k - 2\}} & 1 \\
\end{pmatrix}.
\end{align} 
The entries where the stabilizers of the chambers $s_+$ or $s_-$ might differ from the stabilizer of the edge $e$ are typeset in bold. We distinguish two cases.\\[.3em]
    \textbf{Case $j > k$:}
    In this case, using Notation \ref{schreibweise}, $\{k-j\} = \{k-j-1\} = 0$, so that 
    \[\G_1(t)_e = \G_1(t)_{s_+} \supset \G_1(t)_{s_-}.\] 
    Also, $|t\{j-k-1\}| = q \cdot |t\{j-k-2\}|$, because $j-k-1 \ge 0$,
    so  $|\G_1(t)_{s_+} / \G_1(t)_{s_-}| = q$. Moreover from Lemma \ref{lem:distance} we deduce $s_- \prec s_+$.\\[.3em]
    \textbf{Case $j \leq k$:} In this case, using Notation \ref{schreibweise}, $t\{j-k-1\} = t\{j-k-2\} = 0$, and so 
    \[\G_1(t)_e = \G_1(t)_{s_-} \supset \G_1(t)_{s_+}.\] 
    Also, $|\{k-j\}| = q \cdot |\{k-j-1\}|$, because $k-j\ge0$,
    so  that$|\G_1(t)_{s_-} / \G_1(t)_{s_+}| = q$, and moreover Lemma \ref{lem:distance} yields $s_+ \prec s_-$.

This concludes the proof for edges of the sample shape $e = \{[j-1,k-1],[j,k]\}$.
\end{proof}

\section{Harmonic Cocycles}
\label{sec:cocycles}
In this section, we introduce harmonic cocycles as special functions on the chambers of the Bruhat-Tits building. We introduce a $\GL_3(F)$-action on these functions and give simple descriptions of the spaces of harmonic cocycles that are invariant under this action for various congruence subgroups of $\GL_3(A) \subset \GL_3(F)$. This relies on the results of the previous Section \ref{sec:building}, especially on Theorem \ref{thm:index_q}.
\subsection{Some definitions}
Let $V$ denote an $F$-vector space. Then by $\Maps(\S_2(\B),V)$ we denote the set of maps from the chambers of $\B$ to $V$. It is in a natural way an $F$-vector space.
\begin{defn}
A \emph{harmonic cocycle with values in $V$} is a map $c: \S_2(\B) \rightarrow V$ satisfying the following harmonicity condition: For every edge $e \in \S_1(\B)$ the sum of the values at the chambers having $e$ as a facet is zero:
\[ \sum_{s \supset e} c(s) = 0.\] 
Since these conditions are given by the vanishing of functionals on $\Maps(\S_2(\B),V)$ the harmonic cocyles form an $F$-sub vector space $\Char(V)$ of $\Maps(\S_2(\B),V)$.
\end{defn}

\begin{defn}\label{def:action}
Suppose in a addition that $V$ carries an $F$-linear action by $\GL_3(F)$. Then this induces a $\GL_3(F)$-representation on $\Maps(\S_2(\B),V)$ by defining the action of $\g \in \GL_3(F)$ on $c \in \Maps(\S_2(\B),V)$ by
\[ (c | \g) (s) := \g^{-1} c(\g s)\hbox{ for all }s \in \S_2(\B); \]
here, $\g s$ is the action of $\GL_3(\F)$ on $\S_2(\B)$ and the outer $\g^{-1} c(\g s)$ is the action of $\GL_3(F)$ on $c(\g s) \in V$. 
\end{defn}
\begin{lem}\label{lem:harmonicity-under-G}
The action of $\GL_3(F)$ on $\Maps(\S_2(\B),V)$ preserves the sub vector space $\Char(V)$.
\end{lem}
\begin{proof}
    Let $c \in \Char(V)$ and $e \in \S_1(\B)$ any edge. Then for any $\g \in \GL_3(F)$,
\[
\sum_{s \supset e} (c | \g)(s) = \sum_{s \supset e} \g^{-1} c (\g s) = \g^{-1} \sum_{s' \supset \g e} c (s') = 0,
\]
thanks to the harmonicity of $c$ at the edge $\g e$. 
\end{proof}
\begin{defn}
Let $\G \subset \GL_3(F)$ denote a subgroup. A harmonic cocycle $c \in \Char(V)$ is called \emph{$\G$-invariant} if it is invariant under $\G$ for the action defined above. 

We write $\Char(\G, V) := (\Char(V))^{\Gamma}\subset \Char(V)$ for the $F$-sub vector space of $\G$-equivariant harmonic cocycles. 
\end{defn}
\begin{rmk}
    For $c\in\Char(V)$ being $\G$-equivariant is equivalent to 
\[c(\gamma s) = \gamma c(s)\quad\forall \g \in \G\hbox{ and }s \in \S_2(\B). \]
   For this reason we shall interchangeably use the notions $\G$-equivariance and $\G$-invariance for $c\in\Char(\Gamma,V)$.
\end{rmk}

For later use we record the following result on conjugate groups:
\begin{lem}\label{lem:conjcycle}
Let $c \in \Char(V)$ be a harmonic cocycle, $\G \subset \GL_3(F)$ a subgroup and $\delta \in \GL_3(F)$ any matrix. Then 
\[c \in \Char(\G, V) \Leftrightarrow c | \delta \in \Char(\delta^{-1} \G \delta, V).\]
\end{lem}
\begin{proof}
Let $c \in \Char(\G, V)$ and $\epsilon$ an element in $\delta^{-1} \G \delta$, i.\,e. $\epsilon = \delta^{-1} \g \delta$ for a $\g \in \G$. Then 
\[ (c | \delta)|\epsilon = (c | \delta)|\delta^{-1} \g \delta = c | \g \delta = c | \delta, \]
using the $\G$-invariance of $c$, which gives the forward direction of the implication. The other direction follows from the previous one by replacing $\G$ by $\delta\G\delta^{-1}$ and $\delta$ by $\delta^{-1}$.
\end{proof}

\subsection{Harmonic cocycles for some congruence subgroups}\label{subsec:CharOnSubgps}
For the rest of this section, let $V$ denote an $F$-vector space equipped with a $\GL_3(F)$-action. For the congruence subgroups 
\[
\G(t) \subset \G_1(t) \subset
\G_0(t) \begin{array}{ccc}
\rotatebox[origin=c]{30}{$\subset$} & \G^P_0 & \rotatebox[origin=c]{-30}{$\subset$}\\ 
\rotatebox[origin=c]{-30}{$\subset$} & \G^P_2 & \rotatebox[origin=c]{30}{$\subset$}
\end{array} 
\GL_3(A).
\]
we shall give precise descriptions of the corresponding spaces of equivariant cocycles in terms of $V$. They will be deduced from the isomorphism $\Char(\Gamma_1(t),V)\to V$ given in Theorem~\ref{thm:isoG1}, and a result from \cite{gr2021}.

\subsubsection{Harmonic cocycles for the group $\G_1(t)$}
The following result was first proven in the bachelor thesis of the third author. It is based on Theorem \ref{thm:index_q} on the $\G_1(t)$-stabilizers of chambers in the standard apartment.
\begin{thm}
\label{thm:isoG1}
The map
\begin{align*}
\psi: \Char(\Gamma_1(t),V) & \rightarrow V, \\
c & \mapsto c(s_0)
\end{align*}
is an isomorphism of $F$-vector spaces. Moreover one has the following formula for $c \in \Char(\Gamma_1(t),V)$: If $s\in\S_2(\B)$ is arbitrary, with unique representative $r\in\S_2(\A)$ in its $\G_1(t)$-orbit, and if $\delta\in\G_1(t)$ is any element such that $s=\delta r$, then
\begin{equation}\label{eq:FormulaFor-cw}
    c(s) = \delta c(r) = \delta \sgn(r) \sum_{\gamma \in \Gamma_1(t)_{r}} \gamma c(s_0).
\end{equation} 
\end{thm}
\begin{rmk}
    Given any value $w \in V$, the above theorem and in particular formula~\eqref{eq:FormulaFor-cw} allows us to describe the unique harmonic cocycle $c_w\in \Char(\G_1(t), V)$ with $c_w(s_0) = w$.
\end{rmk}
\begin{proof}
The definition of the vector space structure on $\Char(\G_1(t), V)$ immediately shows that the mapping $\psi$ is linear. What we need to show is its bijectivity. We check the injectivity directly. To show surjectivity, we specify a right inverse. 

\textbf{Injectivity}:
Let $c, c' \in C^{\G_1(t)}_{har}(V)$ be given with $c(s_0) = c'(s_0)$. We want to show that $c$ and $c'$ coincide. Because $\A$ is a fundamental domain for the action of $\G_1(t)$ it suffices to show that $c$ and $c'$ agree on $\S_2(\A)$. We prove this by induction over the distance $d(s, s_0)$ for any $s\in\S_2(\A)$. Note that the base case $d(s,s_0)=0$ is trivial since we assume $c(s_0)=c'(s_0)$.

For the induction step, assume the statement is shown for all chambers $r \in \S_2(\A)$ with distance $n:=d(r, s_0) \in \N_0$, and let $s \in \S_2(\A)$ be a chamber with $d(s, s_0) = n+1$. Let $(s = x_0, x_1, \dots, x_{n+1} = s_0)$ be a gallery of length $n+1$ in $\A$, so that $r := x_1$ must have distance $d(r, s_0) = n$ to $s_0$. Then by induction hypothesis we have $c(r) = c'(r)$. Let now $e\in\S_1(\A)$ be the common edge between $r$ and $s$.

We apply Theorem \ref{thm:index_q}. Let $\gamma_1, \dots, \gamma_q$ be representatives of $\G_1(t)_s / \G_1(t)_r$. By Theorem \ref{thm:index_q} the $q+1$ chambers adjacent to $e$ are
\begin{equation}\label{eq:Edges-adj-to-e}
    \lb \gamma_1 r, \dots, \gamma_q r \rb \cup \lb s \rb 
\end{equation}
Now the harmonicity and $\G_1(t)$-invariance of $c,c'$ give the induction step:
\[c(s) = - \sum_{i=1}^q c(\gamma_i r) = - \sum_{i=1}^q \gamma_i c(r) \stackrel{\mathrm{ind.\ hyp.}}= - \sum_{i=1}^q \gamma_i c'(r) = - \sum_{i=1}^q c'(\gamma_i r) = c'(s).\]

\textbf{Surjectivity}:
Our argument is similar to the proof of \cite[][Theorem 8.5]{gr2021}. Let $v \in V$ be arbitrary. Define
\[ c_v:\S_2(\A) \rightarrow V, s \mapsto c_v(s) = \sgn(s) \sum_{\gamma \in \G_1(t)_s} \gamma v.\]
Extend $c_v$ to a function on $\S_2(\B)$ by assigning to an arbitrary simplex $s \in \S_2(\B)$ with $\delta \in \G_1(t)$ and $r \in \S_2(\A)$ so that $s = \delta r$ the value
\[ c_v(s) = \delta c_v(r) = \delta \sgn(r) \sum_{\gamma \in \G_1(t)_{r}} \gamma v.\]
The representative $r$ of $s$ in the $\G_1(t)$-orbit of $s$ is unique, but the $\delta$ is not. So we need to verify that the extension is well-defined. Suppose for this that $s=\delta r=\delta' r$ for $\delta,\delta'\in\G_1(t)$. Then $\delta^{-1}\delta'\in\G_1(t)_r$, so that $\delta^{-1}\delta'\G_1(t)_r=\G_1(t)_r$, which in turn gives
\[ \delta \sum_{\gamma \in \G_1(t)_{r}} \gamma v = \delta \sum_{\gamma \in \G_1(t)_{r}} (\delta^{-1}\delta')\gamma v = \delta (\delta^{-1}\delta')\sum_{\gamma \in \G_1(t)_{r}} \gamma v =\delta' \sum_{\gamma \in \G_1(t)_{r}} \gamma v.\] 
This shows that $c_v$ is well-defined independently of any choices.

We claim that $c_v$ is harmonic and $\G_1(t)$-invariant. Since $c_v(s_0)=v$ this will show that $v\to c_v$ is a right inverse to $\psi$, and hence that $\psi$ is surjective. Moreover the formula defining $c_v$ proves the remaining assertions of the theorem. 

The $\G_1(t)$-invariance of $c_v$ follows from its construction and the well-definedness we just proved: If for $s\in\S_2(\B)$ we have $r\in\S_2(\A)$ and $\delta\in\G_1(t)$ with $s=\delta r$, and if $\g$ is any element of $\G_1(t)$, then 
\[ c_v(\g s)=c_v(\g\delta r)\stackrel{\mathrm{Def.\ of\ }c_v}= \g\delta c_v(r) \stackrel{\mathrm{Def.\ of\ }c_v}=\g (c_v(\delta r))=\g (c_v(s).\]

It remains to show the harmonicity of $c_v$ on all edges in $\S_1(\B)$. However, by the $\G_1(t)$-equivariance and by Corollary~\ref{cor:A-FD}, it is sufficient to check the harmonicity condition for all edges in the fundamental domain $\A$. So let $e \in \S_1(\A)$ be any edge in the standard apartment. Let $r,s$ be the two chambers of $\A$ adjacent to $e$, and label them, according to Theorem~\ref{thm:index_q}, so that $r$ is closer to $s_0$ then $s$. Let the $\gamma_i$ be as above formula~\eqref{eq:Edges-adj-to-e}, so that again the expression in~\eqref{eq:Edges-adj-to-e} describes the chambers adjacent to $e$. Note that for $\g_i\in \G(t)_s\setminus \G(t)_r$ one has $\g_ir\notin \S_2(\A)$. We find
\begin{align*}
\sum_{s'\supset e}c_v(s')&= c_v(s)+\sum_{i=1,\ldots,q}c_v(\g_i r)\ = c_v(s)+\sum_{i=1,\ldots,q}\g_i c_v(r)\\
&= \sgn(s)\sum_{\g\in\G_1(t)_s}\g v+\sum_{i=1,\ldots,q}\sgn(r)\g_i \sum_{\g\in\G_1(t)_{r}}\g v\\
&= -\sgn(r)\sum_{\g\in\G_1(t)_s}\g v+\sgn( r)\sum_{i=1,\ldots,q} \sum_{\g\in\G_1(t)_{r}}\g_i\g v\\
&= -\sgn(r)\sum_{\g\in\G_1(t)_s}\g v+\sgn( r) \sum_{\g\in\G_1(t)_{s}}\g v \ =0.\\
\end{align*}
This proves harmonicity at the edge $e$.
\end{proof}
\subsubsection{Harmonic cocycles for the group $\G(t)$}
\begin{defn}
We denote the \emph{star} around the vertex $v = [0,0]$ by $\St_0 = \{ s \in \S(\B) \| v \in s \}$. The set $\St_0$ contains exactly $(q+1)(q^2 + q + 1)$ chambers, $2(q^2 + q + 1)$ edges and one vertex; as does the star of any vertex, see \cite[][Lemma 1.33]{m2014}. Note that $\St_0$ is \emph{not} a subcomplex of $\B$.
\end{defn}

\begin{defn}
We define the $F$-vector space
\[\Char(\St_0, V) := \lb c: \S_2(\St_0) \rightarrow V \,\middle|\, \forall e \in \S_1(\St_0) \sum_{s \in \S_2(\St_0), s \supset e} c(s) = 0 \rb\]
of harmonic cocycles on $\St_0$ with values in $V$. The vector space structure is induced by the one on $V$.
\end{defn}
Let $g_1, \dots , g_{n_q}$ be a system of representatives of $\GL_3(\Fq)/B$ for $n_q = \#(\GL_3(\Fq)/B) =(q+1)(q^2 + q + 1)$. The $g_i$ preserve the vertex $[0,0]$ and map $\W$ to a sector with apex~$[0,0]$.
\begin{prop}
A fundamental domain for the action of $\G(t)$ on the simplices of $\B$ is given by $\mathcal{F} = \bigcup_{i=1}^{n_q} g_i \W$. 
\end{prop}
\begin{proof}
See \cite[][Theorem 7.22]{gr2021}.
\end{proof}
\begin{prop}
The $\G(t)$-stable simplices in $\mathcal{F}$ are precisely the simplices of $\St_0$. The set of $\G(t)$-stable chambers of $\mathcal{F}$ is $\{ g_1 s_*, \dots g_{n_q} s_*\}$ where $s_*:=\{ [0,0], [0,1], [1,1] \}$ denotes the standard chamber of~$\A$.
\end{prop}

\begin{proof}
See \cite[][Proposition 8.7]{gr2021}.
\end{proof}

The following result concerns harmonic cocycles for congruence subgroups $\Gamma\subset\GL_3(A)$ that contain the normal subgroup $\G(t)$ of $\GL_3(A)$. Note that via the homomorphism theorem such $\Gamma$ correspond bijectively to subgroups $G\subset\GL_3(\Fq)$ via $\Gamma\mapsto G=\pr(\Gamma)$. 
\begin{prop} \label{prop:gamma_invariants}
Let  $G \subset \GL_3(\Fq)$ be a subgroup and $\G := \pr^{-1}(G)$.
\begin{enumerate}[(i)]
\item $\G$ acts on $\Char(\G(t),V)$ via the formula in Definition \ref{def:action}, and one has
\[ \Char(\G, V)  = \Char(\G(t), V)^{\G}. \]
\item The action in (i) factors via the quotient $G$ and one has 
\[ \Char(\G(t), V)^{\G}= \Char(\G(t), V)^{G}. \]
\item $\GL_3(\Fq)$ acts on $\Char(\St_0, V)$ by the formula from Definition \ref{def:action}.
\end{enumerate}
\end{prop}
\begin{proof}
\begin{enumerate}[(i)]
\item Noting that $\G(t)\subset\G$ is a normal subgroup, part (i) is a standard fact from representation theory, since $\Char(\G',V):=\Char(V)^{\G'}$ for any subgroup $\G'$ of $\GL_3(F)$. 
\item Since by definition $\G(t)$ acts trivially on $\Char(\G(t),V)$, the first part is clear, and the second is an immediate consequence.
\item By Proposition \ref{prop:stab_a}, the group $\GL_3(\Fq)$ is the $\GL_3(A)$-sta\-bilizer of the vertex $[0,0]$. Hence its action preserves the simplices of~$\St_0$, and (iii) follows.
\end{enumerate}
\end{proof}

\begin{thm} \label{thm:isogt}
Under the actions from Proposition~\ref{prop:gamma_invariants}, the map
\begin{align*}
\phi: \Char(\G(t),V) & \rightarrow \Char(\St_0, V), \\
c & \mapsto c|_{\St_0}
\end{align*}
is an isomorphism of $F[\GL_3(\Fq)]$-modules. One has $\dim_F \Char(\G(t),V)=q^3 \dim_F V$. \end{thm}
\begin{proof}
The $F$-linearity is clear. The $\GL_3(\Fq)$-equivariance also follows immediately, since both actions are induced by the one in Definition \ref{def:action}.

In \cite[][Theorem 8.5]{gr2021} it was proved that any element in $\Char(\St_0, V)$ can be extended to a $\G(t)$-equivariant cocycle on all of $\B$, so that $\phi$ is surjective. The dimension formula for $\Char(\St_0, V)$ also follows from the proof of \cite[][Theorem 8.5]{gr2021}. We shall prove that $\dim_F \Char(\G(t),V) \leq q^3 \dim_F V = \dim_F \Char(\St_0, V)$, which together with the above concludes the proof.

We abbreviate $X := \Char(\G(t),V)$ and let $d := \dim_F V$. Recall that $U = U(\Fq) = \G_1(t) / \G(t)$ is a finite group of order $q^3$, and in particular a $p$-group. According to Proposition \ref{prop:gamma_invariants} and Theorem \ref{thm:isoG1}, $X$ is an $F[U]$-module and $X^U = \Char(\G_1(t),V) \cong V$, so that $\dim_F X^U = d$.

Now observe that $F[U]^U$ is the $F$-span of $\sum_{g \in U} g \in F[U]$, which is well-known and straightforward to be checked; this holds for any group. In particular $\dim_F F[U]=1$. We choose an $F$-linear isomorphism $X^U\to (F[U]^U)^{d}$ and extend it, using the inclusion $F[U]^U\to F[U]$, to an $F[U]$-module homomorphism
\[ h: X^U \rightarrow F[U]^d\]
The free $F[U]$-module $F[U]^d$ is projective over $F[U]$ and hence, by \cite[][Theorem 62.3]{cur.rei1966}, also injective. Therefore $h$ extends to an $F[U]$-homomorphism
\[ H: X \rightarrow F[U]^d.\]
We claim that $H$ is injective, which implies that $\dim_FX \le q^3\dim_F V$. Then we must have $\dim_FX =q^3\dim_F V$, and this completes the proof.

For the claim, assume on the contrary that $\ker(H) \neq 0$. Then \cite[][Proposition 26]{ser2012} yields $\ker(H)^U \neq 0$. But $\ker(H)^U \subset X^U$, and $h$ is injective; we reach a contradiction.
\end{proof}

\begin{cor}\label{cor:isogt}
Let the notation be as in Proposition~\ref{prop:gamma_invariants}. Then $\phi$ from Theorem~\ref{thm:isogt} induces an isomorphism 
\[ \Char(\G,V) \rightarrow \Char(\St_0, V)^G.\]
\end{cor}

We shall also need the following variant of Theorem~\ref{thm:isogt}.
\begin{prop}\label{prop:isogt}
    Consider $U\cdot s_0$ as a subset of the chambers of $\St_0$. Then restriction defines a $U$-equivariant isomorphism
    \[ \beta:\Char(\St_0, V)\to \Maps(U\cdot s_0,V).\]
\end{prop}
\begin{proof}
    As shown in the proof of Theorem~\ref{thm:isogt}, we have $\dim_F\Char(\St_0, V)=\dim_FV\cdot\#U$. Since $U$ acts faithfully on $s_0$, we also have $\dim_F \Maps(U\cdot s_0,V)=\dim_FV\cdot\#U$. Therefore it suffices to show that $\beta$ is injective. Taking $U$-invariants gives the left exact sequence
    \[ 0\to (\ker \beta)^U\to \Char(\St_0, V)^U\stackrel{\beta^U}\to \Maps(U\cdot s_0,V)^U.\]
    Now $\beta^U$ has target $\Maps(\{s_0\},V)$, by Corollary~\ref{cor:isogt} the domain can be identified with $\Char(\G_1(t),V)$, and the map $\beta^U$ itself with the restriction isomorphism from Theorem~\ref{thm:isoG1}. It follows that $(\ker \beta)^U=0$. Now as in the proof of Theorem~\ref{thm:isogt}, we deduce $\ker\beta=0$, and this concludes the proof.
\end{proof}
We obtain the following consequence:
\begin{cor}\label{cor:Isom-ResUi}
    Suppose that $U$ is the direct product of subgroups $H_1$ and $H_2$, and let $\G_1=\pr^{-1}(H_1)$. Then restriction on chambers defines a $H_2$-equivariant isomorphism
    \[ \Char(\G_1,V) \rightarrow \Maps(H_2\cdot s_0,V),\]
    and in particular, any $c\in \Char(\G_1,V)$ is uniquely determined by its values on $H_2 \cdot s_0$.
\end{cor}

\subsubsection{Harmonic cocycles for the group $\G_0(t)$}
We want to find the image of $\Char(\Gamma_0(t),V)$ under the isomorphism of Theorem \ref{thm:isoG1}.

Our reasoning relies on the following Lemma.
\begin{lem} \label{lem:G0normal}
The map 
\begin{align*}
\G_0(t) \rightarrow D, \left(\begin{smallmatrix}
a & b & c \\ 
d & e & f \\ 
g & h & i
\end{smallmatrix}\right) \mapsto \left(\begin{smallmatrix}
\bar{a} & 0 & 0 \\ 
0 & \bar{e} & 0 \\ 
0 & 0 & \bar{i}
\end{smallmatrix}\right) ,
\end{align*}
where $\bar{a}, \bar{e}, \bar{i}$ are reduced modulo $t$, is a surjective group homomorphism with kernel $\G_1(t)$. 

Hence, $\G_1(t)$ is normal in $\G_0(t)$ and every $\g \in \G_0(t)$ can be written as $\g = \upsilon \delta$ with $\upsilon \in \G_1(t), \delta \in D$.
\end{lem}
\begin{proof}
Straightforward calculation.
\end{proof}

\begin{thm}\label{thm:IsomG0}
$\Char(\Gamma_0(t),V) \cong V^{D}.$
\end{thm}
\begin{proof}
Let $c \in \Char(\Gamma_0(t),V)$. Then, since $\G_1(t) \subset \G_0(t)$, c is $\G_1(t)$-equivariant as well and corresponds to an element $w = c(s_0) \in V$.
Then for any matrix $\delta \in D = \Stab_{\G_0(t)}(s_0)$ (follows from Proposition \ref{prop:stab_a}), we have 
\[ \delta w = \delta c(s_0) = c(\delta s_0) = c(s_0) = w, \]
so $w$ lies in $V^{D(\Fq)}$.

For the reverse inclusion, let $w \in V^{D}$ arbitrary. We denote by $c_w$ the $\G_1(t)$-equivariant cocycle with $c_w(s_0) = w$. We need to show that $c_w$ is already $\G_0(t)$-equivariant, i.\,e. for all $\g \in \G_0(T)$ we want $c_w | \g = c_w$. 

Now note that according to Lemma \ref{lem:conjcycle}, $c_w | \g$ is an element of $\Char(\g^{-1} \G_1(t) \g, V)$. But $\g$ lies in $\G_0(t)$ and $\G_1(t)$ is normal in $\G_0(t)$, so $\g^{-1} \G_1(t) \g = \G_1(t)$ and thus, $c_w | \g$ is $\G_1(t)$-equivariant. Hence, it suffices to compare $c_w | \g$ and $c_w$ on the stable simplex $s_0$ to see if they are equal. We write $\g = \upsilon \delta$ with $\upsilon \in \G_1(t), \delta \in D$ and get 
\begin{align*}
(c_w | \g) (s_0) & = \g^{-1} c_w(\g s_0) = \delta^{-1} \upsilon^{-1} c_w(\upsilon \delta s_0) \\ 
    & = \delta^{-1} \upsilon^{-1} \upsilon c_w( \delta s_0) = \delta^{-1} c_w(s_0) = \delta^{-1} w = w = c_w(s_0),
\end{align*}
which completes the proof.
\end{proof}

\subsubsection{Harmonic cocycles for the group $\GL_3(A)$}
The approach is similar to the one for $\G_1(t)$, but we also need the results on $\G(t)$-equivariant cocycles and the resulting conditions are slightly more complicated.

\begin{lem} \label{lem:gl3-Bt}
Let $c \in \Char(\GL_3(A), V)$. Then $w = c(s_0) \in V^{B^T}$.
\end{lem}
\begin{proof}
Let $c \in \Char(\GL_3(A), V)$. Then, since $\G_1(t) \subset \GL_3(A)$, c is $\G_1(t)$-equivariant as well and corresponds to an element $w = c(s_0) \in V$.
Then for any matrix $\delta \in B^T = \Stab_{\GL_3(A)}(s_0)$, as follows from Proposition \ref{prop:stab_a}, we have 
\[ \delta w = \delta c(s_0) = c(\delta s_0) = c(s_0) = w, \]
so $w$ lies in $V^{B^T}$.
\end{proof}

\begin{lem}\label{lem:s0St0}
Let $c \in \Char(\G_1(t), V)$ and assume that for all $\gamma \in \GL_3(\Fq)$, that we have $(c | \g)(s_0) = c(s_0)$. Then we get $c | \g = c$ on $\St_0$ for all $\gamma \in \GL_3(\Fq)$.
\end{lem}
\begin{proof}
Let $\gamma, \delta \in \GL_3(\Fq)$. Then under the assumption above
\[ \g c(\delta s_0) = \g \delta c(s_0) = c(\g \delta s_0). \]
But $\GL_3(\Fq) s_0 = \S_2(\St_0)$. Hence, the line above means that for all $\g \in \GL_3(\Fq), s \in \S_2(\St_0)$,
\[ \g c(s) = c(\g s). \]
\end{proof}
\begin{lem}\label{lem:s0E0}
Let $c \in \Char(\G_1(t), V)$ and assume that for all $\gamma \in P_0$, that we have $(c | \g)(s_0) = c(s_0)$. Then we get $c | \g = c$ for all $\gamma \in P_0$.
\end{lem}
\begin{proof}
As in the previous proof, we deduce that $\g c(s)=c(\g s)$ for all $\g\in P_0$ and $s\in P_0 s_0$. This is equivalent to $c | \g=c$ as elements in $\Maps(P_0 \cdot s_0,V)$ and since $U\subset P_0$ also as elements in $\Maps(U\cdot s_0, V)$. Fix $\gamma\in P_0$. Since $\G(t)\subset\GL_3(A)$ is normal, we can regard both $c$ and $c|\gamma$ as elements of $\Char(\G(t),V)$. Now Theorem \ref{thm:isogt} and Proposition \ref{prop:isogt} imply that $c|\g =c$.
\end{proof}

\begin{thm} \label{thm:IdentifyGL3}The map $c\mapsto c(s_0)$ defines an isomorphism
\[\Char(\GL_3(A),V) \cong \lb v \in V^{B^T} \,\middle|\,  \forall \rho \in W: \, \rho v = \sgn(\rho s_0) \sum_{\g \in \G_1(t)_{(\rho s_0)}} \g v \rb. \]
\end{thm}
Note that from Corollary \ref{cor:stab-a} it follows that $\G_1(t)_{(\rho s_0)}\subset \GL_3(\Fq)$.

\begin{proof}
Let $w \in V^{B^T}$ and $c_w \in \Char(\G_1(t), V)$ be the corresponding cocycle with $c_w(s_0) = w$. We want to describe equivalent conditions to $c_w$ being $\GL_3(A)$-invariant. According to Proposition \ref{prop:gamma_invariants}, this is equivalent to $c_w$ being $\GL_3(\Fq)$-invariant on $\St_0$ and according to Lemma \ref{lem:s0St0}, it even suffices to check the invariance at the stable simplex $s_0$.

For this purpose, let $\g \in \GL_3(\Fq)$. We write $\g = \alpha \rho \beta$ with $\alpha \in U, \rho \in W$ and $\beta \in B^T$ according to Proposition \ref{prop:bruhat_decomp}. Then
\begin{eqnarray*}
\g c_w(s_0) = c_w(\g s_0) & \Leftrightarrow&
\alpha \rho \beta w = c_w(\alpha \rho \beta s_0)  \ \Leftrightarrow \ \alpha \rho w = \alpha c_w(\rho s_0)\\
&\Leftrightarrow & \rho w = c_w(\rho s_0) \ \Leftrightarrow \ \rho w = \sgn(\rho s_0) \sum_{\g \in \G_1(t)_{(\rho s_0)}} \g c_w(s_0) \\
&\Leftrightarrow & \rho w = \sgn(\rho s_0) \sum_{\g \in \G_1(t)_{(\rho s_0)}} \g w.
\end{eqnarray*}
The first line holds for all $\g \in \GL_3(\Fq)$ if and only if the last line holds for all $\rho \in W$.
\end{proof}
\begin{rmk}
In the above description, we get conditions for all six permutation matrices $\rho \in W$. One of them is the identity matrix, and this condition is trivially satisfied. Hence, five conditions actually remain.
\end{rmk}

\subsubsection{Harmonic cocycles for the group $\G^P_0$}
Let $H_0 := (\G^P_0)_{s_0} = \left(\begin{smallmatrix}
    *&0&0\\ *&*&0\\0&0&*
\end{smallmatrix}\right) \subset \GL_3(\Fq)$ denote the $\G^P_0$-stabilizer of the stable simplex $s_0$. By $\sigma := \left(\begin{smallmatrix}
    0&1&0\\1&0&0\\0&0&1
\end{smallmatrix}\right)$ we denote the only non-trivial permutation matrix that is contained in $\G^P_0$. Then we get the following description:
\begin{thm}\label{thm:Isog-P0} The map $c\mapsto c(s_0)$ defines an isomorphism
    \[\Char(\G^P_0,V) \cong \lb v \in V^{H_0} \,\middle|\,  \sigma v = - \sum_{a \in \Fq} \left(\begin{smallmatrix}
        1&a&0\\0&1&0\\0&0&1
    \end{smallmatrix}\right) v \rb. \]
\end{thm}
\begin{proof}
The obvious analog of Lemma~\ref{lem:gl3-Bt} for $\G_0^P$ implies $c(s_0)\in V^{H_0}$. The bijectivity of the map is now proved in the same way as Theorem~\ref{thm:IdentifyGL3}, however using Lemma~\ref{lem:s0E0} instead of Lemma~\ref{lem:s0St0}.
\end{proof}

\subsubsection{Harmonic cocycles for the group $\G^P_2$}
Let $H_2 := (\G^P_2)_{s_0} = \left(\begin{smallmatrix}
    *&0&0\\ 0&*&0\\0&*&*
\end{smallmatrix}\right) \subset \GL_3(\Fq)$ denote the $\G^P_2$-stabilizer of the stable simplex $s_0$. By $\tau := \left(\begin{smallmatrix}
    1&0&0\\0&0&1\\0&1&0
\end{smallmatrix}\right)$ we denote the only non-trivial permutation matrix that is contained in $\G^P_2$. Then we get the following description:
\begin{thm}\label{thm:Isog-P2}
    \[\Char(\G^P_2,V) \cong \lb v \in V^{H_2} \,\middle|\,  \tau v = - \sum_{a \in \Fq} \left(\begin{smallmatrix}
        1&0&0\\0&1&a\\0&0&1
    \end{smallmatrix}\right) v \rb. \]
\end{thm}
\begin{proof}
The argument is analogous to the $\G^P_0$-case. We omit details.
\end{proof}

\section{Hecke Operators}
\label{sec:heckeop}
In this section we introduce Hecke-Operators on spaces of $\Gamma$-invariant harmonic cocycles for congruence subgroups $\Gamma$. A special case of them are the so-called $T_i$- and $U_i$-operators that we define here. We use the results from the previous Section \ref{sec:cocycles} to gain explicit formulas for them, that will allow us to calculate transformation matrices and investigate them computationally in the following Section \ref{sec:magma}.
\subsection{Double coset operators}\label{sec:doublecosetop}
Throughout this subsection, we let $\G, \G' \subset \GL_3(A)$ be congruence subgroups and $\delta \in \GL_3(F)$ any element. Note that if $\G'\subset\G$ then the set $\G'\backslash\G$ of left cosets is finite, because it can be identified with the left coset of two subgroups of the finite group $\GL_3(A/N)$ for a suitable non-zero ideal $N$ of~$A$.

To following basic result can easily be obtained by adapting \cite[][Lemma 5.1.2]{dia.shu2005} or \cite[][Proof of Proposition 3.1]{shi1994} to the present situation.
\begin{lem} \label{lem:reps}
The intersection $\delta^{-1} \G' \delta \cap \G$ is a congruence subgroup, the set $\G'\backslash\G'\delta\Gamma$ is finite and one has a bijection of right $\G$-sets
\begin{align*}
(\delta^{-1} \G' \delta \cap \G) \backslash \G & \rightarrow \G' \backslash \G' \delta \G,  \\* 
(\delta^{-1} \G' \delta \cap \G)\, \epsilon & \mapsto \G' \, \delta \, \epsilon.
\end{align*}
\end{lem}
\begin{lemdefn} \label{def:doublecosetop}
We define the \emph{double coset operator} \[ T_{\G' \delta \G} : \Char(\G', V) \rightarrow \Char(\G, V)\] by the formula
\[ 
c \mapsto \sum_{\xi\in \G' \backslash \G' \delta \G} c | \xi \ \ \stackrel{\ref{lem:reps}}= \sum_{\epsilon \in (\delta^{-1} \G' \delta \cap \G) \backslash \G} c | (\delta \epsilon), \]
where the sums are over respective sets of representatives. For  $s \in \S_2(\B)$ we have
\[ (T_{\G' \delta \G} c)(s) = \sum_{\epsilon} \epsilon^{-1} c(\epsilon s). \]
The operator $T_{\G' \delta \G}$ has the following properties.
\begin{enumerate}
    \item It is independent of the chosen set of representatives. 
    \item It maps $\Char(\G', V)$ to $\Char(\G, V)$ as indicated.
    \item It is $F$-linear.
\end{enumerate}
\end{lemdefn}

\begin{proof}
The $F$-linearity is clear. The independence of the chosen set of representatives follows from the $\G'$-invariance of harmonic cocycles $c$ in the domain $\Char(\G',V)$; it implies that for $\alpha\in \G'$ and $\xi\in\GL_3(F)$ we have $c|(\alpha\xi)=(c|\alpha)|\xi=c|\xi$. 

That harmonicity is preserved follows from Lemma~\ref{lem:harmonicity-under-G}. It remains to see that the resulting cocycle is $\G$-invariant. For this let $\g \in \G$ and $c \in\Char(\G', V)$. Then
\[ (T_{\G' \delta \G} c)| \g = \left(\sum_{\epsilon} c | \epsilon \right)|\g = \sum_{\epsilon} c | (\epsilon \g) = T_{\G' \delta \G} c, \]
since $\G' \backslash \G' \delta \G$ is a right $\G$-set, so that $\{\epsilon \gamma\}$ is again a set of representatives of $\G' \backslash \G' \delta \G$.
\end{proof}

\begin{rmk} \label{rmk:doublecosetop}
Some special cases of Lemma-Definition \ref{def:doublecosetop} are the following:
\begin{description}
\item[$\G' \supseteq \G$ and $\delta = Id$.] In this case, $T_{\G' \delta \G}$ is the inclusion of the subspace $\Char(\G', V)$ in $\Char(\G, V)$, mapping $c$ to $c$.
\item[$\G' \subset \G$ and $\delta = Id$.] In this case, $T_{\G' \delta \G}$ is the so called \emph{trace map}, projecting $\Char(\G', V)$ onto its subspace $\Char(\G, V)$.
\item[$\G = \delta^{-1} \G' \delta$.] In this case, $T_{\G' \delta \G}$ maps $c$ to $c | \delta$ and is an isomorphism by Lemma \ref{lem:conjcycle}.
\item[$\G' = \G$.] In this case, $T_{\G \delta \G} =: T_{\delta}$ is a so called \emph{Hecke operator}. These are the main objects of interest in this section.
\end{description}
\end{rmk}

\subsection{Explicit formulas}\label{subsec:ExplFormulas}
We shall be interested in certain Hecke operators on the spaces of $\G$-invariant harmonic cocycles we investigated in the previous section. The matrices $\delta$ of particular interest to us are
\begin{equation}\label{eq:def-delta}
 \delta_1 = \left(\begin{smallmatrix}
	1 & 0 & 0 \\ 
	0 & 1 & 0 \\ 
	0 & 0 & t
\end{smallmatrix}\right), \delta_2 = \left(\begin{smallmatrix}
	1 & 0 & 0 \\ 
	0 & t & 0 \\ 
	0 & 0 & t
\end{smallmatrix}\right) \in \GL_3(F).
\end{equation}
We fix the notation
	\begin{align}
	\label{eq:T_i-def}T_i & = T_{\delta_i} : \Char(\GL_3(A),V) \rightarrow \Char(\GL_3(A),V) \text{ and} \\
	\label{eq:U_i-def}U_i^{\G} & = T_{\delta_i} : \Char(\G,V) \rightarrow \Char(\G,V) \text{ for } \G \in \{ \Gamma_1(t), \Gamma_0(t), \G^P_0, \G^P_2\}.
	\end{align}
    
Our aim is to derive explicit formulas for the above operators ($U_i^\G$ and $T_i$) on \[V\cong\Char(\G_1(t),V),\] under the identification from Theorem~\ref{thm:isoG1}, and on certain subspaces of $V$ given in Theorems~\ref{thm:IdentifyGL3}, \ref{thm:Isog-P0}, \ref{thm:Isog-P2}, \ref{thm:IsomG0}.

For this, we need to express the value of $(T_{\delta_i} c)(s_0) = \sum_{\epsilon \in \G \backslash \G \delta_i \G} \epsilon^{-1} c(\epsilon s_0) \in V$ in terms of $c(s_0) \in V$ for arbitrary $c \in \Char(\G,V)$. The steps for each $\delta_i$ and each $\Gamma$ include: 
\begin{enumerate}
\item Find a set of representatives $\{ \epsilon \}$ for $\G \backslash \G \delta_i \G$.
\item For each representative $\epsilon$ from 1, find the unique chamber $s_\epsilon$ in $\A$ and a $\g \in\G_1(t)$ such $\g \epsilon s_0=s_\epsilon$, as guaranteed by Corollary~\ref{cor:A-FD}. 
\item Express $c(s_\epsilon)$ in terms of~$c(s_0)$ using the $\G$-invariance and the formula in Theorem~\ref{thm:isoG1}.
\item Join all summands and reorder sums if necessary.
\end{enumerate}	
It will turn out that for all the operators we consider, the resulting matrix is the product of a diagonal matrix with powers of $t$ on the diagonal and a matrix with entries in $\Fq$.

We now give explicit formulas for the operators $T_i$ and $U_i^{\G}$ by choosing sets of representatives. A detailed explanation on how these can be calculated is given in the Appendix \ref{subsec:rep}. Here, we only give the result.
\begin{prop}\label{prop:repr-all-op}
We define the following sets
\[
\begin{array}{ccc}
Q_1 := \lb \left(\begin{smallmatrix}
1 & 0 & a \\ 
0 & 1 & b \\ 
0 & 0 & t
\end{smallmatrix}\right) \,\middle|\, a, b \in \Fq \rb, & R_1 := \lb \left(\begin{smallmatrix}
1 & a & 0 \\ 
0 & 0 & 1 \\ 
0 & t & 0
\end{smallmatrix}\right) \,\middle|\, a \in \Fq \rb, & S_1 := \lb \left(\begin{smallmatrix}
0 & 1 & 0 \\ 
0 & 0 & 1 \\ 
t & 0 & 0
\end{smallmatrix}\right) \rb,\\[0.5em]
Q_2 := \lb \left(\begin{smallmatrix}
1 & a & b \\ 
0 & t & 0 \\ 
0 & 0 & t
\end{smallmatrix}\right) \,\middle|\, a, b \in \Fq \rb, & R_2 := \lb \left(\begin{smallmatrix}
0 & 1 & a \\ 
t & 0 & 0 \\ 
0 & 0 & t
\end{smallmatrix}\right) \,\middle|\, a \in \Fq \rb, & S_2 := \lb  \left(\begin{smallmatrix}
0 & 0 & 1 \\ 
t & 0 & 0 \\ 
0 & t & 0 
\end{smallmatrix}\right) \rb. \\
\end{array}
\]
Then we can describe sets of representatives as follows:
\begin{align*}
\G_1(t) \backslash \G_1(t) \delta_1 \G_1(t) = \G_0(t) \backslash \G_0(t) \delta_1 \G_0(t) = \G^P_0 \backslash \G^P_0 \delta_1 \G^P_0 & = Q_1, \\
\G_1(t) \backslash \G_1(t) \delta_2 \G_1(t) = \G_0(t) \backslash \G_0(t) \delta_2 \G_0(t) = \G^P_2 \backslash \G^P_2 \delta_2 \G^P_2 & = Q_2, \\
\G^P_2 \backslash \G^P_2 \delta_1 \G^P_2 & = Q_1 \cup R_1, \\
\G^P_0 \backslash \G^P_0 \delta_2 \G^P_0 & = Q_2 \cup R_2, \\
\GL_3(A) \backslash \GL_3(A) \delta_1 \GL_3(A) & = Q_1 \cup R_1 \cup S_1, \\
\GL_3(A) \backslash \GL_3(A) \delta_2 \GL_3(A) & = Q_2 \cup R_2 \cup S_2. \\
\end{align*}
\end{prop}
\begin{proof}
See Appendix \ref{app:rep-hecke}.
\end{proof}
In Appendix~\ref{appendix}, we execute the steps indicated above. To describe the main result, we shall need the linear operators $A_i, B_i, C_i : V \rightarrow V$, $i=1,2$, defined as follows.
\setlength{\arraycolsep}{2.2 pt} 
\small 
\begin{align*} 
A_1 (v) := & \sum_{a, b \in \Fq^{\times}} \Bigg[ \begin{pmatrix}
 1 & 0 & 0 \\
 0 & 1 & 0 \\
 a & b & 1 \\
\end{pmatrix} \hspace{-.65pt}v - \begin{pmatrix}
 0 & 0 & -a \\
 0 & 1 & 0 \\
 a^{-1} & a^{-1}b & 1 \end{pmatrix} \hspace{-.65pt}v - \begin{pmatrix}
 1 & 0 & 0 \\
 -ab & 0 & -a \\
 b & a^{-1} & 1 \\
\end{pmatrix} \hspace{-.65pt}v + \begin{pmatrix}
 0 			& 0 	 & -a \\
 -a^{-1}b   & 0 	 & -b \\
 a^{-1} 	& b^{-1} & 1 \\
\end{pmatrix} \hspace{-.65pt}v \Bigg ] \\
+ & \sum_{a \in \Fq^{\times}} \Bigg[ \begin{pmatrix}
 1 & 0 & 0 \\
 0 & 1 & 0 \\
 a & 0 & 1 \\
\end{pmatrix} \hspace{-.65pt}v + \begin{pmatrix}
 1 & 0 & 0 \\
 0 & 1 & 0 \\
 0 & a & 1 \\ \end{pmatrix} \hspace{-.65pt}v - 
\begin{pmatrix}
 0 & 0 & -a \\
 0 & 1 & 0 \\
 a^{-1} & 0 & 1 \\
\end{pmatrix} \hspace{-.65pt}v - \begin{pmatrix}
 1	& 0 	 & 0 \\
 0  & 0 	 & -a \\
 0 	& a^{-1} & 1 \\
\end{pmatrix} \hspace{-.65pt}v \Bigg] \ + \ v 
\\[1em]
B_1(v) := & \sum_{a \in \Fq^{\times}} \Bigg[ \begin{pmatrix}
    0&0&-a\\
    a^{-1}&0&1\\
    0&1&0
\end{pmatrix} \hspace{-.65pt}v - \begin{pmatrix}
    1&0&0\\
    a&0&1\\
    0&1&0
\end{pmatrix} \hspace{-.65pt}v \Bigg] - \begin{pmatrix}
    1&0&0\\
    0&0&1\\
    0&1&0
\end{pmatrix} \hspace{-.65pt}v,\\[1em]
C_1(v) := & - \begin{pmatrix}
    0&0&1\\
    1&0&0\\
    0&1&0
\end{pmatrix} \hspace{-.65pt}v
\intertext{and}
A_2 (v) := & \sum_{a, b \in \Fq^{\times}} \Bigg[ \begin{pmatrix}
 1 & 0 & 0 \\
 a & 1 & 0 \\
 b & 0 & 1 \\
\end{pmatrix} \hspace{-.65pt}v - \begin{pmatrix}
 0 & -a & 0 \\
 a^{-1} & 1 & 0 \\
 b & 0 & 1 \end{pmatrix} \hspace{-.65pt}v - \begin{pmatrix}
 0 & 0 & -a \\
 0 & 1 & b \\
 a^{-1} & 0 & 1 \\
\end{pmatrix} \hspace{-.65pt}v + \begin{pmatrix}
 0 		& -a  & 0 \\
 0  	& 1	 & -a^{-1}b \\
 b^{-1} & 0 	 & 1 \\
\end{pmatrix} \hspace{-.65pt}v \Bigg] \\
+ & \sum_{a \in \Fq^{\times}} \Bigg[ \begin{pmatrix}
 1 & 0 & 0 \\
 a & 1 & 0 \\
 0 & 0 & 1 \\
\end{pmatrix} \hspace{-.65pt}v + \begin{pmatrix}
 1 & 0 & 0 \\
 0 & 1 & 0 \\
 a & 0 & 1 \\ \end{pmatrix} \hspace{-.65pt}v - 
\begin{pmatrix}
 0 & -a & 0 \\
 a^{-1} & 1 & 0 \\
 0 & 0 & 1 \\
\end{pmatrix} \hspace{-.65pt}v - \begin{pmatrix}
 0	& 0  & -a \\
 0  & 1 & 0 \\
 a^{-1} 	& 0 & 1 \\
\end{pmatrix} \hspace{-.65pt}v \Bigg] \ + \ v 
\\[1em]
B_2(v) := & \sum_{a \in \Fq^{\times}} \Bigg[ \begin{pmatrix}
    0&1&0\\
    0&0&-a\\
    a^{-1}&0&1
\end{pmatrix} \hspace{-.65pt}v - \begin{pmatrix}
    0&1&0\\
    1&0&0\\
    a&0&1
\end{pmatrix} \hspace{-.65pt}v \Bigg]- \begin{pmatrix}
    0&1&0\\
    1&0&0\\
    0&0&1
\end{pmatrix} \hspace{-.65pt}v,\\[1em]
C_2(v) := & - \begin{pmatrix}
    0&1&0\\
    0&0&1\\
    1&0&0
\end{pmatrix} \hspace{-.65pt}v.
\end{align*}
\normalsize
\setlength{\arraycolsep}{5 pt} 

\begin{thm}
\label{thm:AiBiCi}
The following formulas hold for the Hecke operators introduced above: For $i =1, \G \in \{\G_1(t), \G_0(t), \G^P_0\}$ or $i=2, \G \in \{\G_1(t), \G_0(t), \G^P_2\}$ and $c \in \Char(\G, V)$ we have 
\begin{align*}
(U_i^{\G} c)(s_0) & = A_i (\delta_i^{-1} c(s_0)).
\intertext{For $i =1, \G = \G^P_2$ or $i=2, \G = \G^P_0$ and $c \in \Char(\G, V)$ we have}
(U_i^{\G} c)(s_0) & = (A_i + B_i) (\delta_i^{-1} c(s_0)).
\intertext{For $i \in \{1, 2\}$ and $c \in \Char(\GL_3(A), V)$ we have}
(T_i c)(s_0) & = (A_i + B_i + C_i) (\delta_i^{-1} c(s_0)).
\end{align*}
\end{thm}
\begin{proof}
The result follows from evaluating the expressions describing the operators when inserting the representatives described in Proposition \ref{prop:repr-all-op}.

For all matrices $\epsilon$ in the sets $Q_1, Q_2, R_1, R_2, S_1$, and $S_2$ we find a chamber $s_{\epsilon}$ in the standard apartment and a matrix $\g_{\epsilon} \in \G_1(t)$ such that $\epsilon s_0 = \g_{\epsilon} s_{\epsilon}$, allowing us to use the formula from Theorem \ref{thm:isoG1}. An example is detailed in the appendix \ref{subsec:finding-simplices}. 

Afterwards, we evaluate the formula. Since the calculation is long, but not difficult, we only describe how the representatives in the set $R_1$ lead to the operator $B_1$ as an example. The rest can be shown in an analogous manner.

Representatives in $R_1$ will actually only come into play for $\G^P_2$- or $\GL_3(A)$-equivariant cocycles, but for the calculation that follows $\G_1(t)$-equivariance is enough. So let $c \in \Char(\G_1(t), V)$. We evaluate
\setlength{\arraycolsep}{3.5 pt} 
\begin{align*}
\lefteqn{\sum_{\epsilon \in R_1} (c | \epsilon)(s_0) = \sum_{\epsilon \in R_1} \epsilon^{-1} c(\epsilon s_0) = \sum_{a \in \Fq} \begin{pmatrix}
        1 & 0 & -a \pi \\
        0 & 0 & \pi \\
        0 & 1 & 0
    \end{pmatrix} c \left( \begin{pmatrix}
        1 & a & 0 \\
        0 & 0 & 1 \\
        0 & t & 0
    \end{pmatrix} s_0 \right),} \\
\intertext{now we can apply the results of Appendix \ref{subsec:R1-repr-a} and get } 
 = & \sum_{a \in \Fq^{\times}} \begin{pmatrix}
        1 & 0 & -a \pi \\
        0 & 0 & \pi \\
        0 & 1 & 0
    \end{pmatrix} c \left( \begin{pmatrix}
        1 & 0 & 0 \\
        0 & 1 & 0 \\
        a^{-1} t & 0 & 1
    \end{pmatrix} s_0 \right) + \begin{pmatrix}
        1 & 0 & 0 \\
        0 & 0 & \pi \\
        0 & 1 & 0
    \end{pmatrix} c \left( s \right).
\intertext{ with a simplex $s = \{ [-1, -1], [0, -1], [-1, -2] \}$. We calculate the $\G_1(t)$-stabilizer of the simplex $s$ using Corollary \ref{cor:stab-a} and apply Theorem \ref{thm:isoG1}. Then we obtain}
= & \sum_{a \in \Fq^{\times}} \begin{pmatrix}
        1 & 0 & -a \pi \\
        0 & 0 & \pi \\
        0 & 1 & 0
    \end{pmatrix} \begin{pmatrix}
        1 & 0 & 0 \\
        0 & 1 & 0 \\
        a^{-1} t & 0 & 1
    \end{pmatrix} c(s_0) + \begin{pmatrix}
        1 & 0 & 0 \\
        0 & 0 & \pi \\
        0 & 1 & 0
    \end{pmatrix} \left( - \sum_{a \in \Fq} \begin{pmatrix}
        1 & 0 & 0 \\
        0 & 1 & 0 \\
        at & 0 & 1
    \end{pmatrix} c (s_0) \right) \\
= & \sum_{a \in \Fq^{\times}} \Bigg[ \begin{pmatrix}
    0&0&-a \pi\\
    a^{-1}&0&\pi\\
    0&1&0
\end{pmatrix} c(s_0) - \begin{pmatrix}
    1&0&0\\
    a&0&\pi\\
    0&1&0
\end{pmatrix} c(s_0) \Bigg] - \begin{pmatrix}
    1&0&0\\
    0&0&\pi\\
    0&1&0
\end{pmatrix} c(s_0) \\
= & B_1( \delta_1^{-1} c(s_0) ). \qedhere 
\end{align*}
\setlength{\arraycolsep}{5 pt}
\end{proof}
We have simplified the expressions for the Hecke-operators as the action of two diagonal matrices $\delta_i^{-1}$ concatenated with operators on $V$ that can be described by matrices with entries in the finite field $\Fq$. 

We can only transform the expressions further if we know explicitly what the action of $\GL_3(F)$ on $V$ looks like. This will be described in the following Section \ref{sec:magma}.

\section{Calculations in Magma}
\label{sec:magma}
We want to use the result in Theorem \ref{thm:AiBiCi} to calculate and investigate transformation matrices for our Hecke operators for the coefficients $V_{k,n}$ mentioned in the introduction. Instead of proceeding with calculations manually, we will employ the help of the computer algebra system \texttt{Magma} \cite{bos.can.pla1997}. This will allow us to systematically produce a number of examples.

Specifically, we will calculate the \emph{slopes} of our operators, i.\,e. the $t$-adic valuations of their eigenvalues. They can be determined by constructing the Newton-polygon from the coefficients of their characteristic polynomial \cite[][Section~II.6]{neukirch1999}, for which built-in \texttt{Magma}-functions exist.

The motivation for our coefficient vector space comes from the residue homomorphism from \cite{gr2021} that connects spaces $S_{k, n}(\G)$ of Drinfeld cusp forms of level $\G$ and rank $3$ with $\G$-invariant harmonic cocycles. The weight $k$ and type $n$ of the Drinfeld cusp forms determines the coefficient vector space $V$ we should consider. Assuming the conjecture from \cite{gr2021} that the map is an isomorphism, our results on the present $U_i$-operators can be applied to the $U_i$-operators acting on spaces of Drinfeld cusp forms of level $\G_1(t)$ and rank $3$ (and the corresponding subspaces for the other congruence subgroups we consider). 
\subsection{Coefficient vector space $V$}
The coefficients we are primarily interested in are the dual of a symmetric power representation $\Sym^k(F^3)$ with a determinant twist, 
interpreted as a space of homogeneous polynomials:

Let $k, n$ be non-negative integers. Let $P_{k}$ denote the $F$-vector space $F[X, Y, Z]_{deg=k}$ of homogeneous polynomials over $F$ in variables $X, Y, Z$ of degree $k$. A basis of $P_k$ is given by 
\[ \uB = \{ X^l Y^m Z^{k-(l+m)} \| 0 \leq l \leq k, 0 \leq m \leq k, \text{ s.\,t. } l+m \leq k \}. \]
We define an $F$-vector space $V_{k,n}$ as the dual of $P_{k}$ equipped with the following action of $\GL_3(F)$: Let $\g = \left(\begin{smallmatrix}
 a & b & c \\
 d & e & f \\
 g & h & i \\
\end{smallmatrix}\right) \in \GL_3(F)$ be a matrix and $v\in V_k$. Then we define for basis elements $X^l Y^m Z^{k-(l+m)} \in \uB$:
\[
(\g v) (X^l Y^m Z^{k\shortminus(l+m)})\hspace{-.5pt} := \hspace{-.5pt}\det(\g)^{1\shortminus n} v((aX + bY + cZ)^l (dX + eY + fZ)^m (gX + hY + iZ)^{k\shortminus(l+m)})\hspace{-.5pt}.
\]
Gräf then suggests in \cite[][Chapter 8]{gr2021} for every congruence subgroup $\Gamma(t)\subset\G \subset \GL_3(A)$ a Hecke-invariant isomorphism of $\C$-vector spaces
\[ \Phi_{k,n}: S_{k+3, n}(\G) \rightarrow \Char(\G, V_{k, n}) \otimes_F \C.\]
\textbf{Note: } Our definition for the action of $\GL_3(F)$ on the spaces $V_{k, n}$ differs from that of Gräf by the automorphism $\g \mapsto \g^{-T}$. This is due to the fact that Gräf works with a left action of $\GL_3(\F)$ on Drinfeld cusp forms that involves transposing the matrix compared to the right action defined in \cite{bas.bre.pin2024}, cf. \cite[][Remark 6.3]{gr2021}. Also, in \cite{gr2021} the weight $k$ is always divisible by 3, but his constructions can be extended to general~$k$.

\subsubsection{A basis for $V$} \label{subsec:basis}
For simplicity, we only consider the case $n=0$ and write $V_k := V_{k, 0}$ from here on, but the implementation could be augmented for general type $n$ without much trouble. 

For our computations, we first need to choose a basis of $V_k$ and describe the action of $\GL_3(F)$ in terms of it.
We fix the basis $\uB^* = \{ v_{\la,\mu} \| 0 \leq \la \leq k, 0 \leq \mu \leq k, \text{ such that } \la+\mu \leq k \}$ 
of $V_k$ as the dual basis of $\uB$, where
\[
v_{\la, \mu} (X^l Y^m Z^{k-(l+m)}) = \delta_{\la, l} \delta_{\mu, m} = \begin{cases}
1, & \text{ if }\lambda = l \text{ and } \mu = m, \\
0, & \text{ else.}
\end{cases}
\]
Observe that $\dim V_k = \dim P_k = \binom{k + 2}{2}.$

Now we want to analyze the action of $\GL_3(F)$ on $V_{k}$ so that we can teach it to \texttt{Magma}. 
As above, let $\g = \left(\begin{smallmatrix}
 a & b & c \\
 d & e & f \\
 g & h & i \\
\end{smallmatrix}\right) \in \GL_3(F)$ be a matrix,
$X^l Y^m Z^{k-(l+m)} \in \uB$ a basis element and $v_{\la, \mu} \in \uB^*$ an element of the dual basis. Then 
\begin{align}
\nonumber (\g v_{\la, \mu}) & (X^l Y^m Z^{k-(l+m)}) = \det(\g) v_{\la, \mu} \left((aX + bY + cZ)^l (dX + eY + fZ)^m \right. \\
\nonumber     &     \hspace*{6cm} \left. (gX + hY + iZ)^{k-(l+m)}\right) \\
\nonumber      = &\det(\g) v_{\la, \mu} \left( \left(\sum_{\substack{l_1 + l_2 + l_3\\= l}} \binom{l}{l_1, l_2, l_3} (aX)^{l_1}(bY)^{l_2}(cZ)^{l_3} \right) \right. \\
\nonumber     & \cdot \left( \sum_{\substack{m_1 + m_2 + m_3 \\= m}} \binom{m}{m_1, m_2, m_3} (dX)^{m_1}(eY)^{m_2}(fZ)^{m_3} \right) \\
\nonumber     & \left. \cdot \left( \sum_{\substack{n_1 + n_2 + n_3 \\= k-(l+m)}} \binom{k-(l+m)}{n_1, n_2, n_3} (gX)^{n_1}(hY)^{n_2}(iZ)^{n_3} \right) \right) \\
\label{eq:transformationmatrix}     = & \det(\g) \sum_{\substack{l_1 + l_2 + l_3 \\= l}} \sum_{\substack{m_1 + m_2 + m_3 \\= m}} \sum_{\substack{n_1 + n_2 + n_3 \\= k-(l+m)}} \binom{l}{l_1, l_2, l_3}\binom{m}{m_1, m_2, m_3}\binom{k-(l+m)}{n_1, n_2, n_3} \\
\nonumber     & \hspace*{1cm} a^{l_1} b^{l_2} c^{l_3} d^{m_1} e^{m_2} f^{m_3} g^{n_1} h^{n_2} i^{n_3} \underbrace{v_{\la, \mu}(X^{l_1+m_1+n_1} Y^{l_2+m_2+n_2} Z^{l_3+m_3+n_3})}_{\substack{= 1,\text{ if } l_1+m_1+n_1 = \lambda \text{ and } l_2+m_2+n_2 = \mu,\\ 0 \text{ else.}}} \\
\nonumber     =: & M(\g)_{\la, \mu; l, m}.
\end{align}
This allows us to calculate a transformation matrix $M^{\uB^*}_{\uB^*}(\g)$ for the endomorphism on $V_{k}$ given by $v \mapsto \g v$ with respect to the basis $\uB^*$, which has entries $M(\g)_{\la, \mu; l, m}$. However, it is computationally costly.
\subsubsection{Relevant subspaces}
The subspaces we want to investigate in more detail are given in Theorems~\ref{thm:IdentifyGL3}, \ref{thm:Isog-P0}, \ref{thm:Isog-P2}, \ref{thm:IsomG0}. We use the same notations $H_0, H_2, \sigma$ and $\tau$ that we introduced there.
\begin{align*}
\Char(\Gamma_0(t),V_k) & \cong V_k^{D}, \\
\Char(\G^P_0,V_k) & \cong \lb v \in V_k^{H_0} \,\middle|\,  \sigma v = - \sum_{a \in \Fq} \left(\begin{smallmatrix}
        1&a&0\\0&1&0\\0&0&1
    \end{smallmatrix}\right) v \rb, \\
\Char(\G^P_2,V_k) & \cong \lb v \in V_k^{H_2} \,\middle|\,  \tau v = - \sum_{a \in \Fq} \left(\begin{smallmatrix}
        1&0&0\\0&1&a\\0&0&1
    \end{smallmatrix}\right) v \rb, \text{ and}\\
\Char(\GL_3(A),V_k) & \cong \lb v \in V_k^{B^T} \,\middle|\,  \forall \rho \in W: \, \rho v = \sgn(\rho s_0) \sum_{\g \in \G_1(t)_{(\rho s_0)}} \g v \rb. 
\end{align*}
In the following, we explain how to construct them in \texttt{Magma}. 

The first space $\Char(\Gamma_0(t),V_k) \cong V_k^{D(\Fq)}$ is generated by a subset of the basis vectors in $\uB^*$, so this is the description that we use in the \texttt{Magma} program.
\begin{prop} \label{prop:basis-vd}
A basis for $V_k^{D}$ is given by 
\begin{align*}
C^* := \{ v_{\la,\mu} \in \uB^* \| & 0 \leq \la \leq k, 0 \leq \mu \leq k, \text{ such that } \la+\mu \leq k \text{ and } \\
    & q-1 \text{ divides }(\la+1), (\mu+1), (k - (\la + \mu) + 1)\}.
\end{align*}
This is nonzero only if $q-1$ divides $k+3$. In this case, the third condition of $q-1$ dividing $(k - (\la + \mu) + 1)$ can be dropped.
\end{prop}
\begin{proof}
Let $v \in V_k$ and $\g = \left( \begin{smallmatrix}
a & 0 & 0 \\ 
0 & b & 0 \\ 
0 & 0 & c
\end{smallmatrix}\right) \in D$. Then $\g v = v$ if and only if for all basis elements $X^l Y^m Z^{k-(l+m)} \in \uB$
\begin{align*}
&& \left( \left( \begin{smallmatrix}
a & 0 & 0 \\ 
0 & b & 0 \\ 
0 & 0 & c
\end{smallmatrix}\right) v \right) (X^l Y^m Z^{k-(l+m)}) & = v (X^l Y^m Z^{k-(l+m)}) \\
\Leftrightarrow &&  abc \: v((ax)^l (bY)^m (cZ)^{k-(l+m)}) & = v (X^l Y^m Z^{k-(l+m)}) \\
\Leftrightarrow &&   a^{l+1} b^{m+1} c^{k-(l+m)+1} v (X^l Y^m Z^{k-(l+m)}) & = v (X^l Y^m Z^{k-(l+m)}) \\
\Leftrightarrow && a^{l+1} b^{m+1} c^{k-(l+m)+1} = 1 & \text{ or } v (X^l Y^m Z^{k-(l+m)}) = 0.
\end{align*}
Hence, $v$ is in $V_k^{D}$ if the indices of all $v_{\la, \mu}$ appearing in its basis decomposition with non-zero coefficient satisfy 
\[ a^{\la+1} b^{\mu+1} c^{k-(\la+\mu)+1} = 1 \]
for all $a, b, c \in \Fq^{\times}$. This is equivalent to $q-1$ dividing $(\la+1), (\mu+1)$ and $(k - (\la + \mu) + 1)$ as stated. In this case, $q-1$ also divides the sum of the three exponents, which is $k+3$.
\end{proof}

For the bigger groups $\G^P_0, \G^P_2$ and $\GL_3(A)$, we first want to describe the spaces $V_k^{H_0}$, $V_k^{H_0}$ and $V_k^{B^T}$ to \texttt{Magma}, before implementing the conditions for the permutation matrices.
\begin{lem} 
We define two matrices $\eta_0 := \left(\begin{smallmatrix}
1 & 0 & 0 \\ 
1 & 1 & 0 \\ 
0 & 0 & 1
\end{smallmatrix}\right)$ and $\eta_2 := \left(\begin{smallmatrix}
1 & 0 & 0 \\ 
0 & 1 & 0 \\ 
0 & 1 & 1
\end{smallmatrix}\right)$.
Then we can give the following descriptions:
\begin{enumerate}[(i)]
\item The group $H_0 = \left(\begin{smallmatrix}
    *&0&0\\ *&*&0\\0&0&*
\end{smallmatrix}\right)\! \subset \GL_3(\Fq)$ can be written as 
$ H_0 = D\langle \eta_0 \rangle  = \langle \eta_0 \rangle D$. 
\item The group $H_2 = \left(\begin{smallmatrix}
    *&0&0\\ 0&*&0\\0&*&*
\end{smallmatrix}\right) \!\subset \GL_3(\Fq)$ can be written as 
$ H_0 = D\langle \eta_2 \rangle  = \langle \eta_2 \rangle D$. 
\item We can decompose the group $B^T$ of lower triangular matrices with entries in $\Fq$ as
$ B^T = DU^T = U^T D $. 
\item The group $U^T$ of unipotent lower triangular matrices is generated by $\eta_0$ and $\eta_2$.
\end{enumerate}
It follows that 
\begin{align*}
    V_k^{H_0} & = V_k^{D} \cap V_k^{\eta_0}, \\
    V_k^{H_2} & = V_k^{D} \cap V_k^{\eta_2}, \\
    V_k^{B^T} & = V_k^{D} \cap V_k^{\eta_0} \cap V_k^{\eta_2}  = V_k^{H_0} \cap V_k^{H_2}. 
\end{align*}
\end{lem} 
\begin{proof}
Direct calculation.
\end{proof}
Spaces such as $V_k^{\eta_0}$ can be calculated in \texttt{Magma} as the eigenspace with eigenvalue $1$ of the linear map $v \mapsto \eta_0 v$. Similarly, conditions like $\sigma v \stackrel{!}{=} - \sum_{a \in \Fq} \left(\begin{smallmatrix}
        1&a&0\\0&1&0\\0&0&1
    \end{smallmatrix}\right) v$ are translated to \texttt{Magma} by computing the kernel of the map $v \mapsto \sigma v + \sum_{a \in \Fq} \left(\begin{smallmatrix}
        1&a&0\\0&1&0\\0&0&1
    \end{smallmatrix}\right) v$. 
This uses the fact that we can compute transformation matrices for the linear map given by $v \mapsto \gamma v$ for any matrix $\g \in \GL_3(F)$ as described in \ref{subsec:basis}.
In our \texttt{Magma} program, we type out all the conditions for the permutation matrices in the descriptions of $\Char(\G^P_0,V_k), \Char(\G^P_2,V_k)$ and $\Char(\GL_3(A),V_k)$.

Now we are able to describe all relevant actions and subspaces to \texttt{Magma}.

\subsection{Possible entries in the resulting matrices}\label{subsec:entries}
\subsubsection{Entries in the matrices for $A_i, B_i$ and $C_i$}
In Theorem \ref{thm:AiBiCi}, we showed that the Hecke-operators can be written as the action of two diagonal matrices $\delta_i^{-1}$ concatenated with operators $A_i, B_i$ and $C_i$ on $V$ that can be described by matrices with entries in the finite field $\Fq$. From the description of the transformation matrices in \ref{subsec:basis} it becomes clear that also the transformation matrices for the operators $A_i, B_i$ and $C_i$ only contain entries in $\Fq$. But it turns out that all the entries even lie in $\mathbb{F}_p$! We check this computationally in our \texttt{Magma} program. 

However, this could also be shown theoretically by evaluating the formula from Equation (\ref{eq:transformationmatrix}) by hand for the occurring matrices and using the fact that $\sum_{a \in \Fq^{\times}} a^n \in \{-1, 0\}$ for all $n \in \N$. Then what remains after cancellation are sums of products of multinomial coefficients, which lie in $\mathbb{F}_p$.
\subsubsection{The action of $\delta_i^{-1}$}
Occurrences of $t$ in the representation matrix of the $T_i$- and $U_i$-operators can only arise from the action of $\delta_i^{-1}$. We will perform a short calculation by hand to understand it.

Let $X^l Y^m Z^{k-(l+m)} \in \uB$ be a basis element and $v_{\la, \mu} \in \uB^*$ an element of the dual basis. Then
\begin{align*}
\delta_1^{-1} v_{\lambda, \mu} (X^l Y^m Z^{k-(l+m)}) & = \left(\begin{smallmatrix}
 1 & 0 & 0 \\
 0 & 1 & 0 \\
 0 & 0 & \pi \\
\end{smallmatrix}\right) v_{\lambda, \mu} (X^l Y^m Z^{k-(l+m)}) \\
& = \pi v_{\lambda, \mu} (X^l Y^m (\pi Z)^{k-(l+m)}) \\
& = \pi^{k+1-(l+m)} v_{\lambda, \mu} (X^l Y^m Z^{k-(l+m)}) \\
& =  \begin{cases}
\pi^{k+1-(l+m)} = t^{-(k+1) + (l+m)} , & \text{ if }\lambda = l, \mu = m, \\
0, & \text{ else.}
\end{cases}
\end{align*}
Correspondingly, for $i=2$:
\begin{align*}
\delta_2^{-1} v_{\lambda, \mu} (X^l Y^m Z^{k-3-(l+m)}) & = \left(\begin{smallmatrix}
 1 & 0 & 0 \\
 0 & \pi & 0 \\
 0 & 0 & \pi \\
\end{smallmatrix}\right)  v_{\lambda, \mu} (X^l Y^m Z^{k-(l+m)}) \\
& = \pi^2 v_{\lambda, \mu} (X^l (\pi Y)^m (\pi Z)^{k-(l+m)}) \\
& = \pi^{k+2-l} v_{\lambda, \mu} (X^l Y^m Z^{k-(l+m)}) \\
& =  \begin{cases}
\pi^{k+2-l} = t^{-(k+2) + l}, & \text{ if }\lambda = l, \mu = m, \\
0, & \text{ else.}
\end{cases}
\end{align*}
In both cases, the transformation matrix $M_{\uB^*}^{\uB^*}(\delta_i^{-1})$ for the action of $g_i^{-1}$ with respect to the basis $\uB^*$ is given by a diagonal matrix with powers of $\pi = t^{-1}$ on the diagonal. 

\subsection{Renormalization} 
In the end,\\ 
\centerline{\textbf{we renormalize the matrix for $U_i$ by multiplication with a factor $t^{k+i}$}.}\\[1ex]
We give several heuristic motivations for this renormalization:
\begin{enumerate}[(1)]
    \item The factor $t^{k+i}$ is chosen in such a way that in the diagonal matrix $M_{\uB^*}^{\uB^*}(\delta_i^{-1})$ from above, the variable $t$ only appears with non-negative exponent, and the exponent $0$ occurs. This seems a reasonable choice since Hecke operators should be normalized in such a way that they preserve $A$-integrality.
    \item Under the residue homomorphism from \cite{gr2021} the $U_i$-operator acting on spaces of Drinfeld cusp forms of level $\G_1(t)$ and rank $3$ should pass to the present $U_i$-operator; assuming the conjecture from \cite{gr2021} that the map is an isomorphism. This would result in the coincidence with the present normalization, for a natural choice of $U_i$ on such cusp form spaces. We hope to shed some light on this in a planned more theoretical continuation of the present work.
    \item In Observation~\eqref{ItemSeven} of Section~\ref{sec:Ti+Ui-ops} we suggest for $q=2$ an (experimentally found) inclusion of rank $2$ cocyles into rank $3$ cocycles. If in both cases one takes for the weight the $k$ in the $\Sym^k$ of the coefficients, then it is an inclusion in the same weight. 
    This is an a posteriori motivation, since it was observed after choosing the normalization according to (1) and~(2). Another normalization would have given a shift of the slopes.
\end{enumerate}

\subsection{Pseudocode}
Now we can describe the functionality of the \texttt{Magma} program with pseudocode in Algorithm \ref{alg}. The actual code can be found on GitHub at \cite{code}. It relies on the built-in \texttt{Magma} function \texttt{Eigenspace(M, e)} that returns the eigenspace for the eigenvalue $e$ as a subspace of the vector space $V$ on which the Matrix $M$ acts by multiplication. One has synonymous functions \texttt{Eigenspace(M, 0)} = \texttt{Nullspace(M)} = \texttt{Kernel(M)}. 
\begin{breakablealgorithm}
\caption{Compute slopes for all our Hecke operators}\label{alg}
\footnotesize
\begin{algorithmic}[1]
\Require $q$ the size of $\Fq$, $k$ the degree for $V_k$ s.\,t. $(q-1)$ divides $k+3$.
\Ensure The slopes of $T_i$ and $U_i^{\G}$, for $i \in \{1, 2\}$ and $\G \in \{\G_0(t), \G_1(t), \G^P_0, \G^P_2\}$. 
\State $f := \Fq$, $A := \Fq[t], F := \Fq(t)$;
\State $s := [(l, m) : l \in [0..k], m \in [0..k] | l+m \leq k ]$ a list of tuples describing the indices of basis elements in $V_k$;
\State $V := $ vector space over $F$ of dimension $\#s$. This represents the underlying vector space of $V_k$;
\State $M(s, \g) := $ a function returning the transformation matrix for the action of $\g$ on $V_k$, using the list $s$ to convert indices in our formulas to indices in the \texttt{Magma} representation of $V$;
\State $C := $ those standard basis vectors $e_i \in V$ such that $s[i] = (l,m)$ and $q-1$ divides $l+1$ and $m+1$;
\State $V^D := $ the subspace of $V$ generated by the elements of $C$;
\State $V^{H_0} := V^D \cap Eigenspace(M(s, \eta_0), 1)$ ;
\State $V^{H_2} := V^D \cap Eigenspace(M(s, \eta_2), 1)$;
\State $V_{\G^P_0} := V^{H_0} \cap Nullspace(M(s, \sigma) - \sgn(\sigma s_0) \sum_{\g \in \G_1(t)_{(\sigma s_0)}} M(s,\g))$;
\State $V_{\G^P_2} := V^{H_2} \cap Nullspace(M(s, \tau) - \sgn(\tau s_0) \sum_{\g \in \G_1(t)_{(\tau s_0)}} M(s,\g))$;
\State $V_G := V_{\G^P_0} \cap V_{\G^P_2}$ (initialization);
\ForAll{permutation matrices $\rho \in W \setminus \{ id, \sigma, \tau \}$}
    \State $V_G := V_G \cap Nullspace(M(s, \rho) - \sgn(\rho s_0) \sum_{\g \in \G_1(t)_{(\rho s_0)}} M(s,\g))$, where the values of $\sgn(\rho s_0)$ and $\g \in \G_1(t)_{(\rho s_0)}$ were manually added to the code (also when computing $V_{\G^P_0}$, $V_{\G^P_2}$);
\EndFor
\State $i_{V_{\G_0(t)}}, i_{V_{\G^P_0}}, i_{V_{\G^P_2}}, i_{V_G} := $ base change matrices for a basis of $V^D, \dots, V_G$ extended to a basis of $V$ to the basis whose indices are described by the list $s$;
\For{$i \in \{1, 2\}$}
    \State $A_i, B_i, C_i$ := the transformation matrices for the operators $A_i, B_i, C_i$ from Theorem \ref{thm:AiBiCi} acting on all of $V$, calculated by typing out all matrices in their description and applying the function $M$ many times;
    \State Check that the entries of $A_i, B_i, C_i$ lie in $\mathbb{F}_p$
    \State $D := M(s, t^{k+i} \delta_i^{-1})$ the transformation matrix for the renormalized diagonal matrix;
    \For{$\G \in \{ \G_1(t), \G_0(t), \G^P_0, \G^P_2, \GL_3(A) \}$}
        \State $U_V :=$ the transformation matrix for $U_i^{\G}$, resp. $T_i$, calculated by combining $A_i \circ D$, $B_i \circ D$, and $C_i \circ D$ according to Theorem \ref{thm:AiBiCi}. First we get a function on all of $V$ from the function $M$;
        \State $U := $ a submatrix of $i_{V_{\G}}^{-1} \circ U_V \circ i_{V_{\G}}$ corresponding to the restriction to the correct subspace;
        \State $P := $ Newton polygon computed from the coefficients of $CharacteristicPolynomial(U)$;
        \State Save the current $i, \G$ and the slope and length of each segment in $P$;
        \State Factorize $CharacteristicPolynomial(U)$ and save the result;
    \EndFor
\EndFor
\end{algorithmic}
\end{breakablealgorithm}
\subsubsection*{Division into blocks}
In the current version of the code, the transformation matrices are always calculated for the action on all of $V_k$ and then cut down to the appropriate subspace.
The code could be significantly sped up (for $q > 2$) by including the following optimizations; however, this is \textbf{not} yet implemented. 

We divide the basis elements $v_{\lambda,\mu} \in B^*$ of the whole space $V_k$ into blocks of the form 
\[ B^*_{l, m} := \lb v_{\lambda,\mu} \in B^* \,\middle|\, \lambda \equiv l \text{ and } \mu \equiv m \mod q-1 \rb.\] 
Then one can show (e.g. by evaluating the formula from Equation (\ref{eq:transformationmatrix}) by hand for all the occurring matrices and carrying out cancellations) that the span of each block is closed under the action of $A_i$, and thus also under the action of $U_i^{\G_1(t)}$. One could also develop a theory mixed of nebentypes and types to show that the respective subspaces are preserved under all Hecke operators.

If we rearrange our basis as 
\[
B^* = \lb B^*_{0, 0}, \dots, B^*_{0, q-2}, B^*_{1, 0}, \dots, B^*_{q-2, 0}, \dots, B^*_{q-2, q-2} \rb,
\]
the representation matrix $M^{B^*}_{B^*}(U_i^{\G_1(t)})$ consists of up to $(q-1)^2$ blocks on the diagonal and zero entries everywhere else. We could analyze it by investigating each submatrix $M^{B^*_{l,m}}_{B^*_{l,m}}(U_i^{\G_1(t)})$ separately. 
The vector space $V_k^D = \Char(\G_0(t))$ corresponds to the block $B^*_{0, 0}$ (this is the content of Proposition \ref{prop:basis-vd}), and the matrix for $U_i^{\G_0(t)}$ appears in that for $U_i^{\G_1(t)}$.

For the action of the operators $B_i$ and $C_i$, respectively $U_i^{\G^P_j}$ and $T_i$, this does not apply, but we could still reduce the computation time by computing the transformation matrices for the action of each summand on $V_k^D$ instead of the whole $V_k$ and then cutting it down to the appropriate subspace. This is not implemented yet due to technical details.

\section{Observations and Questions on the Slopes of $T_i$- and $U_i$-Operators} 
\label{sec:Ti+Ui-ops}
In this section we present the results of the code that was explained in the previous section. It was executed for $i \in \{1, 2\}, q \in \{2, 3\}$ and $k \in \{
0 , \dots, 27 \}$. For bigger $k$, we encountered errors in \texttt{Magma}. We only show the results for $q=2$ and $k \leq 16$ here to save space. The full tables can be found on GitHub at \cite{code}. 

For bigger prime powers such as $q=4$ and upwards, it turns out in computation that many of the smaller subspaces such as $\Char(\GL_3(A), V_k)$ are zero for small $k$ (see also \cite[][Corollary 17.10]{bas.bre.pin2024} for the corresponding statement on cusp forms). At the same time, the dimension of $V_k$ itself grows quickly with $k$ and becomes much too high for Magma to compute eigenspaces in a reasonable time. Already for the numbers we considered here, the computation took several days.
\renewcommand{\arraystretch}{1.4}
\begin{table}
\centering
\begin{tabular}{|c|p{0.13\textwidth}|p{0.22\textwidth}|p{0.22\textwidth}|p{0.25\textwidth}|}
    \specialrule{1pt}{0pt}{0pt}
    \textbf{\;$\boldsymbol{k}$\;} & \textbf{$\boldsymbol{T_1}$-Slopes} & \textbf{$\boldsymbol{U_1^{\G^P_0}}$-Slopes} & \textbf{$\boldsymbol{U_1^{\G^P_2}}$-Slopes} & \textbf{$\boldsymbol{U_1^{\G_{0}(t)}}$-Slopes} \\
    \specialrule{1pt}{0pt}{0pt}
        0 &  &   &  & \textcolor{fullblue}{$0^{\mathbf{1}}$} \\
        \hline
        1 &  & $\infty^{\mathbf{1}}$ & $0^{\mathbf{1}}$ & $0^{\mathbf{1}}$, $\infty^{\mathbf{2}}$ \\
        \hline
        2 &  & $1^{\mathbf{1}}$, $\infty^{\mathbf{1}}$ & $0^{\mathbf{1}}$, $1^{\mathbf{1}}$ & $0^{\mathbf{1}}$, $1^{\mathbf{2}}$, $\infty^{\mathbf{3}}$ \\
        \hline
        3 &  & $\frac{3}{2}^{\mathbf{2}}$, $\infty^{\mathbf{1}}$ & $0^{\mathbf{1}}$, $\frac{3}{2}^{\mathbf{2}}$ & $0^{\mathbf{1}}$, $\frac{3}{2}^{\mathbf{4}}$, \textcolor{fullblue}{$2^{\mathbf{1}}$}, $\infty^{\mathbf{4}}$ \\
        \hline
        4 & $1^{\mathbf{1}}$ & $1^{\mathbf{1}}$, $2^{\mathbf{1}}$, $\infty^{\mathbf{3}}$ & $0^{\mathbf{1}}$, $1^{\mathbf{2}}$, $2^{\mathbf{1}}$, $\infty^{\mathbf{1}}$ & $0^{\mathbf{1}}$, $1^{\mathbf{2}}$, $2^{\mathbf{2}}$, \textcolor{fullblue}{$\frac{8}{3}^{\mathbf{3}}$}, $\infty^{\mathbf{7}}$ \\
        \hline
        5 & $2^{\mathbf{1}}$ & $2^{\mathbf{1}}$, $\frac{5}{2}^{\mathbf{2}}$, $\infty^{\mathbf{4}}$ & $0^{\mathbf{1}}$, $2^{\mathbf{3}}$, $\frac{5}{2}^{\mathbf{2}}$, $\infty^{\mathbf{1}}$ & $0^{\mathbf{1}}$, $2^{\mathbf{3}}$, $\frac{5}{2}^{\mathbf{4}}$, \textcolor{fullblue}{$\frac{10}{3}^{\mathbf{3}}$}, $\infty^{\mathbf{10}}$ \\
        \hline
        6 & $1^{\mathbf{1}}$ & $1^{\mathbf{1}}$, $3^{\mathbf{3}}$, $\infty^{\mathbf{5}}$ & $0^{\mathbf{1}}$, $1^{\mathbf{2}}$, $3^{\mathbf{5}}$, $\infty^{\mathbf{1}}$ & $0^{\mathbf{1}}$, $1^{\mathbf{2}}$, $3^{\mathbf{8}}$, \textcolor{fullblue}{$4^{\mathbf{4}}$}, $\infty^{\mathbf{13}}$ \\
        \hline
        7 & $\frac{3}{2}^{\mathbf{2}}$ & $\frac{3}{2}^{\mathbf{2}}$, $\frac{7}{2}^{\mathbf{2}}$, $4^{\mathbf{1}}$, $\infty^{\mathbf{7}}$ & $0^{\mathbf{1}}$, $\frac{3}{2}^{\mathbf{4}}$, $2^{\mathbf{1}}$, $\frac{7}{2}^{\mathbf{2}}$, $4^{\mathbf{2}}$, $\infty^{\mathbf{2}}$ & $0^{\mathbf{1}}$, $\frac{3}{2}^{\mathbf{4}}$, $2^{\mathbf{1}}$, $\frac{7}{2}^{\mathbf{4}}$, $4^{\mathbf{3}}$, \textcolor{fullblue}{$\frac{14}{3}^{\mathbf{6}}$}, $\infty^{\mathbf{17}}$ \\
        \hline
        8 & $1^{\mathbf{1}}$, $2^{\mathbf{1}}$ & $1^{\mathbf{1}}$, $2^{\mathbf{1}}$, $4^{\mathbf{3}}$, $5^{\mathbf{1}}$, $\infty^{\mathbf{9}}$ &  $0^{\mathbf{1}}$, $1^{\mathbf{2}}$, $2^{\mathbf{2}}$, $\frac{8}{3}^{\mathbf{3}}$, $4^{\mathbf{3}}$, $5^{\mathbf{1}}$, $8^{\mathbf{1}}$, $\infty^{\mathbf{2}}$ & $0^{\mathbf{1}}$, $1^{\mathbf{2}}$, $2^{\mathbf{2}}$, $\frac{8}{3}^{\mathbf{3}}$, $4^{\mathbf{6}}$, $5^{\mathbf{2}}$, \textcolor{fullblue}{$\frac{16}{3}^{\mathbf{6}}$}, $8^{\mathbf{1}}$, $\infty^{\mathbf{22}}$ \\
        \hline
        9 & $2^{\mathbf{1}}$, $4^{\mathbf{1}}$ & $2^{\mathbf{1}}$, $4^{\mathbf{1}}$, $\frac{9}{2}^{\mathbf{4}}$, $5^{\mathbf{2}}$, $\infty^{\mathbf{10}}$ & $0^{\mathbf{1}}$, $2^{\mathbf{3}}$, $\frac{10}{3}^{\mathbf{3}}$, $4^{\mathbf{2}}$, $\frac{9}{2}^{\mathbf{4}}$, $5^{\mathbf{2}}$, $8^{\mathbf{1}}$, $\infty^{\mathbf{2}}$ & $0^{\mathbf{1}}$, $2^{\mathbf{3}}$, $\frac{10}{3}^{\mathbf{3}}$, $4^{\mathbf{2}}$, $\frac{9}{2}^{\mathbf{8}}$, $5^{\mathbf{4}}$, \textcolor{fullblue}{$6^{\mathbf{7}}$}, $8^{\mathbf{1}}$, $\infty^{\mathbf{26}}$ \\
        \hline
        10 & $1^{\mathbf{1}}$, $3^{\mathbf{1}}$, $4^{\mathbf{1}}$ & $1^{\mathbf{1}}$, $3^{\mathbf{1}}$, $4^{\mathbf{1}}$, $5^{\mathbf{3}}$, $\frac{23}{4}^{\mathbf{4}}$, $\infty^{\mathbf{12}}$ & $0^{\mathbf{1}}$, $1^{\mathbf{2}}$, $3^{\mathbf{4}}$, $4^{\mathbf{3}}$, $5^{\mathbf{5}}$, $\frac{23}{4}^{\mathbf{4}}$, $\infty^{\mathbf{3}}$ & $0^{\mathbf{1}}$, $1^{\mathbf{2}}$, $3^{\mathbf{4}}$, $4^{\mathbf{3}}$, $5^{\mathbf{8}}$, $\frac{23}{4}^{\mathbf{8}}$, \textcolor{fullblue}{$\frac{20}{3}^{\mathbf{9}}$}, $\infty^{\mathbf{31}}$ \\
        \hline
        11 & $\frac{3}{2}^{\mathbf{2}}$, $4^{\mathbf{2}}$ & $\frac{3}{2}^{\mathbf{2}}$, $4^{\mathbf{2}}$, $\frac{11}{2}^{\mathbf{4}}$, $6^{\mathbf{1}}$, $\frac{13}{2}^{\mathbf{2}}$, $\infty^{\mathbf{15}}$ & $0^{\mathbf{1}}$, $\frac{3}{2}^{\mathbf{4}}$, $2^{\mathbf{1}}$, $4^{\mathbf{5}}$, $\frac{11}{2}^{\mathbf{8}}$, $6^{\mathbf{1}}$, $\frac{13}{2}^{\mathbf{2}}$, $\infty^{\mathbf{4}}$ & $0^{\mathbf{1}}$, $\frac{3}{2}^{\mathbf{4}}$, $2^{\mathbf{1}}$, $4^{\mathbf{5}}$, $\frac{11}{2}^{\mathbf{12}}$, $6^{\mathbf{2}}$, $\frac{13}{2}^{\mathbf{4}}$, \textcolor{fullblue}{$\frac{22}{3}^{\mathbf{12}}$}, $\infty^{\mathbf{37}}$ \\
        \hline
        12 & $1^{\mathbf{1}}$, $2^{\mathbf{1}}$, $5^{\mathbf{2}}$ & $1^{\mathbf{1}}$, $2^{\mathbf{1}}$, $5^{\mathbf{2}}$, $6^{\mathbf{5}}$, $\frac{20}{3}^{\mathbf{3}}$, $9^{\mathbf{1}}$, $\infty^{\mathbf{17}}$ & $0^{\mathbf{1}}$, $1^{\mathbf{2}}$, $2^{\mathbf{2}}$, $\frac{8}{3}^{\mathbf{3}}$, $5^{\mathbf{6}}$, $6^{\mathbf{7}}$, $\frac{20}{3}^{\mathbf{3}}$, $9^{\mathbf{1}}$, $16^{\mathbf{1}}$, $\infty^{\mathbf{4}}$ & $0^{\mathbf{1}}$, $1^{\mathbf{2}}$, $2^{\mathbf{2}}$, $\frac{8}{3}^{\mathbf{3}}$, $5^{\mathbf{6}}$, $6^{\mathbf{12}}$, $\frac{20}{3}^{\mathbf{6}}$, \textcolor{fullblue}{$8^{\mathbf{13}}$}, $9^{\mathbf{2}}$, $16^{\mathbf{1}}$, $\infty^{\mathbf{43}}$ \\
        \hline
        13 & $2^{\mathbf{1}}$, $\frac{5}{2}^{\mathbf{2}}$, $6^{\mathbf{2}}$ & $2^{\mathbf{1}}$, $\frac{5}{2}^{\mathbf{2}}$, $6^{\mathbf{2}}$, $\frac{13}{2}^{\mathbf{4}}$, $7^{\mathbf{2}}$, $\frac{22}{3}^{\mathbf{3}}$, $10^{\mathbf{1}}$, $\infty^{\mathbf{20}}$ & $0^{\mathbf{1}}$, $2^{\mathbf{3}}$, $\frac{5}{2}^{\mathbf{4}}$, $\frac{10}{3}^{\mathbf{3}}$, $6^{\mathbf{7}}$, $\frac{13}{2}^{\mathbf{4}}$, $7^{\mathbf{2}}$, $\frac{22}{3}^{\mathbf{3}}$, $8^{\mathbf{1}}$, $10^{\mathbf{1}}$, $16^{\mathbf{1}}$, $\infty^{\mathbf{5}}$ & $0^{\mathbf{1}}$, $2^{\mathbf{3}}$, $\frac{5}{2}^{\mathbf{4}}$, $\frac{10}{3}^{\mathbf{3}}$, $6^{\mathbf{7}}$, $\frac{13}{2}^{\mathbf{8}}$, $7^{\mathbf{4}}$, $\frac{22}{3}^{\mathbf{6}}$, $8^{\mathbf{1}}$, \textcolor{fullblue}{$\frac{26}{3}^{\mathbf{15}}$}, $10^{\mathbf{2}}$, $16^{\mathbf{1}}$, $\infty^{\mathbf{50}}$ \\
        \hline
        14 & $1^{\mathbf{1}}$, $3^{\mathbf{3}}$, $7^{\mathbf{2}}$ & $1^{\mathbf{1}}$, $3^{\mathbf{3}}$, $7^{\mathbf{7}}$, $\frac{31}{4}^{\mathbf{4}}$, $8^{\mathbf{1}}$, $9^{\mathbf{1}}$, $\infty^{\mathbf{23}}$ & $0^{\mathbf{1}}$, $1^{\mathbf{2}}$, $3^{\mathbf{8}}$, $4^{\mathbf{4}}$, $7^{\mathbf{11}}$, $\frac{31}{4}^{\mathbf{4}}$, $8^{\mathbf{2}}$, $9^{\mathbf{1}}$, $16^{\mathbf{1}}$, $\infty^{\mathbf{6}}$ & $0^{\mathbf{1}}$, $1^{\mathbf{2}}$, $3^{\mathbf{8}}$, $4^{\mathbf{4}}$, $7^{\mathbf{16}}$, $\frac{31}{4}^{\mathbf{8}}$, $8^{\mathbf{3}}$, $9^{\mathbf{2}}$, \textcolor{fullblue}{$\frac{28}{3}^{\mathbf{18}}$}, $16^{\mathbf{1}}$, $\infty^{\mathbf{57}}$ \\
        \hline
        15 & $\frac{3}{2}^{\mathbf{2}}$, $\frac{7}{2}^{\mathbf{2}}$, $4^{\mathbf{1}}$, $8^{\mathbf{1}}$ & $\frac{3}{2}^{\mathbf{2}}$, $\frac{7}{2}^{\mathbf{2}}$, $4^{\mathbf{1}}$, $\frac{15}{2}^{\mathbf{6}}$, $8^{\mathbf{4}}$, $\frac{17}{2}^{\mathbf{2}}$, $\frac{19}{2}^{\mathbf{2}}$, $\infty^{\mathbf{26}}$ & $0^{\mathbf{1}}$, $\frac{3}{2}^{\mathbf{4}}$, $2^{\mathbf{1}}$, $\frac{7}{2}^{\mathbf{4}}$, $4^{\mathbf{3}}$, $\frac{14}{3}^{\mathbf{6}}$, $\frac{15}{2}^{\mathbf{6}}$, $8^{\mathbf{8}}$, $\frac{17}{2}^{\mathbf{2}}$, $\frac{19}{2}^{\mathbf{2}}$, $16^{\mathbf{2}}$, $\infty^{\mathbf{6}}$ & $0^{\mathbf{1}}$, $\frac{3}{2}^{\mathbf{4}}$, $2^{\mathbf{1}}$, $\frac{7}{2}^{\mathbf{4}}$, $4^{\mathbf{3}}$, $\frac{14}{3}^{\mathbf{6}}$, $\frac{15}{2}^{\mathbf{12}}$, $8^{\mathbf{11}}$, $\frac{17}{2}^{\mathbf{4}}$, $\frac{19}{2}^{\mathbf{4}}$, \textcolor{fullblue}{$10^{\mathbf{19}}$}, $16^{\mathbf{2}}$, $\infty^{\mathbf{65}}$ \\
        \hline
        16 & $1^{\mathbf{1}}$, $2^{\mathbf{1}}$, $4^{\mathbf{3}}$, $5^{\mathbf{1}}$, $9^{\mathbf{1}}$ & $1^{\mathbf{1}}$, $2^{\mathbf{1}}$, $4^{\mathbf{3}}$, $5^{\mathbf{1}}$, $8^{\mathbf{5}}$, $\frac{26}{3}^{\mathbf{6}}$, $9^{\mathbf{2}}$, $10^{\mathbf{3}}$, $\infty^{\mathbf{29}}$ & $0^{\mathbf{1}}$, $1^{\mathbf{2}}$, $2^{\mathbf{2}}$, $\frac{8}{3}^{\mathbf{3}}$, $4^{\mathbf{6}}$, $5^{\mathbf{2}}$, $\frac{16}{3}^{\mathbf{6}}$, $8^{\mathbf{6}}$, $\frac{26}{3}^{\mathbf{6}}$, $9^{\mathbf{5}}$, $10^{\mathbf{3}}$, $16^{\mathbf{2}}$, $\infty^{\mathbf{7}}$ & $0^{\mathbf{1}}$, $1^{\mathbf{2}}$, $2^{\mathbf{2}}$, $\frac{8}{3}^{\mathbf{3}}$, $4^{\mathbf{6}}$, $5^{\mathbf{2}}$, $\frac{16}{3}^{\mathbf{6}}$, $8^{\mathbf{11}}$, $\frac{26}{3}^{\mathbf{12}}$, $9^{\mathbf{6}}$, $10^{\mathbf{6}}$, \textcolor{fullblue}{$\frac{32}{3}^{\mathbf{21}}$}, $16^{\mathbf{2}}$, $\infty^{\mathbf{73}}$ \\
    \specialrule{1pt}{0pt}{0pt}
\end{tabular}
  \caption{Slopes for $q=2$, $i=1$. Bold exponents denote multiplicities. Slopes of the form $\frac{2k}{3}$ are marked in blue. } \label{tab:slopes-i1}
\end{table}
\renewcommand{\arraystretch}{1}

\renewcommand{\arraystretch}{1.4}
\begin{table}
\centering
\begin{tabular}{|c|p{0.13\textwidth}|p{0.22\textwidth}|p{0.22\textwidth}|p{0.25\textwidth}|}
    \specialrule{1pt}{0pt}{0pt}
    \textbf{\;$\boldsymbol{k}$\;} & \textbf{$\boldsymbol{T_2}$-Slopes} & \textbf{$\boldsymbol{U_2^{\G^P_0}}$-Slopes} & \textbf{$\boldsymbol{U_2^{\G^P_2}}$-Slopes} & \textbf{$\boldsymbol{U_2^{\G_{0}(t)}}$-Slopes} \\
    \specialrule{1pt}{0pt}{0pt}
        0 &  &   &  & \textcolor{fullblue}{$0^{\mathbf{1}}$} \\
        \hline
        1 &  & $0^{\mathbf{1}}$ & $0^{\mathbf{1}}$ & $0^{\mathbf{2}}$, $\infty^{\mathbf{1}}$ \\
        \hline
        2 &  & $0^{\mathbf{2}}$ & $0^{\mathbf{1}}$, $\infty^{\mathbf{1}}$ & $0^{\mathbf{3}}$, $\infty^{\mathbf{3}}$ \\
        \hline
        3 &  & $0^{\mathbf{3}}$ & $0^{\mathbf{1}}$, $\infty^{\mathbf{2}}$ & $0^{\mathbf{4}}$, \textcolor{fullblue}{$1^{\mathbf{1}}$}, $\infty^{\mathbf{5}}$ \\
        \hline
        4 & $0^{\mathbf{1}}$ & $0^{\mathbf{4}}$, $\infty^{\mathbf{1}}$ & $0^{\mathbf{2}}$, $\infty^{\mathbf{3}}$ & $0^{\mathbf{5}}$, \textcolor{fullblue}{$\frac{4}{3}^{\mathbf{3}}$}, $\infty^{\mathbf{7}}$ \\
        \hline
        5 & $0^{\mathbf{1}}$ & $0^{\mathbf{5}}$, $1^{\mathbf{1}}$, $\infty^{\mathbf{1}}$ & $0^{\mathbf{2}}$, $1^{\mathbf{1}}$, $\infty^{\mathbf{4}}$ & $0^{\mathbf{6}}$, $1^{\mathbf{2}}$, \textcolor{fullblue}{$\frac{5}{3}^{\mathbf{3}}$}, $\infty^{\mathbf{10}}$ \\
        \hline
        6 & $0^{\mathbf{1}}$ & $0^{\mathbf{6}}$, $\frac{3}{2}^{\mathbf{2}}$, $\infty^{\mathbf{1}}$ & $0^{\mathbf{2}}$, $\frac{3}{2}^{\mathbf{2}}$, $\infty^{\mathbf{5}}$ & $0^{\mathbf{7}}$, $\frac{3}{2}^{\mathbf{4}}$, \textcolor{fullblue}{$2^{\mathbf{4}}$}, $\infty^{\mathbf{13}}$ \\
        \hline
        7 & $0^{\mathbf{2}}$ & $0^{\mathbf{7}}$, $1^{\mathbf{2}}$, $2^{\mathbf{1}}$, $\infty^{\mathbf{2}}$ & $0^{\mathbf{3}}$, $1^{\mathbf{1}}$, $2^{\mathbf{1}}$, $\infty^{\mathbf{7}}$ & $0^{\mathbf{8}}$, $1^{\mathbf{3}}$, $2^{\mathbf{2}}$, \textcolor{fullblue}{$\frac{7}{3}^{\mathbf{6}}$}, $\infty^{\mathbf{17}}$ \\
        \hline
        8 & $0^{\mathbf{2}}$ & $0^{\mathbf{8}}$, $\frac{4}{3}^{\mathbf{3}}$, $2^{\mathbf{1}}$, $4^{\mathbf{1}}$, $\infty^{\mathbf{2}}$ & $0^{\mathbf{3}}$, $\frac{4}{3}^{\mathbf{3}}$, $4^{\mathbf{1}}$, $\infty^{\mathbf{8}}$ & $0^{\mathbf{9}}$, $\frac{4}{3}^{\mathbf{6}}$, $2^{\mathbf{1}}$, \textcolor{fullblue}{$\frac{8}{3}^{\mathbf{6}}$}, $4^{\mathbf{2}}$, $\infty^{\mathbf{21}}$ \\
        \hline
        9 & $0^{\mathbf{2}}$ & $0^{\mathbf{9}}$, $1^{\mathbf{3}}$, $\frac{5}{3}^{\mathbf{3}}$, $4^{\mathbf{1}}$, $\infty^{\mathbf{2}}$ & $0^{\mathbf{3}}$, $1^{\mathbf{1}}$, $\frac{5}{3}^{\mathbf{3}}$, $4^{\mathbf{1}}$, $\infty^{\mathbf{10}}$ & $0^{\mathbf{10}}$, $1^{\mathbf{4}}$, $\frac{5}{3}^{\mathbf{6}}$, \textcolor{fullblue}{$3^{\mathbf{7}}$}, $4^{\mathbf{2}}$, $\infty^{\mathbf{26}}$ \\
        \hline
        10 & $0^{\mathbf{3}}$ & $0^{\mathbf{10}}$, $\frac{3}{2}^{\mathbf{6}}$, $2^{\mathbf{1}}$, $\frac{5}{2}^{\mathbf{2}}$, $\infty^{\mathbf{3}}$ & $0^{\mathbf{4}}$, $\frac{3}{2}^{\mathbf{2}}$, $2^{\mathbf{1}}$, $\frac{5}{2}^{\mathbf{2}}$, $\infty^{\mathbf{13}}$ & $0^{\mathbf{11}}$, $\frac{3}{2}^{\mathbf{8}}$, $2^{\mathbf{2}}$, $\frac{5}{2}^{\mathbf{4}}$, \textcolor{fullblue}{$\frac{10}{3}^{\mathbf{9}}$}, $\infty^{\mathbf{32}}$ \\
        \hline
        11 & $0^{\mathbf{3}}$, $1^{\mathbf{1}}$ & $0^{\mathbf{11}}$, $1^{\mathbf{4}}$, $2^{\mathbf{3}}$, $\frac{11}{4}^{\mathbf{4}}$, $\infty^{\mathbf{4}}$ & $0^{\mathbf{4}}$, $1^{\mathbf{2}}$, $2^{\mathbf{1}}$, $\frac{11}{4}^{\mathbf{4}}$, $\infty^{\mathbf{15}}$ & $0^{\mathbf{12}}$, $1^{\mathbf{5}}$, $2^{\mathbf{4}}$, $\frac{11}{4}^{\mathbf{8}}$, \textcolor{fullblue}{$\frac{11}{3}^{\mathbf{12}}$}, $\infty^{\mathbf{37}}$ \\
        \hline
        12 & $0^{\mathbf{3}}$, $2^{\mathbf{1}}$ & $0^{\mathbf{12}}$, $\frac{4}{3}^{\mathbf{6}}$, $2^{\mathbf{2}}$, $\frac{5}{2}^{\mathbf{2}}$, $3^{\mathbf{2}}$, $6^{\mathbf{1}}$, $8^{\mathbf{1}}$, $\infty^{\mathbf{4}}$ & $0^{\mathbf{4}}$, $\frac{4}{3}^{\mathbf{3}}$, $2^{\mathbf{1}}$, $\frac{5}{2}^{\mathbf{2}}$, $3^{\mathbf{2}}$, $8^{\mathbf{1}}$, $\infty^{\mathbf{17}}$ & $0^{\mathbf{13}}$, $\frac{4}{3}^{\mathbf{9}}$, $2^{\mathbf{2}}$, $\frac{5}{2}^{\mathbf{4}}$, $3^{\mathbf{4}}$, \textcolor{fullblue}{$4^{\mathbf{13}}$}, $6^{\mathbf{1}}$, $8^{\mathbf{2}}$, $\infty^{\mathbf{43}}$ \\
        \hline
        13 & $0^{\mathbf{4}}$, $1^{\mathbf{1}}$ & $0^{\mathbf{13}}$, $1^{\mathbf{5}}$, $\frac{5}{3}^{\mathbf{6}}$, $3^{\mathbf{3}}$, $4^{\mathbf{1}}$, $7^{\mathbf{1}}$, $8^{\mathbf{1}}$, $\infty^{\mathbf{5}}$ & $0^{\mathbf{5}}$, $1^{\mathbf{2}}$, $\frac{5}{3}^{\mathbf{3}}$, $3^{\mathbf{3}}$, $4^{\mathbf{1}}$, $8^{\mathbf{1}}$, $\infty^{\mathbf{20}}$ & $0^{\mathbf{14}}$, $1^{\mathbf{6}}$, $\frac{5}{3}^{\mathbf{9}}$, $3^{\mathbf{6}}$, $4^{\mathbf{2}}$, \textcolor{fullblue}{$\frac{13}{3}^{\mathbf{15}}$}, $7^{\mathbf{1}}$, $8^{\mathbf{2}}$, $\infty^{\mathbf{50}}$ \\
        \hline
        14 & $0^{\mathbf{4}}$, $\frac{3}{2}^{\mathbf{2}}$ & $0^{\mathbf{14}}$, $\frac{3}{2}^{\mathbf{10}}$, $2^{\mathbf{5}}$, $\frac{7}{2}^{\mathbf{2}}$, $4^{\mathbf{2}}$, $8^{\mathbf{1}}$, $\infty^{\mathbf{6}}$ & $0^{\mathbf{5}}$, $\frac{3}{2}^{\mathbf{4}}$, $2^{\mathbf{4}}$, $\frac{7}{2}^{\mathbf{2}}$, $4^{\mathbf{1}}$, $8^{\mathbf{1}}$, $\infty^{\mathbf{23}}$ & $0^{\mathbf{15}}$, $\frac{3}{2}^{\mathbf{12}}$, $2^{\mathbf{9}}$, $\frac{7}{2}^{\mathbf{4}}$, $4^{\mathbf{3}}$, \textcolor{fullblue}{$\frac{14}{3}^{\mathbf{18}}$} $8^{\mathbf{2}}$, $\infty^{\mathbf{57}}$ \\
        \hline
        15 & $0^{\mathbf{4}}$, $1^{\mathbf{1}}$, $2^{\mathbf{1}}$ & $0^{\mathbf{15}}$, $1^{\mathbf{6}}$, $2^{\mathbf{5}}$, $\frac{7}{3}^{\mathbf{6}}$, $4^{\mathbf{5}}$, $8^{\mathbf{2}}$, $\infty^{\mathbf{6}}$ & $0^{\mathbf{5}}$, $1^{\mathbf{2}}$, $2^{\mathbf{2}}$, $\frac{7}{3}^{\mathbf{6}}$, $4^{\mathbf{3}}$, $8^{\mathbf{2}}$, $\infty^{\mathbf{25}}$ & $0^{\mathbf{16}}$, $1^{\mathbf{7}}$, $2^{\mathbf{6}}$, $\frac{7}{3}^{\mathbf{12}}$, $4^{\mathbf{8}}$, \textcolor{fullblue}{$5^{\mathbf{19}}$}, $8^{\mathbf{4}}$, $\infty^{\mathbf{64}}$ \\
        \hline
        16 & $0^{\mathbf{5}}$, $2^{\mathbf{1}}$, $4^{\mathbf{1}}$ & $0^{\mathbf{16}}$, $\frac{4}{3}^{\mathbf{9}}$, $2^{\mathbf{3}}$, $\frac{8}{3}^{\mathbf{6}}$, $4^{\mathbf{6}}$, $\frac{9}{2}^{\mathbf{2}}$, $8^{\mathbf{2}}$, $\infty^{\mathbf{7}}$ & $0^{\mathbf{6}}$, $\frac{4}{3}^{\mathbf{3}}$, $2^{\mathbf{1}}$, $\frac{8}{3}^{\mathbf{6}}$, $4^{\mathbf{2}}$, $\frac{9}{2}^{\mathbf{2}}$, $8^{\mathbf{2}}$, $\infty^{\mathbf{29}}$ & $0^{\mathbf{17}}$, $\frac{4}{3}^{\mathbf{12}}$, $2^{\mathbf{3}}$, $\frac{8}{3}^{\mathbf{12}}$, $4^{\mathbf{7}}$, $\frac{9}{2}^{\mathbf{4}}$, \textcolor{fullblue}{$\frac{16}{3}^{\mathbf{21}}$}, $8^{\mathbf{4}}$, $\infty^{\mathbf{73}}$ \\
    \specialrule{1pt}{0pt}{0pt}
\end{tabular}
  \caption{Slopes for $q=2$, $i=2$. Bold exponents denote multiplicities. Slopes of the form $\frac{k}{3}$ are marked in blue. } \label{tab:slopes-i2}
\end{table}

\renewcommand{\arraystretch}{1}

\subsubsection*{Observations and questions}
Conjectures on slopes of classical modular forms go back to the seminal and inspiring paper \cite{GouveaMazurSlopes} by Gouvea and Mazur.
In the case of Drinfeld modular forms of rank $2$, slopes were explored in \cite[][Section 5]{ban.val2018} and \cite[][Section 3]{hat2021}.
Our tables are the first exploration of slopes in the $\GL_3$-case, over number or function fields. Some of the patterns and regularities we observe are parallel to observations in the rank $2$ case. Having two $U$-operators (in level $\G_0(t)$) is unique to the rank 3 case. 

Let us mention again, that the link to $\GL_3$-Drinfeld cusp forms is via the residue homomorphism of \cite{gr2021}, which however is only known for $r=2$ to be an isomorphism; see \cite{tei1991}. It is conjectured in \cite{gr2021} that the map is an isomorphism for $r=3$ as well. This is supported, for instance, by our computations and by using \cite[Theorem 17.11]{bas.bre.pin2024}, which confirm that the space of harmonic cocycles for $\GL_3(A)$ and the corresponding space of Drinfeld cusp forms have the same dimension for all computed weights. For $\G_1(t)$ a formula for the dimension of the corresponding space of Drinfeld cusp forms is given in \cite[][Theorem 3.5.6]{pin2020}, which shows that the dimensions of $S_{k+3,n}(\G_1(t))$ and $\Char(\G_1(t), V_{k,n}) \cong V_{k,n}$ agree for all weights $k$. This suggests, that our tables indeed represent slopes of Drinfeld cusp forms.
\begin{enumerate}[(1)]
    \item 
    \textbf{Zero is not an eigenvalue in level $1$?} Within the range of our computations, no $T_i$-eigenvalue on $\Char(\GL_3(A),V_k)$ was zero (i.\,e. no slope was infinite). 
    \item \textbf{The only $\G_0(t)$-new slopes seem to be $k/3$ and $2k/3$, respectively?} There is obviously an inclusion of the space of cocyles on the left to the two spaces in the middle columns and from the middle columns to the right. But on different spaces, we use different Hecke operators that do not necessarily preserve the subspaces. Therefore, it is remarkable that the slopes of the eigenvalues on the left are a subset of those in the middle columns, and in turn the middle columns are included in the right column. In future work we shall give a theoretical explanation for these containments, and also offer an explanation for many of the infinite slopes (zero eigenvalues) in the middle and right column. 
    
    What seems, from this perspective,  most remarkable to us, is that in our tables only the slope $\frac{2k}{3}$ for $U_1$ (Table \ref{tab:slopes-i1}), respectively $\frac{k}{3}$ for $U_2$ (Table \ref{tab:slopes-i2}), occur as \enquote{new} on the right (marked in blue); all other slopes on the right already occur as slopes in the middle two columns. This also will be explained in future work.
    \item \textbf{The maximal $U_2$-slope in $V_k$ is less than $k-4$?} 
    It is suggested in \cite{nic.ros2021} that the largest non-critical $U$-eigenvalue for their $U$-operator might be less then or equal to $k'-r+1$ where their $k'$ is the weight of higher rank Drinfeld modular forms. Through the residue homomorphism of Gr\"af, it should be expected that $k'=k+r$, and so the highest non-critical slope would be at most $k+1$, independently of $r$. We observe in the tables (for $q=2$) the maximal slope $k-4$ for $U_2$ (and for $U_1$ the maximal slope $2(k-4)$). For $q=3$ our tables seem too small to make any guesses; but even for $q=2$ our tables might be too small.

    It is perhaps also remarkable that the newform slopes are not the largest ones as seems to be the case in rank $2$. In rank $2$ all slopes of $T_\pi$ seem to be bounded by $(k-1)/2$, where $k/2$ is the newform slope.\footnote{If $k'$ denotes the weight in \cite{hat2021}, i.e., the weight of Drinfeld cusp forms, then for the comparison with our tables one should use $k=k'-2$.\label{foot-hattori}}
    Here however the $U_2$ slopes seem to go up to $k-4$ which can be much larger than the newform slope $k/3$ for $U_2$.
    \item \textbf{The largest $U_1$-slope is twice the largest $U_2$-slope?} The double of most $U_2$-slopes are a $U_1$-slope. But for many $k$ there is a small portion of exceptions; for $p=2$ and $k=12,13,20,\ldots,29$ we have 1--3 exceptions. What always is true in our tables is that the largest $U_1$ is twice the largest $U_2$-slope.
    \item \textbf{The $n$-th smallest slope is bounded independently of $k$?} 
    For $q=2$ in our tables up to $k=31$, the $n$-th smallest slopes for $n=1,\ldots,6$ are
    \[ \setstretch{1.3}
    \begin{array}{r|cccccc}
    n&1&2&3&4&5&6\\
    \hline\mathrm{largest \ }U_1\mathrm{-slope}& 0&2&\frac{10}3&4&7&8?\\
    \hline
    \mathrm{largest \ }U_2\mathrm{-slope}&0&\frac32&\frac53&3&4&5?
\end{array} \]
    This suggests that the $n$-th slope is only slightly larger than $n$. But more data needs to be collected, since for larger $n$ the growth may be stronger. The periodicities in the following item suggest that the number of values needed to guess the $n$-th largest slope grows exponentially with $n$. For that reason, our table might well be incorrect for $n=6$. 
    Hattori in \cite{hat2021} made for $r=2$ the rather cautious guess that the $n$-slope is at bounded by $q^{n-1}$.
    
\item \textbf{Periodicity of slopes?} 
In the rank 2 cases, Hattori conjectures in \cite{hat2021} that the $n$-th smallest finite slopes of $S_k(\G_1(t))$ are periodic of $p$-power period with respect to $k$ including multiplicities. This is also in analogy to the classical case, where Emerton’s theorem \cite{eme1998} states that the minimal slopes of $S_k(\G_0(2))$ are periodic of period 8. This is also in line with conjectures from \cite{GouveaMazurSlopes}, that are known to need some adjustments.

The observed slopes for the $U_1$-operator display a periodicity that seems qualitatively similar to the conjectures above. For instance, for $q=2$ one can make the following observations:
    \[ \setstretch{1.3}\begin{array}{r|ccccccccccc}
    U_1\mathrm{-slope}&0^1&1^{\mathbf{2}} &\frac{3^4}2&2^1&2^2&\frac{5^4}2&\frac{8^3}3&3^8&3^4&\frac{10^3}3&\frac{7^4}2\\
    \hline
    \mathrm{observed \ period}&1&2&4&4&4&8&4&8&8&4&8
\end{array} \]
The length of the period is a $2$-power. But its exponent seems not a simple function of the slope. The new form slopes seem to have shorter periods at comparable slopes than other slopes. 
For $q\ge3$ our data is to small to make good predictions.

For $U_2$ a new phenomenon occurs! The slope values display a periodicity, but moreover the multiplicities of the finite slopes grow in each repetition by a fixed increment. For $q=2$ one has the following data for slopes up to $\frac{5}{2}$:
    \[ \addtolength{\arraycolsep}{-.6pt}\setstretch{1.3}\begin{array}{r|llllllllllll}
\mathrm{first \ }U_2\mathrm{-slope\ occurence}&0^{\mathbf{1}}&1^{\mathbf{1}} &\frac{4}{3}^{\mathbf{3}}&\frac{3}{2}^{\mathbf{4}}&\frac{5}{3}^{\mathbf{3}}&2^{\mathbf{4}}&2^{\mathbf{2}}&2^{\mathbf{1}}&2^{\mathbf{2}}&2^{\mathbf{4}}&\frac{7}{3}^{\mathbf{6}}&\frac{5}{2}^{\mathbf{4}}\\
\hline
\mathrm{second \ }U_2\mathrm{-slope\ occurence}
    &0^{\mathbf{2}}&1^{\mathbf{2}} &\frac{4}{3}^{\mathbf{6}}&\frac{3}{2}^{\mathbf{8}}&\frac{5}{3}^{\mathbf{6}}&2^{\mathbf{9}}&2^{\mathbf{6}}&2^{\mathbf{3}}&2^{\mathbf{4}}&2^{\mathbf{8}}&\frac{7}{3}^{\mathbf{12}}&\frac{5}{2}^{\mathbf{8}}\\
    \hline
\mathrm{third \ }U_2\mathrm{-slope\ occurence}
    &0^{\mathbf{3}}&1^{\mathbf{3}} &\frac{4}{3}^{\mathbf{9}}&\frac{3}{2}^{\mathbf{12}}&\frac{5}{3}^{\mathbf{9}}&2^{\mathbf{14}}&2^{\mathbf{10}}&2^{\mathbf{5}}&2^{\mathbf{6}}&2^{\mathbf{12}}&\frac{7}{3}^{\mathbf{18}}&\frac{5}{2}^{\mathbf{12}}\\
    \hline
     \mathrm{observed \ period}&1&2&4&4&4&8&8&8&8&8&8&8   
\end{array} \]
This can be observed even better in the longer tables on GitHub. 

The observations on $U_1$ and $U_2$ explain that with growing $k$ many slopes of $U_2$ will have much higher multiplicities than slopes of $U_1$. Therefore, and this can easily be observed in our tables, the number of $U_2$-slopes is quite a bit smaller than the number of $U_1$-slopes.
\item 
\textbf{Slopes of rank $2$ forms occur in the same weight for $U_1^{\G^P_2}$ and doubled weight for $U_2^{\G^P_0}$ 
in rank $3$?}\label{ItemSeven}
For $k \geq 4$, the slopes from the table in \cite[Section~3]{hat2021} appear to be embedded in the $U_1^{\G^P_2}$-column at the same weight.\footnote{See again footnote~\ref{foot-hattori} for the weight shift for the comparison with \cite{hat2021}.} The multiplicity in rank $2$ can be smaller than in rank $3$. In fact, we observed that apart from the slope $0$ in rank $2$, all finite slopes from rank $2$ oldforms occur as slopes in rank $3$ for level $\GL_3(A)$. The rank $2$ newform slope occurs in all but the left column. The slope $0$ only occurs in the $U^{\G^P_2}_1$-column and in the $\G_0(t)$-column. 

For $U_2$, all finite slopes from rank $2$ forms except for the slope $1$ appear in the $U^{\G^P_0}_2$-column and in the $U_2^{\G_0(t)}$-column at the doubled weight! There seem to also be inclusions for not-1 slopes of rank 2 oldforms in the $T_2$ column, but for $k=10 = 2*5$, a slope $2$ is missing. For the $U^{\G^P_2}_2$, we seem to get inclusions of the not-1 slopes of rank 2 forms as well, but there is also a slope 2 missing at $k=8=2*4$.

We verified independently that not only the slopes but also the corresponding eigenvalues agree. 
However, we did not find obvious inclusions for $q\geq3$.
\item \textbf{Slopes can only have denominators that are $3$ or a power of $2$?} In our tables, when $p$ is odd, only $3$ can occur as denominator. When $p=2$, additionally powers of $2$ can occur. For example the slope $\frac{23}4$ occurs for $i=1, k=10$ and $i=2, k = 23$, and the slope $\frac{119}8$ for $i=1, k=27$.
\end{enumerate}

\appendix
\section{Appendix}
\label{appendix}
This appendix gives some further technical results used in Section~\ref{sec:building} and~\ref{sec:heckeop}. Subsections~\ref{subsec:quotients-fund} and~\ref{subsec:rep} explain how we computed systems of representatives for various (double) quotients of congruence subgroups of $\GL_3(A)$. In Subsection \ref{subsec:finding-simplices}, we give a sample calculation to explain how we found $\G_1(t)$-representatives in the standard apartment for certain simplices. 

To describe subgroups of $\GL_3(\Fq)$, we will often use Notation \ref{schreibweise}.
\subsection{Double quotients in study of fundamental domains}
\label{subsec:quotients-fund}
\begin{prop} \label{prop:double-quot-fund}
The following results hold:
\begin{enumerate}[(1)]
    \item The left action by $ \GL_3(\Fq) $ on $ \P^2(\Fq)$ induced from multiplying matrices on column vectors is transitive. The stabilizer of  $[1: 0:0]^T$ is $P_2=\left(\begin{smallmatrix}
 * & * & * \\
 0 & * & * \\
 0 & * & * \\
    \end{smallmatrix}\right)$. The left action of $U=\left(\begin{smallmatrix}
 1 & * & * \\
 0 & 1 & * \\
 0 & 0 & 1 \\
    \end{smallmatrix}\right)$ on $\P^2(\Fq)$  has three orbits, and $\{[1:0:0]^T, [0:1:0]^T,[0:0:1]^T\}$ is a set of orbit representatives.
\item
    The left action by $ \GL_3(\Fq) $ on $   \P^2(\Fq)$ induced from the action $(\g,v)\mapsto \g^{-T}v$ of matrices on column vectors is transitive. The stabilizer of  $[0:0:1]^T$ is $P_0=\left(\begin{smallmatrix}
 * & * & * \\
 * & * & * \\
 0 & 0 & * \\
    \end{smallmatrix}\right)$. The left action of $U=\left(\begin{smallmatrix}
 1 & * & * \\
 0 & 1 & * \\
 0 & 0 & 1 \\
    \end{smallmatrix}\right)$ on $\P^2(\Fq)$ has three orbits with representatives ${[1:0:0]^t}$, $[0:1:0]^t$ and $[0:0:1]^t$.
\item 
    Both double quotients $U\backslash\GL_3(\Fq)/P_2$ and $U\backslash\GL_3(\Fq)/P_0$ have a set of representatives given by the matrices
    \[
\left(\begin{smallmatrix}
 1 & 0 & 0 \\
 0 & 1 & 0 \\
 0 & 0 & 1 \\
\end{smallmatrix}\right),
\left(\begin{smallmatrix}
 0 & 0 & 1 \\
 1 & 0 & 0 \\
 0 & 1 & 0 \\
\end{smallmatrix}\right),
\left(\begin{smallmatrix}
 0 & 1 & 0 \\
 0 & 0 & 1 \\
 1 & 0 & 0 \\
\end{smallmatrix}\right).
\]
\end{enumerate}
\end{prop}
\begin{proof}
    Parts (1) and (2) are straightforward exercises. Note that $\g\mapsto\g^{-T}$ is an automorphism of $\GL_r$ of order $2$. Part (3) is deduced from (1) and (2) using the orbit stabilizer theorem. 
\end{proof}

\subsection{Finding sets of representatives}\label{subsec:rep} 
Here, we describe the method we used for finding sets of representatives for Hecke operators. Recall that in Lemma~\ref{lem:reps} we already stated how such sets can be found by computing a quotient of the form $\G_1\backslash \G_2$ for congruence subgroups of $\GL_3(A)$. Moreover all groups relevant to us contain $\G(t)$, and we shall therefore proceed as follows:
\begin{enumerate}
\item Reduce both groups modulo $t$, so that it suffices to compute a quotient $\widebar{\G_1} \backslash \widebar{\G_2}$ of subgroups of $\GL_3(\Fq)$; cf.~Lemma \ref{lem:reduction}.
\item Choose a suitable right action of $\GL_3(\Fq)$ on $\P^2(\Fq)$, cf.~Proposition \ref{prop:P2actions}, and a point $x \in \P^2(\Fq)$, such that $\Stab_{\widebar{\G_2}}(x) = \widebar{\G_1}$.
\item Describe the $\widebar{\G_2}$-orbit of $x$.
\item Use the Orbit-Stabilizer-Theorem to identify \[\widebar{\G_1} \backslash \widebar{\G_2} = \Stab_{\widebar{\G_2}}(x) \backslash \widebar{\G_2} \cong x \cdot \widebar{\G_2}, \G_1 \g \mapsto x \cdot \g.\]
\item Choose representatives $\g_i \in \G_2$ such that $\bigcup_i x \cdot \g_i = x \cdot \widebar{\G_2}$. Then these form a system of representatives for $\G_1 \backslash \G_2$.
\end{enumerate} 

We begin with a simple lemma whose proof is left as an exercise.
\begin{lem} \label{lem:reduction}
Denote by $\G(t) \subset \G_1 \subset \G_2 \subset \GL_3(A)$ congruence subgroups, and let $\widebar{\G_i}=\pr(\G_i)\subset\GL_3(\Fq)$. Then following the map is a bijection
\begin{align*}
f: \G_1 \backslash \G_2 & \rightarrow \widebar{\G_1} \backslash \widebar{\G_2},  \\
\G_1 \, \g  & \mapsto \widebar{\G_1} \, \widebar{\g}:=\pr(\g).
\end{align*}
\end{lem}

Next we describe the two right actions of $\GL_3(\Fq)$ on $\P^2(\Fq)$ relevant to us.
\begin{prop} \label{prop:P2actions}
The following rules define right group actions of $\GL_3(\Fq)$ on $\P^2(\Fq)$:
Denote elements of $\P^2(\Fq)$ as row vectors $[u: v: w]$, and let $\g = \left(\begin{smallmatrix}a&b&c\\d&e&f\\g&h&i\end{smallmatrix}\right)$ be in $\GL_3(\Fq)$. 
\begin{enumerate}[(i)]
\item The first action is right multiplication by $\g$ (written with \enquote{$\cdot$}), i.e., 
\[ [u: v: w] \cdot \g := [au + dv + gw : bu + ev + hw : cu + fw + iw] .\]
\item The second action is right multiplication by $\gamma^{-T}$  (written with \enquote{$*$}), i.e., 
\[ [u: v: w] * \g := [u: v: w] \cdot \g^{-T} .\]
\end{enumerate}
\end{prop}

Now we are ready to follow the steps above to find various systems of representatives.

\subsubsection*{Representatives for Hecke-operators} \label{app:rep-hecke} 

Recall
$\delta_1 = \left(\begin{smallmatrix}
	1 & 0 & 0 \\ 
	0 & 1 & 0 \\ 
	0 & 0 & t
\end{smallmatrix}\right)$ and $\delta_2 = \left(\begin{smallmatrix}
	1 & 0 & 0 \\ 
	0 & t & 0 \\ 
	0 & 0 & t
\end{smallmatrix}\right)$ from formula \eqref{eq:def-delta}. 

We shall compute representatives of certain double cosets that involve these. As we apply Lemma \ref{lem:reps}, we have to keep in mind that the representatives we actually compute in this subsection will have to be multiplied with $\delta_i$ in order to get the representatives described for Proposition \ref{prop:repr-all-op}.

We only give a detailed description for finding the representatives for the Hecke operators for the groups $\GL_3(A)$ and $\G_0(t)$. The cases for the other groups $\G_1(t)$, $\G^P_0$ and $\G^P_2$ work analogously.
\begin{prop} \label{prop:repr-Ti}
The following holds:
\begin{enumerate}[(1)]
    \item For $i=1,2$, the intersection $\delta_i^{-1} \GL_3(A) \delta_i \cap \GL_3(A)$ contains $\G(t)$, and one has
    \[ \pr(\delta_i^{-1} \GL_3(A) \delta_i \cap \GL_3(A))=\left\{
    \begin{array}{cc}
      P_0^T =\left(\begin{smallmatrix}
* & * & 0\\
* & * & 0\\
* & * & *
\end{smallmatrix}\right) ,  & \hbox{if }\ i=1, \\
         P_2^T = \left(\begin{smallmatrix}
* & 0 & 0\\
* & * & *\\
* & * & *
    \end{smallmatrix}\right),&\hbox{if }\ i=2, 
    \end{array}
    \right.
\]
inside $\pr(\GL_3(A))=\GL_3(\Fq)$.
\item The right action of $\GL_3(\Fq)$ on $\P^2(\Fq)$ via both \enquote{$\cdot$} or \enquote{$*$} is transitive. The stabilizer for the \enquote{$*$}-action of $[0:0:1]$ is $P_0^T$, the stabilizer for the \enquote{$\cdot$} action of $[1:0:0]$ is $P_2^T$.
\item
A set of matrix representatives for $(\delta_i^{-1} \GL_3(A) \delta_i \cap \GL_3(A)) \backslash  \GL_3(A)$ is given by
\[ \lb \left(\begin{smallmatrix}
1 & 0 & a \\ 
0 & 1 & b \\ 
0 & 0 & 1
\end{smallmatrix}\right), \left(\begin{smallmatrix}
1 & a & 0 \\ 
0 & 0 & 1 \\ 
0 & 1 & 0
\end{smallmatrix}\right), \left(\begin{smallmatrix}
0 & 1 & 0 \\ 
0 & 0 & 1 \\ 
1 & 0 & 0
\end{smallmatrix}\right) \,\middle|\, a, b \in \Fq \rb, \quad\hbox{for }i=1,\]
\[ \lb \left(\begin{smallmatrix}
1 & a & b \\ 
0 & 1 & 0 \\ 
0 & 0 & 1
\end{smallmatrix}\right), \left(\begin{smallmatrix}
0 & 1 & a \\ 
1 & 0 & 0 \\ 
0 & 0 & 1
\end{smallmatrix}\right), \left(\begin{smallmatrix}
0 & 0 & 1 \\ 
1 & 0 & 0 \\ 
0 & 1 & 0 
\end{smallmatrix}\right) \,\middle|\, a, b \in \Fq \rb, \quad\hbox{for }i=2.\]
\end{enumerate}
\end{prop}
\begin{proof}
The first two parts are left as an exercise. For the third part note that the number of matrices in each set in (3) is exactly the cardinality of $\P^2(\Fq)$, which by (1) and (2) is the cardinality of the right coset in question. To complete the proof one simply verifies that acting with the matrices given on $[0:0:1]$ via $*$ or on $[1:0:0]$ via $\cdot$, respectively, exhausts all of $\P^2(\Fq)$ and thus completes the proof.
\end{proof}

\begin{prop}\label{prop:repr-Ui}
The following holds:
\begin{enumerate}[(1)]
    \item For $i=1,2$, the intersection $\delta_i^{-1} \G_0(t)) \delta_i \cap \G_0(t)$ contains $\G(t)$, and one has
    \[ \pr(\delta_i^{-1} \G_0(t) \delta_i \cap \G_0(t))=\left\{
    \begin{array}{cc}
      \widebar\G_1 =\left(\begin{smallmatrix}
* & * & 0\\
0 & * & 0\\
0 & 0 & *
\end{smallmatrix}\right) ,  & \hbox{if }\ i=1, \\
         \widebar\G_2= \left(\begin{smallmatrix}
* & 0 & 0\\
0 & * & *\\
0 & 0 & *
    \end{smallmatrix}\right),&\hbox{if }\ i=2, 
    \end{array}
    \right.
\]
inside $B=\left(\begin{smallmatrix}
* & * & *\\
0 & * & *\\
0 & 0 & *
    \end{smallmatrix}\right)=\pr(\G_0(t))$.
\item For the orbits under $B$ one finds $[0:0:1]*B=\{[a:b:1]\mid a,b\in\Fq\}\cong\Fq^2$ and $[1:0:0]\cdot B=\{[1:a:b]\mid a,b\in\Fq\}\cong\Fq^2$. The \enquote{$*$}-stabilizer of $[0:0:1]$ inside $B$ is $\widebar{\G_1}$, the \enquote{$\cdot$}-stabilizer of $[1:0:0]$ inside $B$ is $\widebar{\G_2}$.
\item
A set of matrix representatives for $(\delta_i^{-1} \G_0(t) \delta_i \cap \G_0(t)) \backslash  \G_0(t)$ is given by
\[ \lb \left(\begin{smallmatrix}
1 & 0 & a \\ 
0 & 1 & b \\ 
0 & 0 & 1
\end{smallmatrix}\right) \,\middle|\, a, b \in \Fq \rb,\quad\hbox{for }i=1,\]
\[ \lb \left(\begin{smallmatrix}
1 & a & b \\ 
0 & 1 & 0 \\ 
0 & 0 & 1
\end{smallmatrix}\right) \,\middle|\, a, b \in \Fq \rb,\quad\hbox{for }i=2.\]
\end{enumerate}
\end{prop}
\begin{proof}
The first two parts are left as an exercise. For the computation of the stabilizers note that one only has to intersect the stabilizers inside $\GL_3(\Fq)$ from Proposition~\ref{prop:repr-Ti} with the subgroup $\G_0(t)$.

For the third part note that the number of matrices in each set in (3) is $q^2$, which by (1) and (2) is also the cardinality of the right coset in question. To complete the proof one simply verifies that acting with the given matrices on $[0:0:1]$ via $*$ or on $[1:0:0]$ via $\cdot$, respectively, exhausts all of 
$[0:0:1]*B$ or  $[1:0:0]\cdot B$, respectively, and this completes the proof.
\end{proof}

\subsection{Finding simplices in the standard apartment}
\label{subsec:finding-simplices}
For all matrices $\epsilon$ in the sets $Q_1, Q_2, R_1, R_2, S_1$, and $S_2$ from Proposition \ref{prop:repr-all-op}, we can find a chamber $s_{\epsilon}$ in the standard apartment and a matrix $\g_{\epsilon} \in \G_1(t)$ such that $\epsilon s_0 = \g_{\epsilon} s_{\epsilon}$. We view the simplices as equivalence classes of matrices in $\GL_3(\F)/\langle R \rangle I \F^{\times}$ (see Theorem \ref{thm:matsimpl}). 
We only give the calculation for the elements of $R_1$ here as a sample. The calculations for the other sets work similarly, by cleverly guessing matrices in $\langle R \rangle$, $I$ and $\G_0(t)$ to land at a simplex in the standard apartment.
Remember that 
\[ R = \left(\begin{smallmatrix}
 0 & 1 & 0 \\
 0 & 0 & 1 \\
 \pi & 0 & 0 \\
 \end{smallmatrix}\right) \text{ and } I = \lb M \in \GL_3(\Oi)
\, \Big| \, M \equiv \left(\begin{smallmatrix}
 * & * & * \\
 0 & * & * \\
 0 & 0 & * \\
\end{smallmatrix}\right) \mod \pi, * \in \Fq \rb \subset \GL_3(\Oi). \]
A representative for the stable simplex is given by
\[
s_0 = \lbk \left(\begin{smallmatrix}
 0 & 1 & 0 \\
 1 & 0 & 0 \\
 0 & 0 & t \\
\end{smallmatrix}\right) \rbk_2 \in \GL_3(\F) / \langle R \rangle I \F^{\times}.
\]

\subsubsection*{Representatives $\epsilon \in R_1$} \label{subsec:R1-repr-a}
In this case, we have $a \in \Fq$ and
\[
\epsilon s_0 = 
\begin{pmatrix}
 1 & a & 0 \\
 0 & 0 & 1 \\
 0 & t & 0 \\
\end{pmatrix}
\lbk
\begin{pmatrix}
 0 & 1 & 0 \\
 1 & 0 & 0 \\
 0 & 0 & t \\
\end{pmatrix} \rbk_2 = 
\lbk \begin{pmatrix}
 a & 1 & 0 \\
 0 & 0 & t \\
 t & 0 & 0 \\
 \end{pmatrix} \rbk_2. 
 \]
 We have to distinguish two cases.
 \paragraph*{$a=0$.}
 We investigate 
 \[ \lbk \begin{pmatrix}
 0 & 1 & 0 \\
 0 & 0 & t \\
 t & 0 & 0 \\
 \end{pmatrix} \rbk_2. \]
This simplex lies in the standard apartment and, using the third line in Theorem~\ref{thm:matsimpl}, it has the vertices
\[
\lbk \begin{pmatrix}
 0 & 1 & 0 \\
 0 & 0 & t \\
 t & 0 & 0 \\
\end{pmatrix} \rbk_0 = [-1, -1], \lbk \begin{pmatrix}
 0 & 1 & 0 \\
 0 & 0 & 1 \\
 t & 0 & 0 \\
\end{pmatrix} \rbk_0 = [0, -1], \lbk \begin{pmatrix}
 0 & \pi & 0 \\
 0 & 0 & 1 \\
 t & 0 & 0 \\
\end{pmatrix} \rbk_0 = [-1, -2].
\]
\paragraph*{$a \neq 0$.}
We investigate 
\begin{align*}
\lbk \begin{pmatrix}
 a & 1 & 0 \\
 0 & 0 & t \\
 t & 0 & 0 \\
\end{pmatrix} \right. 
\underbrace{\left. \begin{pmatrix}
 a^{-1} & 1 	& 0 \\
 0 		& - a 	& 0 \\
 0		& 0 	& 1 \\
\end{pmatrix} \rbk_2 \hspace*{-3mm}}_{\textstyle \in I}\hspace*{3mm} & = 
\lbk \begin{pmatrix}
 1 & 0 & 0 \\
 0 & 0 & t \\
 a^{-1} t & t & 0 \\
\end{pmatrix} \rbk_2 \\
& = 
\underbrace{\begin{pmatrix}
 1 		  & 0 & 0 \\
 0 		  & 1 & 0 \\
 a^{-1} t  & 0 & 1 \\
\end{pmatrix}}_{\textstyle \in \G_1(t)}
\underbrace{\lbk \begin{pmatrix}
 1 & 0 & 0 \\
 0 & 0 & t \\
 0 & t & 0 \\
\end{pmatrix} \rbk_2}_{\textstyle =: s_{\epsilon} \in \A}
\end{align*}
The simplex $s_{\epsilon}$ lies in the standard apartment and has the vertices 
\[
\lbk \begin{pmatrix}
 1 & 0 & 0 \\
 0 & 0 & t \\
 0 & t & 0 \\
\end{pmatrix} \rbk_0 = [-1, -1], \lbk \begin{pmatrix}
 1 & 0 & 0 \\
 0 & 0 & 1 \\
 0 & t & 0 \\
\end{pmatrix} \rbk_0 = [0, -1], \lbk \begin{pmatrix}
 1 & 0 & 0 \\
 0 & 0 & 1 \\
 0 & 1 & 0 \\
\end{pmatrix} \rbk_0 = [0, 0].
\]
In this case, $s_{\epsilon}$ is exactly the stable simplex $s_0$!

\renewcommand\listtablename{List of Figures and Tables}
\listoftablesandfigures 
\addcontentsline{toc}{section}{List of Figures and Tables}

\phantomsection
\addcontentsline{toc}{section}{Bibliography}
\printbibliography

@article{ban.val2018,
  title = {On the Atkin $U_t$-operator for $\Gamma_1(t)$-invariant Drinfeld cusp forms},
  author = {Bandini, Andrea and Valentino, Maria},
  date = {2018-11},
  journaltitle = {International Journal of Number Theory},
  shortjournal = {Int. J. Number Theory},
  volume = {14},
  number = {10},
  pages = {2599-2616},
  issn = {1793-0421, 1793-7310},
  doi = {10.1142/S1793042118501555},
  langid = {english}
}

@article{ban.val2019,
  title = {On the Atkin $U_t$-operator for $\Gamma_0(t)$-invariant Drinfeld cusp forms},
  author = {Bandini, Andrea and Valentino, Maria},
  date = {2019-06-27},
  journaltitle = {Proceedings of the American Mathematical Society},
  shortjournal = {Proc. Amer. Math. Soc.},
  volume = {147},
  number = {10},
  pages = {4171-4187},
  issn = {0002-9939, 1088-6826},
  doi = {10.1090/proc/14491},
  langid = {english}
}

@book {abramenko-brown,
    AUTHOR = {Abramenko, Peter and Brown, Kenneth S.},
     TITLE = {Buildings},
    SERIES = {Graduate Texts in Mathematics},
    VOLUME = {248},
      NOTE = {Theory and applications},
 PUBLISHER = {Springer, New York},
      YEAR = {2008},
     PAGES = {xxii+747},
      ISBN = {978-0-387-78834-0},
   MRCLASS = {20E42 (20F55 20J06 51E24 51F15)},
  MRNUMBER = {2439729},
MRREVIEWER = {Ralf Koehl},
       DOI = {10.1007/978-0-387-78835-7},
}

@article {bas.bre.pin2024,
    AUTHOR = {Basson, Dirk and Breuer, Florian and Pink, Richard},
     TITLE = {Drinfeld {M}odular {F}orms of {A}rbitrary {R}ank},
   JOURNAL = {Mem. Amer. Math. Soc.},
  FJOURNAL = {Memoirs of the American Mathematical Society},
    VOLUME = {304},
      YEAR = {2024},
    NUMBER = {1531},
     PAGES = {},
      ISSN = {0065-9266},
      ISBN = {978-1-4704-7223-8; 978-1-4704-8010-3},
   MRCLASS = {11F52 (11G09)},
  MRNUMBER = {4850411},
       DOI = {10.1090/memo/1531},
}

@article{bos.can.pla1997,
  title = {The {{Magma Algebra System I}}: {{The User Language}}},
  shorttitle = {The {{Magma Algebra System I}}},
  author = {Bosma, Wieb and Cannon, John and Playoust, Catherine},
  date = {1997-09-01},
  journaltitle = {Journal of Symbolic Computation},
  shortjournal = {Journal of Symbolic Computation},
  volume = {24},
  number = {3},
  pages = {235--265},
  issn = {0747-7171},
  doi = {10.1006/jsco.1996.0125}
}

@book{bum2013,
  title = {Lie Groups},
  author = {Bump, Daniel},
  date = {2013},
  series = {Graduate Texts in Mathematics},
  edition = {2nd edition},
  number = {225},
  publisher = {{Springer}},
  location = {{New York}},
  isbn = {978-1-4614-8024-2},
  langid = {english}
}

@book{cur.rei1966,
  title = {Representation {{Theory}} of {{Finite Groups}} and {{Associative Algebras}}},
  author = {Curtis, Charles W. and Reiner, Irving},
  date = {1966},
  eprint = {RKwjeZKMr8oC},
  eprinttype = {googlebooks},
  publisher = {{American Mathematical Soc.}},
  isbn = {978-0-8218-6945-1},
  langid = {english},
  pagetotal = {722}
}

@book{dia.shu2005,
  title = {A First Course in Modular Forms},
  author = {Diamond, Fred and Shurman, Jerry Michael},
  date = {2005},
  series = {Graduate Texts in Mathematics},
  number = {228},
  publisher = {{Springer}},
  location = {{New York}},
  isbn = {978-0-387-27226-9},
  langid = {english}
}

@book {Brown,
    AUTHOR = {Brown, Kenneth S.},
     TITLE = {Buildings},
 PUBLISHER = {Springer-Verlag, New York},
      YEAR = {1989},
     PAGES = {viii+215},
      ISBN = {0-387-96876-8},
   MRCLASS = {20-02 (20E32 20G15 22E99 51B25)},
  MRNUMBER = {969123 (90e:20001)},
MRREVIEWER = {W. M. Kantor},
       DOI = {10.1007/978-1-4612-1019-1},
}

@book {neukirch1999,
    AUTHOR = {Neukirch, J{\"u}rgen},
     TITLE = {Algebraic number theory},
    SERIES = {Grundlehren der Mathematischen Wissenschaften},
    VOLUME = {322},
 PUBLISHER = {Springer-Verlag},
   ADDRESS = {Berlin},
      YEAR = {1999},
     PAGES = {xviii+571},
      ISBN = {3-540-65399-6},
   MRCLASS = {11Rxx (11-02 11S15 11S31 14C40)},
  MRNUMBER = {1697859 (2000m:11104)},
MRREVIEWER = {Cornelius Greither},
}

@book{eme1998,
  title = {2-Adic Modular Forms of Minimal Slope},
  author = {Emerton, Matthew James},
  date = {1998},
  publisher = {{ProQuest LLC, Ann Arbor, MI}},
  isbn = {978-0-591-85438-1},
  mrnumber = {2697462},
  pagetotal = {88},
  note = {Thesis (Ph.D.)--Harvard University}
}

@thesis{geb1996,
  type = {Diplomarbeit},
  title = {Operation arithmetischer Untergruppen von $GL(3)$ auf Bru\-hat-Tits-Gebäuden},
  author = {Gebhardt, Max},
  date = {1996},
  institution = {Universität des Saarlandes}
}

@thesis{gr2021,
  type = {phdthesis},
  title = {Boundary Distributions for $\GL_3$ over a Local Field and Symmetric Power Coefficients},
  author = {Gr{\"a}f, Peter Mathias},
  date = {2021-01},
  institution = {Universität Heidelberg},
  url = {http://www.ub.uni-heidelberg.de/archiv/29302}
}

@article {GouveaMazurSlopes,
    AUTHOR = {Gouv{\^e}a, F. and Mazur, B.},
     TITLE = {Families of modular eigenforms},
   JOURNAL = {Math. Comp.},
  FJOURNAL = {Mathematics of Computation},
    VOLUME = {58},
      YEAR = {1992},
    NUMBER = {198},
     PAGES = {793--805},
      ISSN = {0025-5718},
     CODEN = {MCMPAF},
   MRCLASS = {11F33 (11Y35 14G20)},
  MRNUMBER = {1122070 (93d:11049)},
MRREVIEWER = {Bas Edixhoven},
       DOI = {10.2307/2153218},
}

@article{hat2021,
  title = {Dimension {{Variation}} of {{Gouv\^ea}}--{{Mazur Type}} for {{Drinfeld Cuspforms}} of {{Level $\Gamma$1}}(t)},
  author = {Hattori, Shin},
  date = {2021-02-22},
  journaltitle = {International Mathematics Research Notices},
  shortjournal = {International Mathematics Research Notices},
  volume = {2021},
  number = {3},
  pages = {2389--2402},
  issn = {1073-7928},
  doi = {10.1093/imrn/rnz104}
}

@thesis{m2014,
  type = {phdthesis},
  title = {Modulsymbole und Bruhat-Tits-Gebäude der PGL(3) über Funktionenkörpern},
  author = {M{\"u}ller, Jennylee},
  date = {2014},
  institution = {Universität Kassel},
  url = {https://kobra.uni-kassel.de/handle/123456789/2015012147246},
  langid = {ngerman}
}

@article {nic.ros2021,
    AUTHOR = {Nicole, Marc-Hubert and Rosso, Giovanni},
     TITLE = {Familles de formes modulaires de {D}rinfeld pour le groupe
              g\'{e}n\'{e}ral lin\'{e}aire},
   JOURNAL = {Trans. Amer. Math. Soc.},
  FJOURNAL = {Transactions of the American Mathematical Society},
    VOLUME = {374},
      YEAR = {2021},
    NUMBER = {6},
     PAGES = {4227--4266},
      ISSN = {0002-9947},
   MRCLASS = {11F52 (11F33 11G09)},
  MRNUMBER = {4251228},
MRREVIEWER = {Francesc Bars},
       DOI = {10.1090/tran/8314},
}

@article{pin2020, 
    title={Compactification of Drinfeld moduli spaces as moduli spaces of $A$-reciprocal maps and consequences for Drinfeld modular forms}, 
    volume={30}, 
    DOI={10.1090/jag/772}, 
    number={3}, 
    journaltitle={Journal of Algebraic Geometry}, 
    author={Pink, Richard}, 
    year={2020}, 
    month={Dec}, 
    pages={477–527}
}

@book{ser2012,
  title = {Linear Representations of Finite Groups},
  author = {Serre, Jean-Pierre},
  translator = {Scott, Leonard L.},
  date = {2012},
  series = {Graduate Texts in Mathematics},
  edition = {Softcover reprint of the hardcover 1st edition 1977},
  number = {42},
  publisher = {{Springer-Verlag}},
  location = {{New York Heidelberg Berlin}},
  isbn = {978-1-4684-9458-7 978-1-4684-9460-0},
  langid = {english},
  pagetotal = {170}
}

@book{shi1994,
  title = {Introduction to the Arithmetic Theory of Automorphic Functions},
  author = {Shimura, Gor\=o},
  date = {1994},
  series = {Publications of the {{Mathematical Society}} of {{Japan}} ; {{Kan\=o}} Memorial Lectures},
  number = {11. 1},
  publisher = {{Princeton University Press}},
  location = {{Princeton, N.J}},
  isbn = {978-0-691-08092-5},
  pagetotal = {271},
  note = {Originally published: Tokyo : Iwanami Shoten ; Princeton, N.J. : Princeton University Press, 1971}
}

@article{tei1991,
  title = {The {{Poisson}} Kernel for {{Drinfeld}} Modular Curves},
  author = {Teitelbaum, Jeremy T.},
  date = {1991},
  journaltitle = {Journal of the American Mathematical Society},
  shortjournal = {J. Amer. Math. Soc.},
  volume = {4},
  number = {3},
  pages = {491--511},
  issn = {0894-0347, 1088-6834},
  doi = {10.1090/S0894-0347-1991-1099281-6},
  langid = {english}
}

@misc{code,
  author = {Kaiser, Theresa},
  title = {$U$-operators},
  howpublished = "\url{https://github.com/TheresaKaiser/U-Operators}",
  year = {2025}, 
  note = {The \texttt{Magma}-code written for this article and additional results.}
}

\end{document}